\documentclass[11pt, twoside, leqno]{amsart}  
\usepackage{lipsum}
\usepackage{amsfonts}
\usepackage{graphicx}
\usepackage{epstopdf}
\usepackage{algorithmic}
\usepackage{calligra}
\usepackage[dvipsnames]{xcolor}
\usepackage{amsfonts,amsmath,amsthm,amssymb}
\usepackage{mathtools}
\usepackage{hyperref}
\usepackage[makeroom]{cancel}
\usepackage{autonum}
\usepackage{hhline}
\usepackage{array}
\usepackage{diagbox}
\usepackage{mdframed}
\usepackage{multicol}
\usepackage{graphicx}
\usepackage{subcaption}
\usepackage{moreverb}
\usepackage{bbm}
\usepackage[margin=1.38in]{geometry}
\usepackage{todonotes}
\usepackage{scalerel,amssymb}
\allowdisplaybreaks
\usepackage{mathrsfs}  
\usepackage{lineno}
\usepackage{todonotes}
\usepackage{tikz}
\usepackage{appendix}
\usepackage{enumitem}
\usepackage{pgfplots}
\usetikzlibrary{arrows.meta}
\usepackage[numbers,sort&compress]{natbib}
\definecolor{mygreen}{HTML}{43a047}
\usepackage{subcaption}
\usepackage{doi}
\newcommand{\Om}{\Omega}

\newcommand{\rhob}{\rho_{\textup{b}}}
\newcommand{\rhoa}{\rho_{\textup{a}}}
\newcommand{\Ca}{C_{\textup{a}}}
\newcommand{\Cb}{C_{\textup{b}}}
\newcommand{\kappaa}{\kappa_{\textup{a}}}
\newcommand{\Thetaa}{\Theta_{\textup{a}}}

\newcommand{\ddt}{\frac{\textup{d}}{\textup{d}t}}

\newcommand{\dt}{\, \textup{d} t}
\newcommand{\ds}{\, \textup{d} s }
\newcommand{\dx}{\, \textup{d} x}

\newcommand{\intO}{\int_{\Omega}}
\newcommand{\R}{\mathbb{R}} 
 
\newtheorem{theorem}{Theorem}
\newtheorem{lemma}{Lemma}
\newtheorem{proposition}{Proposition}
\newtheorem{assumption}{Assumption}

\newtheorem{remark}{Remark}
\numberwithin{lemma}{section}
\numberwithin{proposition}{section}
\numberwithin{theorem}{section}
\numberwithin{equation}{section}
\makeatletter
\newcommand{\leqnomode}{\tagsleft@true}
\newcommand{\reqnomode}{\tagsleft@false}
\makeatother
\definecolor{grey}{rgb}{0.5,0.5,0.5}

\title[Westervelt--hyperbolic Pennes system]{The  Westervelt--Pennes--Cattaneo model: local well-posedness and singular limit for vanishing relaxation  time}      
   
               
\author[I. Benabbas]{Imen Benabbas$^\dagger$}
\thanks{$^\dagger$AMNEDP Laboratory, Faculty of Mathematics,
	USTHB (\href{ibenabbas@usthb.dz}{ibenabbas@usthb.dz})}
\author[B. Said-Houari]{Belkacem Said-Houari$^\ddag$}
\thanks{$^\ddag$Department of Mathematics, College of Sciences, University of
	Sharjah, P. O. Box: 27272, Sharjah, United Arab Emirates    (\href{bhouari@sharjah.ac.ae}{bhouari@sharjah.ac.ae})}
\begin{document}
	\vspace*{8mm}
	\begin{abstract}
In this work, we investigate a mathematical model of nonlinear ultrasonic heating based on a coupled system of the Westervelt equation and the hyperbolic Pennes bioheat equation (Westervelt--Pennes--Cattaneo model). Using the energy method together with a fixed point argument, we prove that our model is locally well-posed and does not degenerate under a smallness assumption on the pressure data in the Westervelt equation.  In addition, we perform a singular limit analysis and show that the Westervelt--Pennes--Fourier model can 
be seen as an approximation of the Westervelt--Pennes--Cattaneo model as the relaxation parameter tends to zero. This is done by deriving uniform bounds of the solution with respect to the
relaxation parameter.   
					\end{abstract}   
	\vspace*{-7mm}  
	\maketitle                
     
\section{Introduction}       
We are interested in the analysis of a nonlinear thermo-acoustic system modeling the propagation of ultrasonic waves, such as the high--intensity focused ultrasound (HIFU), through thermoviscous fluids. With various applications including medical and industrial use in lithotripsy, thermotherapy, ultrasound cleaning, and sonochemistry \cite{hahn2018high, hsiao2016clinical, Jhang2009NonlinearUT, maloney2015emerging, zhou2011high}, the behavior of HIFU and the mathematical models describing it are receiving a great deal of attention from researchers. For instance, in medical procedures, focused ultrasound is used to generate localized heating that can destroy the targeted region. Indeed, this technique is proving its success in the treatment of both benign and malignant tumors \cite{ hahn2018high, maloney2015emerging}. Due to the high frequency of the sound waves, the nonlinear effect of their propagation cannot be eluded and it is well-established in nonlinear acoustics that Westervelt equation \cite{westervelt1963parametric} takes this effect into consideration.

In this paper we consider a coupled system  of the Westervelt equation for the pressure and a hyperbolic Pennes equation for the temperature. More precisely, we consider the system   
\begin{equation} 
\left\{ \label{coupled_problem_eq_Cattaneo}
\begin{aligned}
&p_{tt}-c^2(\bar{\Theta})\Delta p - b \Delta p_t = K(\bar{\Theta})\left(p^2\right)_{tt}, \qquad &\text{in} \ \Omega \times (0,T),\\
&\rhoa \Ca\bar{\Theta}_t +\nabla\cdot q + \rhob \Cb W(\bar{\Theta}-\Thetaa) = \mathcal{Q}(p_t), \qquad &\text{in} \ \Omega \times (0,T),
\\
&\tau q_t+q+\kappaa \nabla \bar{\Theta}=0, \qquad &\text{in} \ \Omega \times (0,T).
\end{aligned}
\right.
\end{equation}
Here the acoustic pressure and the temperature fluctuations are denoted respectively by $p$ and $\bar{\Theta}$; $c$ is the speed of sound and $b>0$ is the sound diffusivity. The function $K(\bar{\Theta})$ is allowed to depend on $\bar{\Theta}$ and is given by
$$K(\bar{\Theta})=\frac{\beta_{\text{acous}}}{\rho c^2(\bar{\Theta })},$$
where $\rho$ is the mass density and $\beta_{\text{acous}}$ is the parameter of nonlinearity.   
 The source term in the second equation $\mathcal{Q}(p_t)$ represents the acoustic energy absorbed by the tissue. The medium parameters $\rhoa, \Ca$ and $\kappaa$ stand, respectively, for the ambient density, the ambient heat capacity and the thermal conductivity of the tissue. The additional term $\rhob \Cb W(\bar{\Theta}-\Thetaa)$ accounts for the heat loss due to blood circulation, with  $\rhob, \Cb$ being the density and the specific heat capacity of blood, and $W$ expressing the tissue's volumetric perfusion rate measured in milliliters of blood per milliliter of tissue per second.

The second and the third equations in \eqref{coupled_problem_eq_Cattaneo} constitute the hyperbolic version of the Pennes equation 
\eqref{Hyperbolic_Pennes}, where the heat conduction is described by the Cattaneo law (or the Cattaneo--Maxwell law).  
We supplement \eqref{coupled_problem_eq_Cattaneo} with the initial conditions 
\begin{equation}\label{Initial_Condi}
p|_{t=0}=p_0,\quad p_t|_{t=0}=p_1,\quad \bar{\Theta}|_{t=0}=\bar{\Theta}_0,\quad q|_{t=0}=q_0
\end{equation}
and Dirichlet boundary conditions  
\begin{equation} \label{coupled_problem_BC}
p\vert_{\partial \Om}=0, \qquad \bar{\Theta}\vert_{\partial \Om}= \Thetaa, 
\end{equation}
where $\Thetaa$ denotes the ambient temperature, that is typically taken in the human body to be $37^\circ C$; see \cite{connor2002bio}.

When $c$ and $K$ are constants, 
the first equation in \eqref{coupled_problem_eq_Cattaneo} reduces to the  the Westervelt equation for the pressure $p(x,t)$:
\begin{equation}\label{Westervaelt_Eq}
p_{tt}-c^2\Delta p - b \Delta p_t = K\left(p^2\right)_{tt}.
\end{equation}
Equation \eqref{Westervaelt_Eq}  is  widely used  in acoustics   and  describes the propagation of sound waves in a fluid medium. It can be derived from the Navier--Stokes--Fourier model by assuming  the Fourier law of heat conduction and by  taking   into account the thermoviscous effects. See \cite{westervelt1963parametric},  \cite{crighton1979model} and \cite[Chapter 5]{kaltenbacher2007numerical} for more details. Significant progress has been made recently  toward the understanding of the solutions to the Westervelt equation and their behavior; see \cite{kaltenbacher2009global, Clason2013AvoidingDI, Kaltenbacher2015MathematicsON, meyer2011optimal, kaltenbacher2019Well, Simonett2016WellposednessAL} and the references therein. The results in \cite{kaltenbacher2009global, meyer2011optimal, Simonett2016WellposednessAL} investigated local well-posedness, global well-posedness, and asymptotic  behavior of  solutions, subject to different types of boundary conditions and in various functional settings. Particularly in \cite{meyer2011optimal, Simonett2016WellposednessAL}, the authors relied on maximal regularity in $L^p$-spaces to obtain the existence of a unique solution with low regularity assumptions on the initial data.
% In \cite{kaltenbacher2009global}, the authors considered Westervelt equation with Dirichlet conditions and proved global well-posedness under the assumption of small initial data, and asymptotic decay at an exponential rate as $t \rightarrow  \infty$.
% Later, Meyer and  Wilke \cite{meyer2011optimal} relied on maximal regularity in $L^p$-spaces to obtained the existence of a unique solution with optimal regularity, which decays exponentially to zero. 
 This is feasible thanks to the strong damping represented by the term $-b \Delta p_t$ when $b>0$, which lends to the Westervelt equation its parabolic character. 

Combining the   second and the third equation equations in \eqref{coupled_problem_eq_Cattaneo}, we obtain the hyperbolic Pennes equation (see  \cite[Eq.\ 3]{kabiri2021analysis} and \cite[Eq. 7]{xu2008non})
\begin{equation} \label{Hyperbolic_Pennes}
\begin{aligned}
&\tau \rhoa \Ca \bar{\Theta}_{tt}+(\rhoa \Ca +\tau \rhob \Cb W) \bar{\Theta}_t +\rhob \Cb W (\bar{\Theta}-\Thetaa)-\kappaa \Delta \bar{\Theta} \\
=&\,\mathcal{Q}(p_t)+ \tau \partial_t\mathcal{Q}(p_t).
\end{aligned}
\end{equation}
The terms in \eqref{Hyperbolic_Pennes} that appear with the $\tau$ prefactor  arise from the use of the Cattaneo law of heat conduction \cite{Ca48}:  
 \begin{equation}\label{Cattaneo}
\tau q_t+q+\kappaa \nabla \bar{\Theta}=0
\end{equation}
  which is a modified version of the classical Fourier law: 
 \begin{equation}\label{Fourier_Law}
q+\kappaa \nabla \bar{\Theta}=0. 
\end{equation} 
Equation \eqref{Fourier_Law} implies  instantaneous thermal energy deposition in the medium. That is, any temperature disturbance causes instantaneous perturbation in the temperature at each point in the medium. 
The Cattaneo law was introduced to overcome this  drawback of the infinite speed of thermal signals in the Fourier law. The idea is to introduce a time lag into the relationship between the heat flux and the temperature gradient, which results in the term $\tau q_t$ in  \eqref{Cattaneo}, where $\tau $ is the relaxation time parameter.    

%{\color{blue}Under the constitutive law \eqref{Cattaneo}, the heat flow within  the thermoviscous medium occurs 
%through the propagation of thermal waves at finite speed}.
The Cattaneo model \eqref{Cattaneo} has the effect of changing the heat equation into a damped wave equation, which describes a physical phenomenon of heat waves.   
 This phenomenon is known as the  \textsl{second sound} effect, and it is experimentally observed in materials at a very low temperature, where the heat seems to propagate as a thermal wave, which is the reason for this name (see the review paper  \cite{Racke_2009}).    
   
For $\tau=0$, equation \eqref{Hyperbolic_Pennes} becomes  
\begin{equation} \label{Heat_Eq}
	\rhoa \Ca\bar{\Theta}_t -\kappaa\Delta \bar{\Theta}+ \rhob \Cb W(\bar{\Theta}-\Thetaa) = \mathcal{Q}(p_t).
\end{equation}
Equation \eqref{Heat_Eq} is the parabolic Pennes equation, which is a bioheat transfer equation that is widely used for studying heat transfer in biological  systems.  It  takes into account the heat transfer by conduction in the tissues and the convective heat transfer due to blood perfusion.  See \cite{pennes1948analysis} for the derivation of \eqref{Heat_Eq}. 

In \cite{Nikolic_2022} Nikoli\'c and Said-Houari  considered the Westervelt--Pennes--Fourier  system, which is the coupling between the first equation in \eqref{coupled_problem_eq_Cattaneo}
 and \eqref{Heat_Eq}. That is, they investigated the system 
  \begin{equation}\label{Westervelt--Pennes--Fourier}
  \left\{
  \begin{aligned}
	&p_{tt}-c^2(\bar{\Theta})\Delta p - b \Delta p_t = k(\bar{\Theta}) \left(p^2\right)_{tt},\\
	&\rhoa \Ca\bar{\Theta}_t -\kappaa\Delta \bar{\Theta}+ \rhob \Cb W(\bar{\Theta}-\Thetaa) = \mathcal{Q}(p_t)
\end{aligned}
\right.	
\end{equation}
in $\Omega \times (0,T)$ with Dirichlet--Dirichlet boundary conditions. They proved a local well-posedness of \eqref{Westervelt--Pennes--Fourier} using the energy method together with a fixed point argument.  The work in \cite{Nikolic_2022} was followed by \cite{NIKOLIC2022628}, where the authors proved the global existence of the solution of \eqref{Westervelt--Pennes--Fourier} under a smallness assumption on the initial data, and they studied its asymptotic behavior as $t$ goes to infinity. Using the maximal regularity estimate for parabolic systems, Wilke in \cite{Wilke_2023} improves slightly the regularity assumptions in \cite{Nikolic_2022} and also considered the case $b=b(\bar{\Theta})$.    
   
   In this paper, we consider the Westervelt--Pennes--Cattaneo   system \eqref{coupled_problem_eq_Cattaneo} and investigate the  local well-posedness and the singular limit as the time relaxation parameter $\tau$ tends to zero. To state and prove our result and to lighten the notation, we put 
\begin{equation}
m=\rhoa \Ca\qquad \text{and}\qquad \ell=\rhob \Cb W,
\end{equation} 
and  make the change of variables 
$\Theta=\bar{\Theta}-\Theta_a$
in the temperature. Denoting by 
 \begin{equation}\label{funct_k}
k (\Theta)=K(\Theta+\Thetaa)=\frac{1}{\rho c^2(\Theta+\Thetaa)} \beta_{\text{acou}}\quad \text{and}\quad  h (\Theta)=c^2(\Theta+\Thetaa),
\end{equation}
we get the following system 
\begin{subequations}\label{Main_system} 
\begin{equation}
\left\{ \label{modified_temp_eq}
\begin{aligned}
&p_{tt}-h (\Theta)\Delta p - b \Delta p_t = k (\Theta)\left(p^2\right)_{tt}, \qquad &\text{in} \ \Omega \times (0,T)\\
&m\Theta_t +\nabla\cdot q + \ell \Theta = \mathcal{Q}(p_t), \qquad &\text{in} \ \Omega \times (0,T)
\\   
&\tau q_t+q+\kappaa \nabla \Theta=0, \qquad &\text{in} \ \Omega \times (0,T)
\end{aligned}
\right.
\end{equation} 
complemented with homogeneous boundary conditions
\begin{equation} \label{homog_dirichlet}
p\vert_{\partial \Om}=0, \qquad \Theta\vert_{\partial \Om}=0
\end{equation} 
and the initial conditions
\begin{equation} \label{init_cond}
p|_{t=0}=p_0,\quad p_t|_{t=0}=p_1,\quad \Theta|_{t=0}=\Theta_0:=\bar{\Theta}_0-\Thetaa,\quad q|_{t=0}=q_0.
\end{equation}
\end{subequations}
As in \cite{Nikolic_2022}, here the medium parameters $c$ and $K$ in the Westervelt equation are not constant. They are taken explicitly dependent on the temperature in order to account for the fact that the heating generated by the ultrasound waves affects their speed of propagation and the position of the focal region, a phenomenon known as thermal lensing \cite{connor2002bio, hallaj2001simulations}. Precisely, we assume this dependence to be polynomial, in agreement with the experimentally observed behavior documented in \cite{bilaniuk1993speed}. Also for simplicity, we assume that  the function $\mathcal{Q}$ has the form  
\begin{equation}
\mathcal{Q}(p_t)=\frac{2b}{\rhoa\Ca^4} (p_t)^2, 
\end{equation}
although our proof works for quite general $\mathcal{Q}(p_t)$ satisfying  Assumption 2 in \cite{Nikolic_2022}. 

To establish well-posedness, we carry out an energy analysis for the linearization of the underlying system, which will allow applying a fixed-point argument to work out the existence of a unique solution to the nonlinear problem \eqref{Main_system}. In doing so, we encounter two main challenges. On the one hand, we have to take into account the interplay between the pressure and the temperature owing to the coupling of their respective equations. In the pressure equation, the  dependence of the coefficients $h(\Theta)$ and $k(\Theta)$ on the temperature suggests that they should be kept bounded for all time. That is to say, the function $\Theta$ needs to be in $L^\infty((0, T) \times \Om)$. Moreover, note that the first equation in \eqref{Main_system} is quasilinear. To see this, it suffices to write the right-hand side as $k(\Theta)(p^2)_{tt}=2k(\Theta)p p_{tt}+2k(\Theta)(p_t)^2$. This brings about the risk of degeneracy of the term $(1-2k(\Theta)p) p_{tt}$, which is usually avoided by imposing a smallness constraint on the acoustic pressure \cite{kaltenbacher2009global}. Therefore, we are led to work with higher-order energies for both the pressure and the temperature. One key ingredient in obtaining the \textit{a priori} estimates for these energies is the  Sobolev embedding theorem, especially the continuous embedding  $H^2(\Om) \hookrightarrow L^\infty(\Om)$. On the other hand, obtaining uniform energy bounds with respect to $\tau$ for the heat conduction system requires careful attention. 
  Further, considering that in practical situations, $\tau$ takes very small values, it is interesting to investigate the behavior of the system when $\tau$ goes to zero. In order to do this, we adopt the approach in \cite{Kaltenbacher2019TheJE, nikolic2023nonlinear}.   See also the important results in \cite{Bongarti:2021aa}. Even though the models under consideration in these works involve single equations, the method is quite constructive and could be applied to  the coupled problem at hand. The key ideas are to first derive estimates that are uniform with respect to $\tau$, which will justify taking the limit as $\tau$ tends to zero; then we make use of the compactness of Sobolev embeddings to show that the solution of \eqref{Main_system} converges to the solution of the system corresponding to $\tau=0$.

Our paper is organized as follows. The main results are stated in Section \ref{sect3}. In Section \ref{sect2}, we collect some theoretical results that will prove useful in the sequel, and we state the general assumptions on the coefficients in system \eqref{Main_system}. Sections \ref{sect4}  is devoted to the energy analysis of the hyperbolic Pennes equation (the Cattaneo system), while Section  \ref{sect5}  treats the linearized Westervelt equation. In Section \ref{sect6}, we prove the local well-posedness of the nonlinear problem \eqref{Main_system}.  Finally, in Section \ref{sect7}, we perform the singular limit analysis and show that the solution of the  Westervelt--Pennes--Cattaneo model converges to the solution of the Westervelt--Pennes--Fourier system when the time relaxation vanishes.  
 
 \section{Main results} \label{sect3} 
 In this section, we state the main results of this paper and give the strategy of the proof. 
In order to give context to the main results, we begin by specifying the functional setup adopted in this work. We define
 \begin{equation}
\mathcal{X}:= X_p \times X_\Theta \times X_q,
\end{equation}
where the spaces $X_p, X_\Theta$ and $X_q$ are given, respectively  as 
\begin{equation}\label{Functional_Spaces}
\begin{aligned}
X_p=&\,\Big\{p \in L^{\infty}(0, T; H^3( \Om) \cap H^1_0(\Om)), \\
&  \quad p_t \in L^{\infty}(0, T; H^2( \Om) \cap H^1_0(\Om))\cap L^{2}(0, T; H^3( \Om) \cap H^1_0(\Om)),\\
& \quad p_{tt} \in L^{\infty}(0, T; H^1_0(\Om))\cap L^{2}(0, T; H^2( \Om) \cap H^1_0(\Om)),\\
& \quad p_{ttt} \in L^{2}(0, T; L^2(\Om)) \Big\};\\
X_\Theta=&\, \{ \Theta \in L^\infty (0, T;  H^2(\Om) \cap H^1_0(\Om)), \Theta_t \in L^\infty (0, T; H^1_0(\Om)),\\
& \quad \Theta_{tt} \in L^\infty(0, T; L^2(\Om))\};\\
X_q=&\, \{ q \in L^\infty(0, T; (H^1(\Om))^d); q_t, q_{tt} \in L^2(0, T; (L^2(\Om))^d)\}.
\end{aligned}
\end{equation}
Our first result ensures local in time well-posedness, which is uniform with respect to the relaxation parameter $\tau$.
  \begin{assumption} \label{Assumption_compatibility} 
 Let the initial data satisfy
 \begin{equation}   
\begin{aligned}
(p_0, p_1) & \in \big[H^3(\Om)\cap H^1_0(\Om)\big]\times \big[H^3(\Om) \cap H^1_0(\Om)\big],\\  
(\Theta_0, q_0) & \in \big[H^2(\Om) \cap H^1_0(\Om)\big] \times (H^2(\Om))^d,
\end{aligned}
\end{equation}
such that $1-2k(\Theta_0) p_0$ does not degenerate. We also assume  the compatibility conditions:
\begin{equation}
p_2, \Theta_1 \in H^1_0(\Om), \quad \Theta_2, q_1, q_2 \in L^2(\Om) 
\end{equation}
where $p_2:= \partial_t^2 p(0,x), \Theta_k:=\partial_t^k \Theta(0,x), q_k:=\partial_t^kq(0,x), x \in \Omega, k=1, 2$ are defined formally and recursively in terms of $p_0, p_1, \Theta_0, q_0$ from the equations \eqref{modified_temp_eq}.
 %See \cite{lasiecka2019long}. 
 
%$p_{tt}(0)=\frac{1}{1-2k(\Theta_0) p_0}(h(\Theta_0) \Delta p_0+b \Delta p_1 +k(\Theta_0)p_1^2)$
 \end{assumption}
 \begin{theorem} \label{wellposedness_thm}
Let $T>0$ and $\bar{\tau}>0$ be a fixed small constant.  Let  $\tau \in (0, \bar{\tau}]$ and let Assumption \ref{Assumption_compatibility} hold.     
%\begin{equation}   
%\begin{aligned}
%(p_0, p_1) & \in {\color{blue}\big[}H^3(\Om)\cap H^1_0(\Om) {\color{blue}\big]}\times {\color{blue}\big[}H^2(\Om) \cap H^1_0(\Om){\color{blue}\big]},\\  
%(\Theta_0, q_0) & \in {\color{blue}\big[}H^2(\Om) \cap H^1_0(\Om){\color{blue}\big]} \times (H^1(\Om))^d.
%\end{aligned}
%\end{equation}
There exists $\delta=\delta(T)>0$ such that if  
\begin{equation}
\|p_0\|_{H^3}+\|p_1\|_{H^2}+\|p_{2}\|_{H^1}\leq \delta
\end{equation}
then system \eqref{Main_system} has a unique solution $(p, \Theta, q) \in \mathcal{X}$. 
\end{theorem}

Let us now outline the main steps in the proof of Theorem \ref{wellposedness_thm} and list some comments about the above result. 
\begin{enumerate}
\item[1.] We consider a linearization of the underlying model \eqref{Main_system} that will see the Westervelt equation decoupled from the Cattaneo system for heat transfer. This allows us to treat the first equation in \eqref{Main_system} separately from the second and third equations. However, the decoupling of the linearized system does not allow us to transfer the damping induced by the damped wave equation for the temperature to the linearized  Westervelt equation.

 We use Galerkin approximations to prove the existence of a unique solution for the linearized Cattaneo system \eqref{eq_Cattaneo}, together with uniform energy estimates with respect to $\tau$. 
%Since, we have the intention of looking at what becomes of the system when the time lag $\tau$ vanishes, it is only natural to take its value in an interval $(0, \bar{\tau}]$.    
 Next, motivated by \cite{Nikolic_2022}, we often rely on the Sobolev embedding theorem to conduct an energy analysis yielding the well-posedness of the linearized Westervelt equation in a finite time horizon $T>0$. In addition, we derive some energy bounds that will be useful in the analysis of the nonlinear problem.  

\item[2.] Having all the necessary ingredients, we can tackle the nonlinear coupled problem \eqref{Main_system} by defining the solution of the system \eqref{Main_system} as the fixed point of a carefully defined mapping. Thus, bringing together the already established energy estimates and using the
 Banach fixed-point theorem,   we prove well-posedness for the nonlinear problem and show that the solution does not degenerate under a smallness assumption on the initial data of the acoustic pressure. 
 
 \item[3.]  Note that the smallness condition in Theorem \ref{wellposedness_thm} is imposed only on the pressure data and not on the temperature data. This seems necessary to avoid the degeneracy of the Westervelt equation.
\item [4.] Using the same method, we can also treat the case $b=b(\Theta)$. Our assumption on $b$  to be constant is just to avoid technicalities in the proof. Also, the condition $b>0$ is crucial in our analysis. It is an important open problem to study the case $b=0$.  
 
\end{enumerate}

In the following theorem, we state a result on the convergence of the solution of the Westervelt--Pennes--Cattaneo model to the solution of the Westervelt--Pennes--Fourier  model  as the relaxation time $\tau$ tends to zero, and as a consequence, we recover at the limit, the well-posedness  result in \cite{Nikolic_2022}. To facilitate relating the limit to the solution of the Westervelt--Pennes--Fourier  model, in this part we use the wave equation \eqref{Hyperbolic_Pennes} instead of Cattaneo's system. We denote by $\Theta_1^\tau:= \Theta^\tau_t(0,x)$.
\begin{theorem}\label{limit_thm}
Given $T>0$ and $ \tau \in (0, \bar{\tau}]$. Let the initial data $(p_0^\tau, p_1^\tau), (\Theta_0^\tau, \Theta_1^\tau)$  
%\begin{equation}
%\begin{aligned} 
%&(p_0^\tau, p_1^\tau) \in {\color{blue}\big[}H^3(\Om) \cap H^1_0(\Om){\color{blue}\big]} \times {\color{blue}\big[}H^2(\Om) \cap H^1_0(\Om){\color{blue}\big]},\\
%&(\Theta_0^\tau, \Theta_1^\tau) \in {\color{blue}\big[}H^2(\Om) \cap H^1_0(\Om){\color{blue}\big]} \times H^1_0(\Om),
%\end{aligned}
%\end{equation}
%\begin{equation}
%\begin{aligned}
%&\mathfrak{E}[p](0) \leq \delta, \\
%& \mathcal{E}[\Theta](0) \leq C\Big(\Vert q_0 \Vert^2_{H^1}+(1+ \bar{\tau} +\bar{\tau}^2)\big(E^{\bar{\tau}}[\Theta, q](0)+R_1^2 \big) \Big) \leq C R_2^2,
%\end{aligned}
%\end{equation} 
satisfy Assumption \ref{Assumption_compatibility} and the assumptions of Theorem \ref{wellposedness_thm}.   Then, the familly of solutions $(p^\tau, \Theta^\tau)_{\tau \in (0, \bar{\tau}]}$ converges weakly (see \eqref{weak_star_conv}, \eqref{initial-weak-conv}) to the solution $(p, \Theta) \in X_p \times X_\Theta$ of the  Westervelt--Pennes--Fourier  system: 
\begin{equation} \label{limiting_system}
\left\{   
\begin{aligned}
&(1-2k(\Theta)p) p_{tt}-h(\Theta) \Delta p  -b \Delta p_t =2k(\Theta ) (p_t )^2,  \qquad &\text{in} &\ \Om& \times (0, T), \\
&  m \Theta_t +\ell \Theta -\kappaa \Delta \Theta  =\mathcal{Q}(p_t), \qquad &\text{in} &\ \Om& \times (0, T),\\
&(p , p_t )|_{t=0}=(p_0, p_1), \quad \Theta |_{t=0}=\Theta_0, 
\end{aligned}
\right.
\end{equation}
with homogeneous Dirichlet conditions \eqref{homog_dirichlet}.
\end{theorem}

It is essential to note that the main difficulty in proving   Theorem \ref{limit_thm}  lies in obtaining energy estimates that are uniform with respect to $\tau$.   In other words, our goal is to prevent the constants from becoming  infinitely large as $\tau$ approaches zero in the estimates. This justifies the process of passing to the limit when the parameter $\tau$ tends to zero.
  The uniformity in $\tau$ is not a requirement  for the well-posedness result in Theorem \ref{wellposedness_thm}, but it is essential in the proof of Theorem \ref{limit_thm}.  Lastly, we remark that as the estimates are dependent on the time $T>0$, the existence of solutions is local in nature. However, we can expect to reach global well-posedness by taking the approach in \cite{NIKOLIC2022628}. 

\section{Preliminaries and assumptions} \label{sect2}
In this section, we introduce a few notations and collect some  helpful embedding results and inequalities that we will repeatedly use in the proofs.

\subsection{Notation}
Throughout the paper, we assume that $\Omega \subset \R^d$, where $d \in \{1,2,3\}$, is a bounded, $C^2$ smooth domain. We denote by $T>0$ the final propagation time. The letter $C$ denotes a generic positive constant
that does not depend on time, and can have different values on different occasions.  
We often write $f \lesssim g$ where there exists a constant $C>0$, independent of parameters of interest such that $f\leq C g$. 
%We note that dependance of a constant on the medium parameters is not always made explicit; however, we make sure to emphasize dependance on the relaxation time $\tau$. 
We often omit the spatial and temporal domain when writing norms; for example, $\|\cdot\|_{L^p L^q}$ denotes the norm in $L^p(0,T; L^q(\Omega))$.

\subsection{Inequalities and embedding results} In the upcoming analysis, we shall employ the continuous embeddings $H^1(\Om) \hookrightarrow L^4(\Om)$, $H^1(\Om) \hookrightarrow L^6(\Om)$ and $H^2(\Om) \hookrightarrow L^\infty(\Om)$. In particular, using Poincar\'{e}'s inequality we obtain for $v \in H^1_0(\Om)$ (see \cite[Theorem 7.18]{salsa2016partial})
\begin{equation}\label{Sobolev_Embedding}
\begin{aligned}
& \text{if} \ d>2, \quad \Vert v \Vert_{L^p} \leq C \Vert \nabla v \Vert_{L^2}  \quad \text{for} \quad  2 \leq p \leq \frac{2d}{d-2},\\
& \text{if} \ d=2, \quad \Vert v \Vert_{L^p} \leq C \Vert \nabla v \Vert_{L^2}  \quad \text{for} \quad  2 \leq p < \infty.\\
\end{aligned}
\end{equation}

Moreover, taking into account the boundedness of the operator $(-\Delta)^{-1}: L^2(\Om) \rightarrow H^2(\Om) \cap H^1_0(\Om)$, we find the inequality
$$\qquad \Vert v \Vert_{L^\infty} \leq C_1 \Vert v \Vert_{H^2} \leq C_2 \Vert \Delta v \Vert_{L^2}.$$ 
We will also call on the 1D-embedding $H^1(0, T; L^2(\Om)) \hookrightarrow C(0, T; L^2(\Om))$. That is, we have the following inequality
\begin{equation} \label{1D_embedding_1}
\begin{aligned}
\max_{t\in [0, T]} \Vert v(t) \Vert_{L^2} &\leq C(\Vert v \Vert_{L^2 L^2}+\Vert v_t \Vert_{L^2 L^2})\\
\end{aligned}    
\end{equation}
where the constant $C>0$ depends only on $T$ (see, e.g. \cite[Theorem 2, p. 286]{evans2010partial}). Additionally, if $v \in H^1(0, T, H^1_0(\Om))$, then the above embedding combined with Poincar\'{e}'s inequality yields
\begin{equation} \label{1D_embedding} 
\begin{aligned}
\max_{t\in [0, T]} \Vert v(t) \Vert_{L^2} 
& \leq C(\Vert \nabla v \Vert_{L^2 L^2}+\Vert v_t \Vert_{L^2 L^2}).
\end{aligned}    
\end{equation}
We recall Young's $\varepsilon$-inequality 
\begin{equation}
xy \leq \varepsilon x^n+C(\varepsilon) y^m, \quad \text{where}\quad \ x, y >0, \quad 1 <m,n <\infty,\quad \frac{1}{m}+\frac{1}{n}=1,
\end{equation}
and $C(\varepsilon)=(\varepsilon n)^{-m/n}m^{-1}$. \\
Further, we will make use of  Ladyzhenskaya's inequality for $u \in H^1(\Om)$
\begin{equation} \label{lady}
\Vert u \Vert_{L^4} \leq C\Vert u \Vert_{L^2}^{1-d/4} \Vert u \Vert_{H^1}^{d/4},\qquad 1\leq d\leq 4.  
\end{equation}   
We state a version of Gronwall's inequality which will be utilized in the proofs and has been provided in  \cite{Nikolic_2022}. 
\begin{lemma} \label{gronwall} Let $I=[0, t]$ and let $\alpha, \beta : I \rightarrow \mathbb{R}$ be locally integrable functions. Given  $ u, v : I \rightarrow \mathbb{R}$ such that $v$ is non-negative and integrable and $u$ is in $C^1(I)$. We assume that
\begin{equation}
u'(t)+v(t) \leq \alpha(t) u(t)+\beta(t), \ \text{for} \ t \in I, \quad u(0)=u_0.
\end{equation}
Then, it holds that
\begin{equation}
u(t)+\int_0^t v(s) \ds \leq u_0 e^{A(t)}+\int_0^t \beta(s)e^{A(t)-A(s)} \ds,
\end{equation}
where
\begin{equation}
A(t)=\int_0^t \alpha(s) \ds.
\end{equation}
\end{lemma}
\subsection{Assumptions}
In accordance with the perceived polynomial growth of the speed of sound $c=c(\Theta)$, we make the following assumptions on the functions $h $ and $k $.\\
We assume  that $h \in C^2(\mathbb{R})$ and  there exists $h_1>0$ such that
\begin{subequations}
\begin{equation} \label{bound_h}
h (s) \geq h_1, \quad \forall s \in \mathbb{R}.\tag{H1}
\end{equation}
Moreover, we assume that there exist $\gamma_1 >0$ and $C>0$, such that
\begin{equation} \label{h''_assump}
\vert h ''(s) \vert \leq C (1+\vert s \vert^{\gamma_1}), \quad \forall s \in \mathbb{R}.\tag{H2}
\end{equation} 
So that by using Taylor's formula, we also have
\begin{equation} \label{h'_assump}
\vert h '(s) \vert \leq C (1+\vert s \vert^{1+\gamma_1}), \quad \forall s \in \mathbb{R}.\tag{H3}
\end{equation}
\end{subequations}
Since the function $k $ is related to the speed of sound by the formula \eqref{funct_k}, it follows that
\begin{subequations}
\begin{equation}\label{k_1}
\vert k (s) \vert \leq k_1:=\frac{\beta_{\text{acous}}}{\rho h_1}, \quad \forall s \in \mathbb{R}.\tag{K1}
\end{equation}
Further, we have 
\begin{equation}
\begin{aligned}
%\vert k '(s) \vert &\lesssim k_1^2\vert h '(s) \vert \lesssim k_1^2 (1+\vert s \vert^{1+\gamma_1}),\\
\vert k''(s) \vert &\lesssim k_1^2 \vert h''(s) \vert+k_1^3 \vert h'(s) \vert^2 \lesssim  k_1^2(1+\vert s \vert^{\gamma_1})+k_1^3(1+\vert s \vert^{1+\gamma_1})^2,
\end{aligned}
\end{equation}
which by using Taylor's formula, implies that there exists $\gamma_2>0$, such that
\begin{equation} \label{properties_k}
\vert k'(s) \vert \lesssim (1+\vert s \vert^{1+\gamma_2}), \qquad \vert k''(s) \vert \lesssim (1+\vert s \vert^{\gamma_2}), \quad \forall s \in \mathbb{R}..\tag{K2}
\end{equation}
\end{subequations}

\section{The hyperbolic bioheat equation} \label{sect4}
In this section, we consider the hyperbolic heat equation 
\begin{equation}\label{eq_Cattaneo}
\left\{
\begin{aligned}
&m\Theta_t +\nabla\cdot q + \ell \Theta = f, \qquad &\text{in} \ \Omega \times (0,T),
\\   
&\tau q_t+q+\kappaa \nabla \Theta=0, \qquad &\text{in} \ \Omega \times (0,T),
\end{aligned}
\right.
\end{equation}
together with the relevant initial conditions in \eqref{init_cond} and the boundary condition in \eqref{homog_dirichlet}. 
  Our main goal is to prove a priori estimates under minimal assumptions on   the initial data $\Theta_0$ and $q_0$ and on the the source term $f$.
    
 In order to state and prove our main result,  we define the total energy associated to \eqref{eq_Cattaneo} as
\begin{equation}\label{E_Total}
E^{\tau}[\Theta,q](t):=\sum_{k=0}^2 E_k[\Theta,q](t)
%E^{\tau}[\Theta,q](t):=E_0[\Theta,q](t)+E_1[\Theta,q](t)+E_2[\Theta,q](t)
, \quad t \geq 0
\end{equation}
where the energies $E_k, k=0, 1, 2$ are given by    
%\begin{equation}
%\begin{aligned}
%E_1[\Theta,q](t)&=\frac{1}{2}(m \kappaa \Vert \Theta(t) \Vert^2_{L^2}+\tau \Vert q(t) \Vert^2_{L^2}),\\
%E_2[\Theta,q](t)&=\frac{1}{2}(m \kappaa \Vert \Theta_t(t) \Vert^2_{L^2}+\tau \Vert q_t(t) \Vert^2_{L^2}),\\
%E_3[\Theta,q](t)&=\frac{1}{2}(m \kappaa \Vert \Theta_{tt}(t) \Vert^2_{L^2}+\tau \Vert q_{tt}(t) \Vert^2_{L^2}).\\
%\end{aligned}
%\end{equation}
\begin{equation} \label{E_k_def}
E_k[\Theta,q](t):=\frac{1}{2}\Big(m \kappaa \Vert \partial_t^k\Theta(t) \Vert^2_{L^2}+\tau \Vert \partial_t^k q(t) \Vert^2_{L^2}\Big).
\end{equation}
We also define the associated dissipation 
\begin{equation} \label{D_total}
D[\Theta,q](t):=\sum_{k=0}^2 D_k[p,\Theta](t), \quad t \geq 0
\end{equation}
with 
\begin{equation} \label{D_k_def}
D_k[\Theta,q](t):=\ell \kappaa \Vert \partial_t^k \Theta(t) \Vert^2_{L^2}+\Vert \partial_t^k q(t) \Vert^2_{L^2} ,\quad k=0,1,2. 
\end{equation} 
Since the coefficients in the first equation in \eqref{modified_temp_eq} depend on $\Theta$, we should establish some higher--order estimates on $\Theta$, which will allow us to control the $L^\infty$-norm of $\Theta$ through the Sobolev embedding theorem. To this end, we  introduce the following energy in terms of  $\Theta$ only  
\begin{equation}\label{E_Theta}
\mathcal{E}[\Theta](t):=\mathcal{E}_0[\Theta](t)+\mathcal{E}_1[\Theta](t),  \quad t \geq 0
\end{equation}
where $\mathcal{E}_0$ and $\mathcal{E}_1$ are  defined as follows
\begin{equation} \label{E_0}
\left\{
\begin{aligned}
&\mathcal{E}_0[\Theta](t):=\frac{m \kappaa}{2}(\Vert \Theta(t) \Vert^2_{L^2}+ \Vert \Theta_t(t) \Vert^2_{L^2}+ \Vert \Theta_{tt}(t) \Vert^2_{L^2}),\\
&\mathcal{E}_1[\Theta](t):=\frac{m+\tau \ell}{2} \Vert \nabla \Theta(t) \Vert^2_{L^2}+\kappaa \Vert \nabla \Theta_t(t) \Vert^2_{L^2}+\kappaa \Vert \Delta \Theta(t) \Vert^2_{L^2}.
\end{aligned}
\right.
\end{equation}
The associated dissipation rate to $\mathcal{E}[\Theta](t)$ is 
\begin{equation}\label{dissip_Theta}
\mathcal{D}[\Theta](t):=\mathcal{D}_0[\Theta](t)+\mathcal{D}_1[\Theta](t), 
\end{equation}
with 
\begin{equation} \label{D_0}
\left\{
\begin{aligned}
&\mathcal{D}_0[\Theta](t):= \ell \kappaa (\Vert \Theta(t) \Vert^2_{L^2}+ \Vert \Theta_t(t) \Vert^2_{L^2}+ \Vert \Theta_{tt}(t) \Vert^2_{L^2}),\\
&\mathcal{D}_1[\Theta](t):=\ell \Vert \nabla \Theta(t) \Vert^2_{L^2}+\kappaa \Vert \nabla \Theta_t(t) \Vert^2_{L^2}+\kappaa \Vert \Delta \Theta(t) \Vert^2_{L^2}.
\end{aligned}
\right.
\end{equation}
The definitions of $\mathcal{E}[\Theta](t)$ and $\mathcal{D}[\Theta](t)$ are inspired from the damped wave equation in $\Theta$ \eqref{Hyperbolic_Pennes}  that can be obtained by combining the two equations in system \eqref{eq_Cattaneo}. See also equation \eqref{telegraph_eq} below. 
 
The following lemma will allow us to estimate $\mathcal{E}[\Theta]$ in terms of $E^{\tau}[\Theta,q]$ uniformly with respect to $\tau$. 
\begin{lemma} \label{lemma1} 
Let $\bar{\tau}>0$ be a fixed small number. Then for all $t\geq 0$, the estimate 
%\begin{subequations} \label{Eq_Norms}
 \begin{equation} \label{eq16}
\begin{aligned}
 \mathcal{E}[\Theta](t) \lesssim (1+\bar{\tau} + & \bar{\tau}^2)\big( E^{\bar{\tau}}[\Theta,q](t)+\Vert f(t)\Vert^2_{L^2}+\Vert f_t(t)\Vert^2_{L^2}\big)
\end{aligned} 
\end{equation} 
holds uniformly in $\tau\in (0,\bar{\tau}]$.
%\begin{equation} \label{eq32}
%(\tau +\tau^3) E[\Theta, q](t) \lesssim \Vert q \Vert^2_{L^2} + (1+\tau+\tau^2+\tau^3) \mathcal{E}[\Theta](t).
%\end{equation}
%\end{subequations}
The hidden constant in \eqref{eq16} does not depend on $\tau$.
\end{lemma}   
\begin{proof}
It is clear that we have for all $t\geq 0$  
\begin{equation} \label{eq14}
\mathcal{E}_0[\Theta](t)\leq E^{\tau}[\Theta,q](t).
\end{equation}
Now, if we want to show that $\mathcal{E}_1[\Theta](t)$  is also bounded by $E^{\tau}[\Theta, q](t)$, we can simply make use of the second equation in \eqref{eq_Cattaneo} to get 
\begin{equation}
\kappaa^2\|\nabla \Theta\|_{L^2}^2\leq 2\tau^2\|q_t\|_{L^2}^2+2 \|q\|_{L^2}^2\leq C(\tau)E^{\tau}[\Theta, q].
\end{equation}
However the constant $C(\tau)\rightarrow \infty$ as $\tau\rightarrow 0$. To avoid this, we take the time derivative $\partial_t^k, k=0,1$ of the system \eqref{eq_Cattaneo} to obtain 
\begin{equation}\label{eq_Cattaneo_k_1}
\left\{    
\begin{aligned}
&m\partial_t^k \Theta_t +\nabla\cdot (\partial_t^kq) + \ell \partial_t^k\Theta = \partial_t^k f, \qquad &\text{in} \ \Omega \times (0,T),
\\   
&\tau \partial_t^kq_t+\partial_t^kq+\kappaa \nabla (\partial_t^k \Theta)=0, \qquad &\text{in} \ \Omega \times (0,T).
\end{aligned}
\right.
\end{equation}
%\begin{equation}
%\begin{aligned}
%\Vert \nabla \Theta (t)\Vert^2_{L^2} &\leq \frac{2 \tau^2}{\kappaa^2} \Vert q_t(t) \Vert^2_{L^2}+ \frac{2}{ \kappaa^2} \Vert q(t)\Vert^2_{L^2},\\
%&  \lesssim  \frac{1}{\tau} E_1[\Theta,q](t)+\tau E_2[\Theta,q](t), \quad t \geq 0.
%\end{aligned}
%\end{equation}
We multiply the second equation in \eqref{eq_Cattaneo_k_1} by $\nabla(\partial_t^k \Theta), k=0,1$ and  integrate over $\Om$.  Using the divergence theorem and the fact that $\partial_t^k \Theta|_{\partial \Om}=0$, we obtain  
%\begin{equation} 
%\tau \intO q_t \cdot \nabla \Theta \dx+\intO q \cdot \nabla \Theta \dx+\kappaa \Vert \nabla \Theta \Vert_{L^2}^2=0,
%\end{equation}    
%which gives after integrating by parts the second term  
\begin{equation}\label{nabla_theta}  
\kappaa \Vert \nabla (\partial_t^k \Theta) \Vert_{L^2}^2=-\tau \intO \partial_t^k q_t \cdot \nabla (\partial_t^k \Theta) \dx+\intO \nabla \cdot (\partial_t^k q) \partial_t^k \Theta \dx, \quad k=0, 1.
\end{equation}
Note that we can recover the last term on the right from the first equation in \eqref{eq_Cattaneo_k_1} by testing by $\partial_t^k \Theta$ and integrating over $\Om$. Hence, we get
\begin{equation} \label{q_term} 
\intO \nabla \cdot (\partial_t^k q) \partial_t^k \Theta \dx=-m\intO \partial_t^k \Theta_t \partial_t^k \Theta \dx-\ell \Vert \partial_t^k  \Theta \Vert^2_{L^2}+\intO \partial_t^k f \partial_t^k \Theta \dx. 
\end{equation}
Combining \eqref{nabla_theta} and \eqref{q_term}, we have
\begin{equation}
\begin{aligned}
\kappaa \Vert \nabla (\partial_t^k \Theta) \Vert_{L^2}^2=-\tau &\intO \partial_t^k  q_t \cdot \nabla (\partial_t^k \Theta) \dx-m\intO \partial_t^k \Theta_t \partial_t^k \Theta \dx\\
&-\ell \Vert \partial_t^k  \Theta \Vert^2_{L^2}+\intO \partial_t^k f \partial_t^k  \Theta \dx.
\end{aligned}
\end{equation}
Thus using Cauchy-Schwarz and Young inequalities, it results that   
\begin{equation}\label{eq15_k}
\begin{aligned}
\frac{\kappaa}{2} \Vert \nabla \partial_t^k\Theta \Vert_{L^2}^2+\frac{\ell}{2}\|\partial_t^k\Theta\|_{L^2}^2 &\leq \frac{\tau^2}{2 \kappaa} \Vert \partial_t^kq_t \Vert^2_{L^2} + \frac{m^2}{\ell}\Vert \partial_t^k\Theta_t \Vert^2_{L^2}+\frac{1}{\ell}\Vert \partial_t^kf \Vert^2_{L^2}\\
&\leq  \Big(\frac{2m}{\ell\kappaa} +\frac{\tau}{\kappaa}\Big) E^{\tau}[\Theta, q]+\frac{1}{\ell}\Vert \partial_t^k f \Vert^2_{L^2},\quad k=0,1. 
\end{aligned}
\end{equation}
The above estimate for $k=0$ also implies 
\begin{equation} \label{eq15+}
\frac{\tau \ell}{2}  \Vert \nabla \Theta \Vert_{L^2}^2  \lesssim (\tau+\tau^2) E^{\tau}[\Theta, q]+\tau \Vert  f \Vert^2_{L^2}.
\end{equation}

Next, we focus on estimating the term $\Vert \Delta \Theta \Vert_{L^2}$ in $\mathcal{E}_1[\Theta](t)$. First, we take the time derivative of the first equation in \eqref{eq_Cattaneo} and apply the divergence operator  to the second equation, so that we have
\begin{equation}\label{System_Deriv}
\left\{
\begin{aligned}
&m \Theta_{tt}+\nabla \cdot q_t + \ell \Theta_t=f_t,\\
&\tau \nabla \cdot q_t + \nabla \cdot q + \kappaa \Delta \Theta=0.\\
\end{aligned}
\right.
\end{equation}
We note that the above process of obtaining system \eqref{System_Deriv}  is formal.   This can be made rigorous by first considering smooth approximations for the initial data so that the corresponding solutions will  have the required smoothness that will allow us to pass to the limit to infer that \eqref{System_Deriv} still holds.   \\

Combining these two equations, 
%we find
%\begin{equation}\label{eq11}
%\tau m \Theta_{tt}+\tau \ell  \Theta_t-\nabla \cdot q-\kappaa \Delta \Theta =\tau f_t.
%\end{equation} 
%By summing up \eqref{eq11} 
and using the first equation in \eqref{eq_Cattaneo}, we obtain
\begin{equation} \label{telegraph_eq}
\tau m \Theta_{tt}+(m +\tau \ell) \Theta_t +\ell \Theta-\kappaa \Delta \Theta =f+ \tau f_t.
\end{equation}
From here, we infer that
\begin{equation}\label{Delta_Theta_Est}
\begin{aligned}
\kappaa^2 \Vert \Delta \Theta \Vert^2_{L^2} &\lesssim   \ell^2 \Vert \Theta \Vert^2_{L^2}+(m+\tau \ell)^2 \Vert \Theta_t \Vert^2_{L^2}+\tau^2 m^2 \Vert \Theta_{tt} \Vert^2_{L^2}+\Vert f \Vert^2_{L^2}+\tau^2 \Vert f_t \Vert^2_{L^2}\\
&\lesssim (1 +\tau^2)\big(E^{\tau}[\Theta, q]+\Vert f \Vert^2_{L^2}+\Vert f_t \Vert^2_{L^2}\big).\\ 
\end{aligned}
\end{equation}
Putting together the estimates \eqref{eq14}, \eqref{eq15_k}, \eqref{eq15+}, \eqref{Delta_Theta_Est} and noting that $\tau \in (0, \bar{\tau}]$, we get the  estimate \eqref{eq16}.
%\begin{equation}
%\begin{aligned}
%\mathcal{E}[\Theta](t)\lesssim ( &1+\tau+\tau^2)E [\Theta,q](t)+\Vert f(t) \Vert^2_{L^2}+(1+\tau^2) \Vert f_t(t) \Vert^2_{L^2}.\\
%\end{aligned}
%\end{equation}
%Next, to prove \eqref{eq32}, notice that from the second equation in \eqref{eq_Cattaneo} we have
%\begin{equation} \label{1}
%\tau^2 \Vert q_t \Vert^2_{L^2} \leq 2\Vert q \Vert^2_{L^2} +2 \kappaa^2 \Vert \nabla \Theta \Vert^2_{L^2}.
%\end{equation}
%We can also obtain a similar estimate for $q_{tt}$. Taking the time derivative of the second equation in \eqref{eq_Cattaneo}
%\begin{equation}
%\tau q_{tt}+q_t+ \kappaa \nabla \Theta_t=0,
%\end{equation}
%it results that
%\begin{equation} \label{2} 
%\tau^2 \Vert q_{tt} \Vert^2_{L^2} \leq 2\Vert q_{t} \Vert^2_{L^2} +2 \kappaa^2 \Vert \nabla \Theta_t \Vert^2_{L^2}.
%\end{equation}  
%Then, adding $\dfrac{\tau^2}{4}\times\eqref{2}$ to \eqref{1}, we find 
%\begin{equation} 
%\frac{\tau^2}{2} \Vert q_t \Vert^2_{L^2}+\frac{\tau^4}{4} \Vert q_{tt} \Vert^2_{L^2} \leq 2 \Vert q \Vert^2_{L^2} + 2 \kappaa^2 \Vert \nabla \Theta \Vert^2_{L^2}+\frac{\kappaa^2 \tau^2}{2} \Vert \nabla \Theta_t \Vert^2_{L^2}.
%\end{equation}
%Thus, recalling the definition of $E[\Theta, q]$ and $\mathcal{E}[\Theta]$, we arrive at the estimate \eqref{eq32}. 
This completes the proof of Lemma \ref{lemma1}.
\end{proof}
\begin{proposition}\label{total-energy}
Given $T>0$,  and $f \in H^2(0, T; L^2(\Om))$. Assume that the initial data $(\Theta_0,q_0)$ satisfy Assumption \ref{Assumption_compatibility}. 
%$$(\Theta_0,q_0) \in {\color{blue}\big[}H^2(\Om)\cap H^1_0(\Om){\color{blue}\big]}  \times (H^2(\Om))^d. $$
Then, the system \eqref{eq_Cattaneo} has a unique solution $(\Theta, q) \in X_\Theta \times X_q$. Furthermore, the solution satisfies 
\begin{equation}\label{Pro_1_Estimate}
\begin{aligned}
&\mathcal{E}[\Theta](t) +\int_0^t \mathcal{D}[\Theta](s) \ds \lesssim (1+\bar{\tau}+\bar{\tau}^2) \big(E^{\bar{\tau}}[\Theta, q](0)+ \Vert f \Vert^2_{H^2L^2}\big) 
\end{aligned}     
\end{equation}
and 
\begin{equation} \label{q_estimate}
\begin{aligned}
&\Vert q(t) \Vert_{H^1}^2+  \sum_{k=0}^2 \int^t_0 \Vert \partial_t^k  q(s) \Vert_{L^2}^2 \ds \\
\lesssim &\,  \Vert q_0 \Vert_{H^1}^2 +(1+\bar{\tau}+\bar{\tau}^2)\big( E^{\bar{\tau}}[\Theta, q](0)+\Vert f \Vert^2_{H^2 L^2}\big)\\
 \end{aligned}    
\end{equation}
for all $0 \leq t \leq T$, where the hidden constants in  \eqref{Pro_1_Estimate} and \eqref{q_estimate} are uniform with respect to $\tau\in(0,\bar{\tau}]$. 
\end{proposition}

We rely on the Faedo--Galerkin approach to prove the result above. In the first step, we show some uniform \textit{a priori} estimates to hold for the approximations of the solution $(\Theta, q)$. Precisely, we take the eigenfunctions of the Dirichlet--Laplacian as approximations for the temperature $\Theta$. As for the heat flux $q$, we exploit the fact that $(H^2(\Om))^d$ is separable, yielding the existence of a dense sequence that can be used to approximate $q$ (see \cite{L69}). In the second step,  due to the fact that the estimates are uniform, we can pass to the limit as  in  \cite[chapter 7]{evans2010partial} (see also  \cite[Chapter 1]{L69}) to obtain the existence of solutions. Further, uniqueness follows by noting that the only solution to the homogeneous problem ($f=0, \Theta_0=0$ and $q_0=0$) is $(\Theta, q)=(0, 0)$. Lastly, from the weak and weak-$\star$ semi-continuity of norms, we get that the estimates \eqref{Pro_1_Estimate}, \eqref{q_estimate} also hold for the solution $(\Theta, q)$. Since the approach is quite classical, our emphasis in the following is on the energy analysis, and  
%On account of the approach being rather classical, we focus below on the energy analysis and     
we refer to \cite{evans2010partial, L69} for the procedure of passing to the limit. To start with, we prove the following intermediate estimates.
\begin{lemma}
Let $\tau\in (0,\bar{\tau}]$. Assume that $f \in H^2(0, T; L^2(\Om))$. Then for all $0 \leq t \leq T$, we have the following estimate 
\begin{equation} \label{estimate_E_k}
\ddt E_k[\Theta,q](t)+D_k[\Theta,q](t) \lesssim \Vert \partial_t^k f(t) \Vert^2_{L^2}, \quad k=0, 1, 2.
\end{equation}
where the hidden constant is independent of $\tau$.
\end{lemma}    
\begin{proof}
Applying $\partial_t^k,\ k=0,1,2$ to system \eqref{eq_Cattaneo}, we get 
 \begin{equation}\label{eq_Cattaneo_k}
 \left\{
\begin{aligned}
&m\partial_t^k\Theta_t +\nabla\cdot \partial_t^kq + \ell \partial_t^k\Theta = \partial_t^kf, \qquad &\text{in} \ \Omega \times (0,T),
\\   
&\tau \partial_t^kq_t+\partial_t^kq+\kappaa \nabla \partial_t^k\Theta=0, \qquad &\text{in} \ \Omega \times (0,T).
\end{aligned}
\right.
\end{equation}
Multiplying the first equation in \eqref{eq_Cattaneo_k} by $ \kappaa \partial_t^k \Theta$ and integrating over $\Om$, it follows
\begin{equation} \label{eq3_k}
\frac{m \kappaa}{2} \ddt \Vert \partial_t^k\Theta \Vert^2_{L^2}+\kappaa \intO \partial_t^k\Theta \nabla \cdot \partial_t^k q \dx +\ell \kappaa \Vert \partial_t^k\Theta \Vert^2_{L^2}=\kappaa \intO \partial_t^k f \partial_t^k\Theta \dx.
\end{equation}
Using the divergence theorem on the second term on the left and the fact that $\partial_t^k\Theta|_{\partial\Omega}=0$ (due to \eqref{homog_dirichlet}), we obtain
\begin{equation}\label{eq1}
    \frac{m \kappaa}{2} \ddt \Vert \partial_t^k\Theta \Vert^2_{L^2}-\kappaa\intO \partial_t^kq \cdot \nabla \partial_t^k\Theta \dx+\ell \kappaa \Vert \partial_t^k\Theta \Vert^2_{L^2}=\kappaa\intO \partial_t^kf \partial_t^k\Theta \dx.
\end{equation}
Next, multiplying the second equation in \eqref{eq_Cattaneo_k} by $\partial_t^kq$ and integrating over $\Omega$,  we get
\begin{equation} \label{eq2}    
    \frac{\tau}{2} \ddt \Vert \partial_t^k q \Vert^2_{L^2}+ \Vert \partial_t^k q \Vert^2_{L^2}+\kappaa \intO \partial_t^k q \cdot \nabla \partial_t^k \Theta \dx=0.
\end{equation}
Collecting  \eqref{eq1} and \eqref{eq2} yields
\begin{equation}
\ddt E_k[\Theta,q]+D_k[\Theta,q]=\kappaa\intO \partial_t^k f \partial_t^k \Theta \dx, 
\end{equation}
where $E_k[\Theta,q]$ and $ D_k[\Theta,q]$ are given by \eqref{E_k_def}, \eqref{D_k_def}, respectively. Using  Cauchy--Schwarz inequality together with Young's inequality on the right-hand side gives the desired estimate \eqref{estimate_E_k}. 
\end{proof}  
 
\begin{remark}
Let $f\in H^2(0,T,L^2(\Omega))$. Recalling \eqref{E_Total}--\eqref{D_k_def} and using  \eqref{estimate_E_k}, we obtain 
\begin{equation} \label{tolal_energy_E}
\ddt E^{\tau}[\Theta, q](t)+ cE^{\tau}[\Theta, q](t)\lesssim \|f (t) \|^2_{L^2}+\|f_t(t) \|^2_{L^2}+\|f_{tt}(t) \|^2_{L^2}
\end{equation}
with 
\begin{equation}
c=\min\left\{\frac{\ell}{m}, \frac{2}{\tau}\right\}.
\end{equation}
Multiplying \eqref{tolal_energy_E} by $e^{ct}$, we obtain
\begin{equation}
\ddt\left(e^{ct}E^{\tau}[\Theta, q]\right)(t)\lesssim e^{ct}\big(\|f (t) \|^2_{L^2}+\|f_t(t) \|^2_{L^2}+\|f_{tt}(t) \|^2_{L^2} \big), \quad t \geq 0
\end{equation}
Hence, integrating in time,  we have
\begin{equation} \label{time_integrated_E}
\begin{aligned}    
E^{\tau}[\Theta, q](t)&\lesssim e^{-ct} E^{\tau}[\Theta, q](0)+\int_0^t e^{-c(t-s)}(\|f (s) \|^2_{L^2}+\|f_t(s) \|^2_{L^2}+\|f_{tt}(s) \|^2_{L^2})\ds\\
& \lesssim e^{-ct} E^{\tau}[\Theta, q](0)+\frac{1-e^{-ct}}{c}\| f\|_{H^2L^2}^2.\\
\end{aligned}
\end{equation} 
It is clear from \eqref{time_integrated_E} that $E^{\tau}[\Theta, q](t)$ has an exponential decay.
\end{remark}

\subsection{Proof of Proposition \ref{total-energy}}\label{Proof_Prop_1}

First, we integrate over time the estimate \eqref{estimate_E_k}, sum up over $k=0,1,2$ and recall  \eqref{E_Total} and \eqref{D_total} to find 
\begin{equation}\label{E_Main_Estimate}
E^{\tau}[\Theta,q](t)+\int_0^t D[\Theta,q](s) \ds \lesssim E^{\tau}[\Theta, q](0)+\| f\|_{H^2L^2}^2,
\end{equation}
which also implies, by recalling  \eqref{E_0} and  \eqref{D_0} that   
\begin{equation} \label{estimate_E_0}
\mathcal{E}_0[\Theta](t)+\int_0^t \mathcal{D}_0[\Theta](s) \ds \lesssim E^{\tau}[\Theta, q](0)+\| f\|_{H^2L^2}^2.
\end{equation}
%\begin{equation} 
%E[\Theta, q](t)+{\rm cont \cdot} \int_0^t E[\Theta, q](s)\ds\lesssim  E[\Theta, q](0)+\Vert f \Vert^2_{L^2(L^2)}+\Vert f_{t} \Vert^2_{L^2(L^2)}+\Vert f_{tt} \Vert^2_{L^2(L^2)}
%\end{equation}
On the other hand, applying the divergence operator to the second equation in \eqref{eq_Cattaneo} and testing the resulting equation by $\Delta \Theta $, we obtain for all $t \geq 0$ 
\begin{equation} \label{eq8}
\tau \intO \nabla \cdot q_t \Delta \Theta \dx+\intO\nabla \cdot q \Delta \Theta \dx+\kappaa \Vert \Delta \Theta \Vert^2_{L^2}=0.
\end{equation}
Next, we multiply the first equation in \eqref{eq_Cattaneo} by $\Delta \Theta$ and  integrate by parts to get 
\begin{equation} \label{eq6}
\begin{aligned}
\intO\nabla \cdot q \Delta \Theta \dx 
&=\frac{m}{2} \ddt \Vert \nabla \Theta \Vert^2_{L^2}+\ell \Vert \nabla \Theta \Vert^2_{L^2}+\intO f \Delta \Theta \dx.
\end{aligned}
\end{equation}
Further, differentiating the first equation in \eqref{eq_Cattaneo} with respect to $t$, we find 
\begin{equation} \label{Theta_t_Eq}
m \Theta_{tt}+\nabla \cdot q_t+\ell \Theta_t=f_t.
\end{equation}
Testing the above equation  by $\tau \Delta \Theta$ and integrating over $\Om$, it follows
\begin{equation} \label{eq7}
\begin{aligned}
\tau \intO\nabla \cdot q_t \Delta \Theta \dx&=-\tau m \intO \Theta_{tt} \Delta \Theta \dx+\frac{\tau \ell}{2} \ddt \Vert \nabla \Theta \Vert^2_{L^2}+\tau \intO f_t \Delta \Theta \dx.\\
\end{aligned}
\end{equation}  
Plugging \eqref{eq6} and \eqref{eq7} into \eqref{eq8}, we deduce 
\begin{equation} \label{laplacian_identity}
\begin{aligned}
&\frac{1}{2}(m+\tau \ell) \ddt \Vert \nabla \Theta \Vert^2_{L^2}+ \ell \Vert \nabla \Theta \Vert^2_{L^2}+\kappaa \Vert \Delta \Theta \Vert^2_{L^2}\\
=\tau  & m \intO \Theta_{tt} \Delta \Theta \dx-\intO f \Delta \Theta \dx-\tau \intO f_t \Delta \Theta \dx.\\
\end{aligned}
\end{equation}
Therefore, by applying Cauchy-Schwarz and Young inequalities, we have
\begin{equation} \label{laplacian_estimate_1} 
\frac{1}{2}(m+\tau \ell) \ddt \Vert \nabla \Theta \Vert^2_{L^2}+ \ell \Vert \nabla \Theta \Vert^2_{L^2}+\frac{\kappaa}{4} \Vert \Delta \Theta \Vert^2_{L^2} \lesssim \tau^2 \Vert \Theta_{tt} \Vert^2_{L^2}+\Vert f \Vert^2_{L^2}+ \tau^2 \Vert f_t \Vert^2_{L^2}.
\end{equation}
Integrating with respect to $t$, we get
\begin{equation} \label{eq33} 
\begin{aligned}
& \frac{1}{2}(m+\tau \ell) \Vert \nabla \Theta(t) \Vert^2_{L^2}+ \ell \int_0^t\Vert \nabla \Theta \Vert^2_{L^2} \ds+\frac{\kappaa}{4} \int_0^t \Vert \Delta \Theta \Vert^2_{L^2} \ds \\
\lesssim &\, \frac{1}{2}(m+\tau \ell) \Vert \nabla \Theta_0 \Vert^2_{L^2}+\tau^2 \int_0^t  \Vert \Theta_{tt} \Vert^2_{L^2} \ds+\Vert f \Vert^2_{L^2 L^2}+ \tau^2 \Vert f_t \Vert^2_{L^2 L^2}.
\end{aligned}
\end{equation}
Using inequality \eqref{eq16}, we have for all $\tau\in(0,\bar{\tau}]$
\begin{equation}
\begin{aligned}
(1+\tau) \Vert \nabla \Theta (t)\Vert_{L^2}^2 &\lesssim (1+\bar{\tau} +\bar{\tau}^2) \big(E^{\bar{\tau}}[\Theta, q](t)+ \Vert f(t) \Vert^2_{H^1L^2}\big).\\
\end{aligned}
\end{equation}
Further, according to \eqref{1D_embedding_1}, we have
%{\color{blue} using the interpolation inequality
%\begin{equation}
%\|f\|_{L^\infty_tL^2}\lesssim \|f\|_{L^2L^2_t}^{1/2}\|\partial_t f\|_{L^2L^2}^{1/2},  
%\end{equation}
%
%} 
\begin{equation}
\Vert \partial_t^k f(0) \Vert^2_{L^2}  \leq C_T\big(\Vert \partial_t^k   f \Vert^2_{L^2L^2}+\Vert \partial_t^k  f_t \Vert^2_{L^2L^2}\big), \quad k=0, 1.
\end{equation}
Then, it follows that
\begin{equation}\label{grad_Theta_0_Estimate}
(1+\tau) \Vert \nabla \Theta_0 \Vert_{L^2}^2 \lesssim (1+\bar{\tau}+\bar{\tau}^2)\big(E^{\bar{\tau}}[\Theta, q](0)+\Vert f \Vert^2_{H^2L^2}\big).  
\end{equation}
Moreover, from estimate \eqref{E_Main_Estimate}, we have 
\begin{equation} \label{eq34}
\tau^2 \int_0^t \Vert \Theta_{tt} \Vert^2_{L^2} \ds \lesssim  \bar{\tau}^2 \big(E^{\bar{\tau}}[\Theta, q](0)+\Vert f \Vert^2_{H^2 L^2}\big).
\end{equation}
Combining the estimates  \eqref{eq33}, \eqref{grad_Theta_0_Estimate} and \eqref{eq34},  we obtain 
\begin{equation}\label{gradient_estimate} 
\begin{aligned}
 \frac{1}{2}(m+&\tau \ell) \Vert \nabla \Theta(t) \Vert^2_{L^2}+ \ell \int_0^t \Vert \nabla \Theta \Vert^2_{L^2}\ds+\frac{\kappaa}{4} \int_0^t\Vert \Delta \Theta \Vert^2_{L^2}\ds \\
\lesssim &\, (1+\bar{\tau}+\bar{\tau}^2)\big(E^{\bar{\tau}}[\Theta, q](0)+\Vert f \Vert^2_{H^2L^2}\big).
\end{aligned}
\end{equation}
Next, according to the first inequality in \eqref{eq15_k} for $k=1$ and  the estimate \eqref{E_Main_Estimate}, we have for all $t\geq 0$
\begin{equation}\label{theta_t_estimate}
\begin{aligned}
\kappaa \int_0^t \Vert \nabla \Theta_t (s)\Vert^2_{L^2} \ds  & \lesssim \tau^2 \int_0^t  \Vert q_{tt}(s)\Vert^2_{L^2} \ds +\int_0^t \Vert \Theta_{tt} (s)\Vert^2_{L^2} \ds+\Vert f_t \Vert^2_{L^2L^2}\\
& \lesssim (1+\bar{\tau}^2) (E^{\bar{\tau}}[\Theta, q](0)+\Vert f \Vert^2_{H^2 L^2}).
\end{aligned}
\end{equation} 
On the other hand, thanks to the 1D embedding $H^1(0, T) \hookrightarrow L^\infty (0,T)$, we have $f, f_t \in L^\infty(0, T; L^2(\Om))$ (see inequality \eqref{1D_embedding_1}). Then, the estimate \eqref{eq16} together with \eqref{E_Main_Estimate} gives for all $t \geq 0$ 
\begin{equation}\label{eq9} 
\begin{aligned}   
\kappaa \Vert \nabla \Theta_t (t)\Vert^2_{L^2}   \lesssim &\, (1+\bar{\tau}+\bar{\tau}^2) \big(E^{\bar{\tau}}[\Theta,q](0)+\Vert f \Vert^2_{H^2 L^2}\big). 
\end{aligned}  
\end{equation}
Similarly, we make use of \eqref{eq16}, \eqref{E_Main_Estimate} and the fact that $f, f_t \in L^\infty(0, T; L^2(\Om))$ to get 
%\begin{equation}  
%\begin{aligned}
%\kappaa \Vert \Delta \Theta (t)\Vert^2_{L^2}  \lesssim &\, E_0 [\Theta,q](0)+ (1+\tau^2) E_1[\Theta,q](0)+\tau^2 E_2[\Theta,q](0)\\
%& +\Vert f \Vert^2_{L^2 L^2}+(1+\tau^2)\Vert f_{t} \Vert^2_{L^2 L^2}+ \tau^2 \Vert f_{tt} \Vert^2_{L^2 L^2}. 
%\end{aligned}  
%\end{equation}
\begin{equation} \label{eq20} 
\begin{aligned}
\kappaa \Vert \Delta \Theta (t)\Vert^2_{L^2}  \lesssim &\, (1+\bar{\tau}+\bar{\tau}^2) \big(E^{\bar{\tau}}[\Theta,q](0)+\Vert f \Vert^2_{H^2 L^2}\big). 
\end{aligned}  
\end{equation}
Collecting the estimates \eqref{estimate_E_0}, \eqref{gradient_estimate}, \eqref{theta_t_estimate}, \eqref{eq9} and \eqref{eq20}, we arrive at the estimate \eqref{Pro_1_Estimate}.

To establish estimate \eqref{q_estimate}, first observe that from \eqref{E_Main_Estimate}, we immediately have for $t \geq 0$ 
%\begin{equation} 
%\tau (\Vert q(t) \Vert_{L^2}^2+\Vert q_{t}(t) \Vert_{L^2}^2+ \Vert q_{tt}(t) \Vert_{L^2}^2) \lesssim E^{\tau}[\Theta,q](0)+ \Vert f \Vert_{H^2L^2}^2. 
%\end{equation}
\begin{equation} \label{q_derivatives} 
\int_0^t (\Vert q(s) \Vert_{L^2}^2+\Vert q_{t}(s) \Vert_{L^2}^2+ \Vert q_{tt}(s) \Vert_{L^2}^2) \ds \lesssim E^{\tau}[\Theta,q](0)+ \Vert f \Vert_{H^2L^2}^2. 
\end{equation}
Thus, it remains to prove that $q$ is in $(H^1(\Om))^d$. From the second equation in \eqref{eq_Cattaneo}, we obtain
\begin{equation} \label{solution_q}
q(x,t)=q_0(x)-\frac{\kappaa}{\tau} \int_{0}^t \nabla \Theta e^{-(t-s)/\tau}\ds,
\end{equation}
which implies that for all $t \geq 0$
\begin{equation}
\Vert q(t) \Vert_{H^1} \leq \Vert q_0 \Vert_{H^1}+ \frac{\kappaa}{\tau} \int_0^t e^{-(t-s)/\tau} \Vert \nabla \Theta(s) \Vert_{H^1} \ds. 
\end{equation}
Using elliptic regularity, one finds
\begin{equation}
\begin{aligned} 
\Vert q(t) \Vert_{H^1} & \leq \Vert q_0 \Vert_{H^1}+\kappaa(1-e^{-t/\tau}) \Vert \Delta \Theta \Vert_{L^\infty L^2}\\
& \lesssim  \Vert q_0 \Vert_{H^1}+ \Vert \Delta \Theta \Vert_{L^\infty L^2}.
\end{aligned}
\end{equation}
This last estimate along with \eqref{eq20} gives
\begin{equation}\label{q_H1}
\begin{aligned}
\Vert q(t) \Vert_{H^1}^2 \lesssim \Vert q_0 \Vert_{H^1}^2+(1+\bar{\tau}+\bar{\tau}^2) \big( E^{\bar{\tau}}[\Theta, q](0)+\Vert f \Vert^2_{H^2 L^2} \big).
\end{aligned}
\end{equation}
Summing up the estimates \eqref{q_derivatives} and \eqref{q_H1}, we obtain    \eqref{q_estimate}. Further, from \eqref{Pro_1_Estimate}, \eqref{q_estimate}, we deduce that $(\Theta, q) \in X_\Theta \times X_q$. This concludes the proof of Proposition \ref{total-energy}. 
%%%%%%%%%%%%%%%%%%%%%%%%%%%%%%%%%%%%%%%%%%%%%%%%%%%%%%%%%%%%%%%%%%%%%%%%%%%%%%%%%%%%%%%%%%
\section{The linearized Westervelt equation} \label{sect5}
In this section, we consider the linearization of the Westervelt equation in \eqref{Main_system}: 
\begin{equation} \label{linear_West}
    \alpha(x,t) p_{tt}-r(x,t) \Delta p-b \Delta p_t=g(x,t), \qquad x \in \Om, \quad t \geq 0,
\end{equation}
supplemented by the relevant initial conditions in  \eqref{init_cond} and homogeneous Dirichlet boundary condition \eqref{homog_dirichlet}.
 
We define the energies associated to \eqref{linear_West}  as follows
\begin{equation}
    \begin{aligned}
        E_1[p](t)&:=\frac{1}{2} \big(\Vert \sqrt{\alpha(t)}p_t(t) \Vert^2_{L^2}+\Vert \sqrt{r(t)}\nabla p(t) \Vert^2_{L^2} \big),\\
        E_2[p](t)&:=\frac{1}{2} \big(\Vert \sqrt{\alpha(t)}p_{tt}(t) \Vert^2_{L^2}+\Vert \sqrt{r(t)}\nabla p_t(t) \Vert^2_{L^2} + \Vert \sqrt{b}\Delta p(t) \Vert^2_{L^2} \big),\\   
        E_3[p](t)&:=\frac{1}{2} \big(\Vert \sqrt{b} \nabla p_{tt}(t) \Vert_{L^2}^2+\Vert \sqrt{b} \nabla \Delta p(t) \Vert^2_{L^2} \big).\\
    \end{aligned}
\end{equation}
The total acoustic energy is given by
\begin{equation}\label{Energy_Tot}
    \mathfrak{E}[p](t):=\sum_{k=1}^3 E_k[p](t), \qquad t\geq 0.
\end{equation}
We denote by  $\mathfrak{D}[p]$  its  associated dissipation rate given by 
\begin{equation}\label{dissipation}
\begin{aligned}
\mathfrak{D}[p](t):=\mathfrak{D}_0[p](t)+ b \Vert \nabla \Delta p_t(t)  \Vert^2_{L^2} +b \Vert \Delta p_{tt}(t) \Vert^2_{L^2}, \qquad t\geq 0
\end{aligned}
\end{equation}
where $\mathfrak{D}_0[p]$ is given by 
\begin{equation}\label{dissipation_D_0}
\begin{aligned}
\mathfrak{D}_0[p](t):=b \Vert \nabla p_t(t)& \Vert^2_{L^2}+b \Vert \nabla p_{tt}(t) \Vert^2_{L^2}+\Vert \sqrt{r(t) } \Delta p(t) \Vert_{L^2}^2+b \Vert \Delta p_t(t) \Vert^2_{L^2}\\
&+\Vert \sqrt{r(t) } \nabla \Delta p(t) \Vert^2_{L^2} +\Vert \sqrt{\alpha(t) } p_{ttt}(t) \Vert^2_{L^2}.\\
\end{aligned}
\end{equation}
%We introduce the following function space for the solutions of the linear acoustic pressure equation \eqref{linear_West}:
%\begin{equation}
%\begin{aligned}
%X_p=&\,\Big\{p \in L^{\infty}(0, T; H^3( \Om) \cap H^1_0(\Om)), \\
%&  \quad p_t \in L^{\infty}(0, T; H^2( \Om) \cap H^1_0(\Om))\cap L^{2}(0, T; H^3( \Om) \cap H^1_0(\Om)),\\
%& \quad p_{tt} \in L^{\infty}(0, T; H^1_0(\Om))\cap L^{2}(0, T; H^2( \Om) \cap H^1_0(\Om)),\\
%& \quad p_{ttt} \in L^{2}(0, T; L^2(\Om)) \Big\}.
%\end{aligned}
%\end{equation}
%{\color{red} The same space defined in \eqref{Functional_Spaces}.}{\color{magenta} Yes}\\
We make the following regularity and non-degeneracy assumptions on the coefficients $\alpha$ and $ r$.
\begin{assumption}\label{Assumption_1}
Assume that
\begin{enumerate}
\item [(A)]
\begin{itemize}
\item $\alpha \in L^{\infty}(0, T; L^\infty (\Om)) \cap L^2(0,T; W^{1,3}(\Om))$,
\item $\alpha_t \in L^2(0, T;L^3(\Om)) \cap L^{\frac{4}{4-d}}(0, T; L^2(\Om))$,
\item There exist $0<\alpha_0\leq \alpha_1$ such that 
\begin{equation}
\alpha_0 \leq \alpha(x, t) \leq \alpha_1 \quad \text{a.e. in}\quad \Om \times (0, T).
\end{equation}
\end{itemize}
\item [(R)]
\begin{itemize}
\item $r \in L^{\infty}(0, T; L^\infty (\Om)) \cap L^2(0,T; W^{1,3}(\Om))$,
\item $ r_t \in L^2(0, T;L^3(\Om)) \cap L^{\frac{4}{4-d}}(0, T; L^2(\Om))$,
\item There exist $0<r_0\leq r_1$
\begin{equation}
r_0 \leq r(x, t) \leq r_1 \quad \text{a.e. in} \quad \Om \times (0, T).
\end{equation}
\end{itemize}  
\end{enumerate}
\end{assumption}
The main result of this section reads as follows. 
\begin{proposition} \label{prop2} 
Given $T>0$ and $g \in H^1(0, T; L^2(\Om)) \cap L^2(0, T; H^1_0(\Om))$. Let $(p_0, p_1)$ satisfy Assumption \ref{Assumption_compatibility}. 
%\begin{equation}
%(p_0, p_1) \in \big[H^3( \Om) \cap H^1_0(\Om)\big] \times \big[H^2( \Om) \cap H^1_0(\Om)\big].
%\end{equation}
Then under Assumption \ref{Assumption_1}, the linearized pressure equation \eqref{linear_West} has a unique solution $p \in X_p$, which satisfies for all $0 \leq t \leq T$
%\begin{equation} 
%\begin{aligned}    
%&\mathfrak{E}[p](t)+b\Vert \Delta p_t(t) \Vert^2_{L^2}+\int_0^t \mathfrak{D}[p](s) \ds+\int_0^t b \Vert \nabla \Delta p_t(s) \Vert^2_{L^2} \ds +\int_0^t b \Vert \Delta p_{tt}(s) \Vert^2_{L^2} \ds\\ 
%\lesssim & \mathfrak{E}[p](0) \exp \Big(\int_0^t (1+ \Lambda(s)) \ds \Big)+\int_0^t \mathfrak{F}(s) \exp \Big( \int_s^t (1+\Lambda (\sigma)) \textup{d} \sigma \Big) \ds.\\
%\end{aligned}
%\end{equation}
\begin{equation} \label{total_energy_p}
\begin{aligned}    
\mathfrak{E}[p](t)+b\Vert \Delta p_t(t) \Vert^2_{L^2}+\int_0^t \mathfrak{D}[p](s) \ds 
\lesssim&\, \mathfrak{E}[p](0) \exp \Big(\int_0^t (1+ \Lambda(s)) \ds \Big)\\
&+\int_0^t \mathfrak{F}(s) \exp \Big( \int_s^t (1+\Lambda (\sigma)) \textup{d} \sigma \Big) \ds\\
\end{aligned}
\end{equation}
where 
\begin{equation}\label{lambda}
\begin{aligned}
\Lambda(t)=\Vert \alpha_t(t) \Vert_{L^2}^2 &+\Vert \alpha_t(t) \Vert_{L^2}^{\frac{4}{4-d}}+\Vert r_t(t) \Vert_{L^2}^{\frac{4}{4-d}}+\Vert \nabla r(t) \Vert_{L^2}^2+\Vert  r_t(t) \Vert_{L^3}^2 \\
&+\Vert  \alpha_t(t) \Vert_{L^3}^2+ \Vert \nabla r(t) \Vert_{L^3}^2+\Vert \nabla \alpha(t) \Vert_{L^3}^2,\\
\end{aligned}
\end{equation}
and
\begin{equation}\label{source_term}
\mathfrak{F}(t)=\Vert \nabla g(t) \Vert_{L^2}^2+\Vert g_t(t) \Vert^2_{L^{2}}.
\end{equation}
\end{proposition}     
We carry out the proof of Proposition \ref{prop2} by means of the Faedo-Galerkin method using the smooth eigenfunctions of the Dirichlet-Laplacian as approximations of the solution in space, see \cite{evans2010partial}. Below, we present the energy analysis for the mentioned approximations, which leads to  
 uniform \textit{a priori} estimates, and we direct the reader interested in the process of taking the limit towards  \cite[chapter 7]{evans2010partial}.  \\
 Before embarking on the proof of Proposition \ref{prop2}, we first prove two lemmas.  
\begin{lemma} \label{lemma2}
Given $g \in  L^2(\Om)$ such that $g_t \in L^{2}(\Om)$. Let Assumption \ref{Assumption_1} hold. Then for all $t \geq 0$, we have the following energy estimate
\begin{equation} \label{estimate_E_1_2_p}
    \begin{aligned}
       & \ddt \Big(E_1[p](t)+E_2[p](t)\Big)+b \Vert \nabla p_t(t) \Vert^2_{L^2}\\
        &+b \Vert \nabla p_{tt}(t) \Vert^2_{L^2}+\Vert \sqrt{r(t)} \Delta p(t) \Vert_{L^2}^2+b \Vert \Delta p_t(t) \Vert^2_{L^2}\\
        \lesssim&\, (1+\Lambda_1(t))(E_1[p](t)+ E_2[p](t))+\Vert g(t) \Vert_{L^2}^2+\Vert g_t(t) \Vert^2_{L^2}.\\
    \end{aligned}
\end{equation}
where 
\begin{equation}
\begin{aligned}
\Lambda_1(t)=&\,\Vert  r_t(t) \Vert_{L^3}^2+\Vert \nabla r(t) \Vert_{L^3}^2+\Vert r_t(t) \Vert_{L^2}^{\frac{4}{4-d}}+\Vert \nabla r(t) \Vert_{L^2}^2\\
&+\Vert \alpha_t(t) \Vert_{L^2}^{\frac{4}{4-d}}+ \Vert \alpha_t(t) \Vert_{L^2}^2.
\end{aligned}
\end{equation}
\end{lemma}
\begin{proof}
The proof is based on ideas that are drawn from  \cite{Nikolic_2022}. 
Multiplying the equation \eqref{linear_West} by $p_t$, integrating over $\Om$ and using integration by parts, we find
\begin{equation} \label{identity_1}
\begin{aligned}
 &\ddt E_1[p](t)+b\intO \vert \nabla p_t \vert^2 \dx\\
= &\, \intO  g  p_t  \dx+\frac{1}{2}\intO   \alpha_t p_t^2 \dx-\intO \nabla r \cdot \nabla p p_t \dx+\frac{1}{2}\intO r_t \vert \nabla p \vert^2 \dx.\\ 
\end{aligned}   
\end{equation}
Our goal now is to estimate the terms on the right-hand side of \eqref{identity_1}. Applying  Young and Poincar\'{e} inequalities, we  obtain 
\begin{equation}
\intO g  p_t  \dx \leq C(\varepsilon) \Vert g \Vert^2_{L^2} +\varepsilon \Vert \nabla p_t \Vert^2_{L^2}.  
\end{equation}
The next two terms on the right-hand side of \eqref{identity_1} can be estimated  using H\"{o}lder, Young and Poincar\'{e}'s inequalities as well as the embedding $H^1(\Om) \hookrightarrow L^4(\Om)$. Hence, we have  
\begin{equation}
\begin{aligned}
\intO \alpha_t p_t^2 \dx &\leq \Vert \alpha_t \Vert_{L^2} \Vert p_t \Vert_{L^4}^2\\
& \leq  C(\varepsilon) \Vert \alpha_t \Vert_{L^2}^2   \Vert \sqrt{r}\nabla p_t \Vert_{L^2}^2+ \varepsilon \Big\Vert \frac{1}{\sqrt{r}} \Big\Vert_{L^\infty}^2 \Vert \nabla p_t \Vert_{L^2}^2.
\end{aligned}
\end{equation}
Considering $\varepsilon$ suitably small and using Assumption \ref{Assumption_1}, we can absorb the $\varepsilon$--terms on the above estimates    by  the dissipation terms  on the left-hand side of $\eqref{identity_1}$.\\
Likewise, we obtain
\begin{equation}
\begin{aligned}
\intO \nabla r \cdot \nabla p p_t \dx &\leq  \Vert \nabla r \Vert_{L^2} \Vert \nabla p \Vert_{L^4} \Vert p_t \Vert_{L^4}\\
&\leq C(\varepsilon) \Vert \nabla r \Vert^2_{L^2} \Vert \sqrt{r}\nabla p_t \Vert^2_{L^2}  + \varepsilon \Big\Vert \frac{1}{r} \Big\Vert_{L^\infty}^2  \Vert \sqrt{r}\Delta p \Vert^2_{L^2} \\
\end{aligned}
\end{equation}
where we have also taken into account the elliptic estimate 
\begin{equation}\label{elliptic}
\Vert \nabla p \Vert_{H^1} \leq \Vert p \Vert_{H^2} \leq C \Vert \Delta p \Vert_{L^2}.
\end{equation}
To estimate  the last term on the right-hand side of \eqref{identity_1}, we employ the Ladyzhenskaya inequality \eqref{lady} together with \eqref{elliptic} to find 
\begin{equation}
\begin{aligned}
    \intO r_t \vert \nabla p \vert^2 \dx &\leq \Vert r_t \Vert_{L^2} \Vert \nabla p \Vert^2_{L^4}\\
    & \lesssim \Vert r_t \Vert_{L^2} \Big\Vert \frac{1}{r} \Big\Vert_{L^\infty}\Vert \sqrt{r}\nabla p \Vert_{L^2}^{2(1-\frac{d}{4})}  \Vert \sqrt{r} \Delta p \Vert_{L^2}^{\frac{d}{2}}\\
    &\leq C(\varepsilon)\Vert r_t \Vert_{L^2}^{\frac{4}{4-d}}\Vert \sqrt{r} \nabla p \Vert_{L^2}^2+\varepsilon \Big\Vert \frac{1}{r}  \Big\Vert_{L^\infty}^{\frac{4}{d}} \Vert \sqrt{r} \Delta p \Vert^2_{L^2}. \\
    \end{aligned}
\end{equation}
Altogether, the estimates above yield for all $t \geq 0$
\begin{equation}\label{estimate_E_p_1}
\begin{aligned}
    \ddt E_1[p](t)+b \Vert \nabla p_t(t) \Vert^2_{L^2}
 \lesssim &\,\Vert r_t(t) \Vert_{L^2}^{\frac{4}{4-d}}  E_1[p](t)\\
 &+ (\Vert \alpha_t(t) \Vert_{L^2}^2
 +\Vert \nabla r(t) \Vert^2_{L^2})E_2[p](t)\\
 &+\Vert g(t) \Vert_{L^2}^2
 +\varepsilon \Vert \sqrt{r(t)}\Delta p(t) \Vert^2_{L^2}.\\
    \end{aligned}
\end{equation}

In order to get an estimate for the energy $E_2[p]$, we  differentiate \eqref{linear_West} with respect to $t$ to  obtain
 \begin{equation} \label{linear_West_t}
    \alpha p_{ttt}-r \Delta p_t-b \Delta p_{tt}=-\alpha_t p_{tt} +r_t \Delta p +g_t.
\end{equation}
Next,  we multiply  \eqref{linear_West_t} by $p_{tt}$ and integrate over $\Om$. Using integration by parts, we find  
\begin{equation} \label{identity_2}
\begin{aligned}
\frac{1}{2} &\, \ddt \Big( \Vert \sqrt{\alpha}p_{t t} \Vert^2_{L^2}+ \Vert \sqrt{r} \nabla p_t \Vert^2_{L^2}\Big)+b \Vert \nabla p_{tt} \Vert^2_{L^2}\\
=-\frac{1}{2}&\intO \alpha_t p_{t t}^2 \dx+\frac{1}{2}\intO r_t \vert \nabla p_t \vert^2 \dx-\intO \nabla r \cdot \nabla p_t p_{tt} \dx\\
&+\intO r_t \Delta p p_{tt} \dx +\intO g_t p_{tt} \dx.\\
\end{aligned}
\end{equation}
First, Ladyzhenskaya and Young inequalities along with the embedding $H^1(\Om) \hookrightarrow L^4(\Om)$ allow us to get the following upper bound for the first term on the right-hand side of \eqref{identity_2}. That is,   
\begin{equation}
\begin{aligned}
-\frac{1}{2}\intO \alpha_t p_{tt}^2 \dx &\leq  \frac{1}{2}\Vert \alpha_t \Vert_{L^2} \Vert p_{tt} \Vert_{L^4}^2\\
&\lesssim \Big\Vert \frac{1}{\alpha} \Big\Vert_{L^\infty}^{1-\frac{d}{4}}\Vert \alpha_t \Vert_{L^2} \Vert \sqrt{\alpha}p_{tt} \Vert_{L^2}^{2(1-\frac{d}{4})}\Vert \nabla p_{tt} \Vert_{L^2}^{\frac{d}{2}}\\
& \leq   C(\varepsilon) \Vert \alpha_t \Vert_{L^2}^{\frac{4}{4-d}}  \Vert \sqrt{\alpha} p_{tt} \Vert_{L^2}^2+ \varepsilon \Big\Vert \frac{1}{\alpha} \Big\Vert_{L^\infty}^{\frac{4}{d}-1} \Vert \nabla p_{tt} \Vert_{L^2}^2.
\end{aligned}
\end{equation}
We can take $\varepsilon$ as small as needed in order to absorb the last term in the estimate above into the dissipation term on the left-hand side of \eqref{identity_2}.\\
Similarly, we can derive the following estimate for the second term on the right-hand side of \eqref{identity_2}
\begin{equation}
\begin{aligned}
\intO r_t \vert \nabla p_t \vert^2 \dx &\leq \Vert r_t \Vert_{L^2} \Vert \nabla p_t\Vert^2_{L^4}\\
&\lesssim \Big\Vert \frac{1}{r} \Big\Vert_{L^\infty}^{1-\frac{d}{4}}\Vert r_t \Vert_{L^2} \Vert \sqrt{r}\nabla p_t\Vert_{L^2}^{2(1-\frac{d}{4})}\Vert \Delta p_t\Vert_{L^2}^{\frac{d}{2}}\\
&\leq C(\varepsilon) \Vert r_t \Vert_{L^2}^{\frac{4}{4-d}} \Vert \sqrt{r}\nabla p_t\Vert^2_{L^2}+ \varepsilon  \Big\Vert \frac{1}{r} \Big\Vert_{L^\infty}^{\frac{4}{d}-1} \Vert \Delta p_t\Vert^2_{L^2}
\end{aligned}
\end{equation}
where we have also used elliptic regularity as in \eqref{elliptic}.  Moreover, making use of the embedding $H^1(\Om) \hookrightarrow L^6(\Om)$ and Poincar\'{e}'s inequality, it follows
%\begin{equation}\label{nabla_r_Estimate}
%\begin{aligned}
%    \intO \nabla r \cdot \nabla p_t p_{tt} \dx &\leq  \Vert \nabla r \Vert_{L^2} \Vert \nabla p_t \Vert_{L^4} \Vert p_{tt} \Vert_{L^4}\\
%    %& \lesssim \frac{1}{\sqrt{b}}\Vert \nabla r \Vert_{L^2} \Vert \sqrt{b} \Delta p_t \Vert_{L^2} \Vert \nabla p_{tt} \Vert_{L^2}\\
%    & \leq C(\varepsilon) \Vert \nabla r \Vert^2_{L^2} \Vert \sqrt{b} \Delta p_t \Vert_{L^2}^2+\varepsilon  \Vert \nabla p_{tt} \Vert^2_{L^2}.
%\end{aligned}    
%\end{equation}
\begin{equation}\label{nabla_r_Estimate}
\begin{aligned}
    \intO \nabla r \cdot \nabla p_t p_{tt} \dx &\leq  \Vert \nabla r \Vert_{L^3} \Vert \nabla p_t \Vert_{L^2} \Vert p_{tt} \Vert_{L^6}\\
    & \leq C(\varepsilon) \Vert \nabla r \Vert^2_{L^3} \Vert \sqrt{r} \nabla  p_t \Vert_{L^2}^2+\varepsilon \Big\Vert \frac{1}{\sqrt{r}} \Big\Vert_{L^\infty}^2 \Vert \nabla p_{tt} \Vert^2_{L^2}.
\end{aligned}    
\end{equation}
Again, we call on the same tools to estimate the two remaining terms on the right of \eqref{identity_2}. So we can show that 
\begin{equation}
\begin{aligned}
\intO r_t \Delta p p_{tt} \dx &\leq \Big\Vert \frac{r_t}{b} \Big\Vert_{L^3} \Vert p_{tt} \Vert_{L^6}  \Vert \sqrt{b}\Delta p \Vert_{L^2}\\
&\leq C(\varepsilon) \Vert r_t \Vert_{L^3}^2 \Vert \sqrt{b} \Delta p \Vert_{L^2}^2+\varepsilon  \Vert \nabla p_{tt} \Vert_{L^2}^2.\\ 
\end{aligned}
\end{equation} 
In addition, we have 
\begin{equation}
\intO g_t p_{tt} \dx \leq C(\varepsilon) \Vert g_t \Vert^2_{L^2}+\varepsilon \Vert \nabla p_{tt} \Vert_{L^2}^2. 
\end{equation} 
%\begin{equation}
%<g_t ,p_{tt}>_{H^{-1}, H^1} \leq C(\varepsilon) \Vert g_t \Vert^2_{H^{-1}}+\varepsilon \Vert \nabla p_{tt} \Vert_{L^2}^2. 
%\end{equation}
Collecting the above estimates and selecting  $\varepsilon$ small enough we obtain for   all $t\geq 0$
\begin{equation} \label{estimate_E_p_2}
\begin{aligned}
&\frac{1}{2} \ddt \Big(\Vert  \sqrt{\alpha(t)} p_{t t}(t) \Vert^2_{L^2}+  \Vert \sqrt{r(t)} \nabla p_t(t) \Vert^2_{L^2}\Big)+ \frac{b}{2}\Vert \nabla p_{tt}(t) \Vert^2_{L^2}\\
\lesssim &\, \Big(\Vert  r_t(t) \Vert_{L^3}^2+\Vert r_t(t) \Vert_{L^2}^{\frac{4}{4-d}}+\Vert \alpha_t(t) \Vert_{L^2}^{\frac{4}{4-d}}+\Vert \nabla r(t) \Vert^2_{L^3}\Big) E_2[p](t)\\ 
&+\varepsilon  \Big\Vert \frac{1}{r} \Big\Vert_{L^\infty}^{\frac{4}{d}-1} \Vert \Delta p_t(t)\Vert^2_{L^2}
+\Vert g_t(t) \Vert^2_{L^2}.\\
\end{aligned}
\end{equation}

Now, we focus on establishing estimates for $\Vert \Delta p \Vert_{L^2}$ and $\Vert \Delta p_t \Vert_{L^2}$. Testing the equation \eqref{linear_West} by $-\Delta p$ and integrating over $\Om$ gives
\begin{equation}
 \frac{b}{2} \ddt \Vert \Delta p \Vert^2_{L^2}+\intO r \vert \Delta p \vert^2 \dx=-\intO g (\Delta p)  \dx+ \intO \alpha p_{tt} \Delta p \dx. 
\end{equation}
Applying H\"{o}lder, Young and Poincar\'{e} inequalities, we get the following bound
\begin{equation}\label{eq35}
\begin{aligned}
&\frac{b}{2} \ddt \Vert \Delta p \Vert^2_{L^2}+\Vert \sqrt{r} \Delta p \Vert_{L^2}^2 \\
\leq &\, 2 \Vert g \Vert^2_{L^2} +  \frac{2}{b} \Vert \sqrt{b} \Delta p \Vert_{L^2}^2+C (\varepsilon) \Vert \sqrt{b} \Delta p \Vert_{L^2}^2+ \varepsilon \Vert \alpha  \Vert_{L^\infty}^2 \Vert  \nabla p_{tt} \Vert_{L^2}^2.
\end{aligned}
\end{equation}
%\begin{equation} 
%\frac{b}{2} \ddt \Vert \Delta p \Vert^2_{L^2} +\Vert \sqrt{r} \Delta p \Vert_{L^2}^2
%\lesssim  \Vert \sqrt{b} \Delta p \Vert_{L^2}^2+\Vert g \Vert^2_{L^2}+ \varepsilon \Vert \alpha  \Vert_{L^\infty}^2 \Vert  \nabla p_{tt} \Vert_{L^2}^2 .
%\end{equation}
Adding up \eqref{eq35} and \eqref{estimate_E_p_2}, now the last term on the right of \eqref{eq35} can be absorbed by the left-hand side of \eqref{estimate_E_p_2}. Hence, we get 
\begin{equation} \label{estimate_E_p_2*}
\begin{aligned}
& \ddt E_2[p](t)+\frac{b}{3} \Vert \nabla p_{tt}(t) \Vert^2_{L^2}+\Vert \sqrt{r(t)} \Delta p(t) \Vert_{L^2}^2\\
\lesssim &\, (1+\Vert  r_t(t) \Vert_{L^3}^2  +\Vert r_t(t) \Vert_{L^2}^{\frac{4}{4-d}}+\Vert \alpha_t(t) \Vert_{L^2}^{\frac{4}{4-d}}+\Vert \nabla r(t) \Vert^2_{L^3}) E_2[p](t)\\
& +\varepsilon   \Big\Vert \frac{1}{r} \Big\Vert_{L^\infty}^{\frac{4}{d}-1} \Vert \Delta p_t(t)\Vert^2_{L^2}
+\Vert g(t) \Vert^2_{L^2}+\Vert g_t(t) \Vert^2_{L^2}.\\
\end{aligned}
\end{equation} 
Next, we multiply the equation \eqref{linear_West} by $-\Delta p_t$ and we integrate over $\Om$ to get 
\begin{equation}
b \Vert \Delta p_t \Vert^2_{L^2}=\intO \alpha p_{tt} \Delta p_t \dx-\intO r \Delta p \Delta p_t \dx +\intO g (-\Delta p_t)  \dx. 
\end{equation}
Using H\"{o}lder and Young inequalities and recalling  Assumption \ref{Assumption_1}, we obtain
\begin{equation}
\begin{aligned}
b \Vert \Delta p_t \Vert^2_{L^2}  \leq &\, C(\varepsilon) \Big(\Vert \sqrt{\alpha} p_{tt} \Vert^2_{L^2}+ \Vert r \Vert_{L^\infty}^2 \Vert \sqrt{b} \Delta p \Vert^2_{L^2} + \Vert g \Vert^2_{L^2}\Big)+ \varepsilon \Vert \Delta p_t \Vert^2_{L^2}.
\end{aligned}
\end{equation}
Taking  $\varepsilon$ small enough, we find 
\begin{equation} \label{Delta_p_t}
\frac{b}{2} \Vert \Delta p_t(t) \Vert^2_{L^2} \lesssim  E_2[p](t)+ \Vert g (t) \Vert^2_{L^2}, \quad t \geq 0.
\end{equation} 
Thus, putting \eqref{estimate_E_p_2*} and \eqref{Delta_p_t} together,  we have
\begin{equation} \label{estimate_E_p_2**}
\begin{aligned}
&\ddt E_2[p](t)+\frac{b}{3} \Vert \nabla p_{tt}(t) \Vert^2_{L^2}+\Vert \sqrt{r(t)} \Delta p(t) \Vert_{L^2}^2+\frac{b}{3} \Vert \Delta p_t(t) \Vert^2_{L^2}\\
\lesssim &\, \Big(1+\Vert  r_t(t) \Vert_{L^3}^2  +\Vert r_t(t) \Vert_{L^2}^{\frac{4}{4-d}}+\Vert \alpha_t(t) \Vert_{L^2}^{\frac{4}{4-d}}+\Vert \nabla r(t) \Vert_{L^3}^2 \Big) E_2[p](t)\\
&+\Vert g(t) \Vert^2_{L^2}+\Vert g_t(t) \Vert^2_{L^2}.\\
\end{aligned}
\end{equation} 
Finally, summing up the estimates \eqref{estimate_E_p_1},  \eqref{estimate_E_p_2**} and selecting  $\varepsilon$ small enough in the last term on the right-hand side  of \eqref{estimate_E_p_1},  we obtain  \eqref{estimate_E_1_2_p}. This completes the proof of Lemma \ref{lemma2}.
\end{proof}
%%%%%%%%%%%%%%%%%%%%%%%%%%%%%%%%%%%%%%%%%%%%%%%%%%%%%%%%%%%%%%%%%%%%%%%%%%%%%%%%%%%%%%%%%%%%%%%%%%%%%%%%%%%%%%%%%%%%%%%%%%%%%%%%%%%%%%%%%%%%%%%%%%%
\begin{lemma}
Let $g \in H^1_0(\Om)$ and $g_t \in L^2(\Om)$. Then, under Assumption \ref{Assumption_1},   the energy $E_3$ satisfies for $t \geq 0$ the estimate
\begin{equation} \label{estimate_E_3_p}
\begin{aligned}
&\ddt E_3[p](t)+\Vert \sqrt{r} \nabla \Delta p(t) \Vert^2_{L^2} +\Vert \sqrt{\alpha} p_{ttt}(t) \Vert^2_{L^2}\\
\lesssim &\, (1+\Lambda_2(t))E_3[p](t)+ \Vert \sqrt{b} \Delta p_t(t) \Vert_{L^2}^2+\varepsilon  \Vert \nabla p_{tt}(t) \Vert_{L^2}^2\\
&+\Vert \nabla g(t)\Vert^2_{L^2}+\Vert g_t(t) \Vert_{L^2}^2\\
\end{aligned}
\end{equation}    
where
\begin{equation}
\Lambda_2(t)=\Vert r_t(t) \Vert_{L^3}^2+\Vert \alpha_t(t) \Vert_{L^3}^2+\Vert \nabla \alpha(t) \Vert_{L^3}^2+\Vert \nabla r(t) \Vert_{L^3}^2.
\end{equation}  
\end{lemma} 
\begin{proof}
We test the linearized Westervelt equation \eqref{linear_West} by $\Delta^2 p$ and we integrate over $\Om$. Using integration by parts, we obtain   
%\begin{equation}
%\intO (\alpha p_{tt}-r \Delta p-b \Delta p_t)\Delta^2 p \dx=\intO g (\Delta^2 p) \dx
%\end{equation}
%which, after integration by parts, implies
\begin{equation} \label{eq18}
\begin{aligned}
&\intO (r \vert\nabla \Delta p \vert^2+b \nabla \Delta p_t \cdot \nabla \Delta p )\dx\\
=&\,\intO (-\alpha \nabla p_{tt} -p_{tt} \nabla \alpha -\Delta p \nabla r-\nabla g) \cdot \nabla \Delta p \dx.
\end{aligned}
\end{equation}
Using H\"{o}lder's inequality, we find  
\begin{equation}
\begin{aligned}
&\frac{1}{2}\ddt  \Vert \sqrt{b}\nabla \Delta p \Vert_{L^2}^2+\Vert \sqrt{r} \nabla \Delta p \Vert^2_{L^2} \\
\leq \Vert \alpha & \Vert_{L^\infty} \Vert \nabla p_{tt} \Vert_{L^2} \Vert \nabla \Delta p \Vert_{L^2}+\Vert \nabla \alpha \Vert_{L^3} \Vert p_{tt} \Vert_{L^6} \Vert \nabla \Delta p \Vert_{L^2}\\
&+\Vert \nabla r \Vert_{L^3} \Vert \Delta p \Vert_{L^6} \Vert \nabla \Delta p \Vert_{L^2}+ \Vert \nabla g \Vert_{L^2} \Vert \nabla \Delta p \Vert_{L^2}.\\
\end{aligned}
\end{equation}
Moreover, taking advantage of Assumption \ref{Assumption_1}, the embedding $H^1(\Om) \hookrightarrow L^6(\Om)$ and applying Young's inequality, we find 
\begin{equation}
\begin{aligned}
&\frac{1}{2}\ddt  \Vert \sqrt{b}\nabla \Delta p \Vert_{L^2}^2+\Vert \sqrt{r} \nabla \Delta p \Vert^2_{L^2} \\
 \leq&\,  C(\varepsilon)\Big(1+\Vert \nabla \alpha \Vert_{L^3}^2 \Big) \Vert \sqrt{b} \nabla \Delta p \Vert^2_{L^2}+ \varepsilon \Vert \nabla p_{tt} \Vert_{L^2}^2\\
 &+C(\varepsilon)\Big(\Vert \nabla r \Vert_{L^3}^2 \Vert \sqrt{b} \nabla \Delta p \Vert^2_{L^2} +\Vert \nabla g \Vert^2_{L^2}\Big)+\varepsilon \Big\Vert \frac{1}{\sqrt{r}} \Big\Vert_{L^\infty}^2 \Vert \sqrt{r}\nabla \Delta p \Vert_{L^2}^2.\\
\end{aligned}
\end{equation}
%{\color{red}
%The term $\varepsilon \Big\Vert \frac{1}{r} \Big\Vert_{L^\infty}^2 \Vert \sqrt{r}   \nabla \Delta p \Vert^2_{L^2}$ above is repeated two times. I have also collected the terms involving $\Vert \sqrt{b} \nabla \Delta p \Vert^2_{L^2}$, please check.\\
%}
%{\color{red} Just use $\varepsilon$ above. Also, do we have $\Delta p=0$ on $\partial\Omega$? In addition, above you used $1/r$ but here $r^{-1}$. Just use $1/r$. You can also have $\Vert \nabla r \Vert_{L^3} \Vert \sqrt{b} \nabla \Delta p \Vert^2_{L^2}$ instead of $\Vert \nabla r \Vert_{L^3}^2 \Vert \sqrt{b} \nabla \Delta p \Vert^2_{L^2}$, no?}\\
%{\color{blue} we are deriving these estimates for the smooth eigenfunctions of the Dirichlet laplacian, so from the fact that $\Delta \psi =\lambda \psi$ I think, we have $\Delta \psi=0$ on the boundary. I don't understand how we can replace $\Vert \nabla r \Vert_{L^3}^2 \Vert \sqrt{b} \nabla \Delta p \Vert^2_{L^2}$ by  $\Vert \nabla r \Vert_{L^3} \Vert \sqrt{b} \nabla \Delta p \Vert^2_{L^2}$. We got $\Vert \nabla r \Vert_{L^3}$ squared after using Young inequality, and we can't take it with $\varepsilon$-term since it's not $L^\infty$ with respect to time.}\\
We can fix $\varepsilon$ as small as needed  to absorb the last term on the right-hand side  into the dissipative term on the left-hand side. Hence, we have    
\begin{equation} \label{estimate_nabla_Delta_p}
\begin{aligned}
&\frac{1}{2}\ddt  \Vert \sqrt{b}\nabla \Delta p \Vert_{L^2}^2+\Vert \sqrt{r} \nabla \Delta p \Vert^2_{L^2} \\
\lesssim &\,(1+\Vert \nabla \alpha \Vert_{L^3}^2+\Vert \nabla r \Vert_{L^3}^2) \Vert \sqrt{b} \nabla \Delta p \Vert^2_{L^2}+\Vert \nabla g \Vert^2_{L^2}
+\varepsilon  \Vert \nabla p_{tt} \Vert_{L^2}^2.\\
\end{aligned}
\end{equation} 
Next, we multiply the time-differentiated equation \eqref{linear_West_t} by $p_{ttt}$ and we integrate over $\Om$ to find 
\begin{equation} \label{identity_p_ttt}
    \frac{b}{2} \ddt \Vert \nabla p_{tt} \Vert^2_{L^2}+\Vert \sqrt{\alpha} p_{ttt} \Vert^2_{L^2}=\intO (-\alpha_t p_{tt} +r_t \Delta p +r \Delta p_t+g_t) p_{ttt} \dx.
\end{equation}
Our goal now is to estimate the terms of the right-hand side of \eqref{identity_p_ttt}.\\
 First, we have by using H\"{o}lder's inequality  
\begin{equation}
\begin{aligned}
&\frac{b}{2} \ddt \Vert \nabla p_{tt} \Vert^2_{L^2}+\Vert \sqrt{\alpha} p_{ttt} \Vert^2_{L^2} \\
\leq &\,\Vert \alpha_t \Vert_{L^3} \Vert p_{tt} \Vert_{L^6} \Vert p_{ttt} \Vert_{L^2}+\Vert r_t \Vert_{L^3} \Vert \Delta p \Vert_{L^6} \Vert p_{ttt} \Vert_{L^2}\\
&+\Vert r \Vert_{L^\infty} \Vert \Delta p_t \Vert_{L^2} \Vert p_{ttt} \Vert_{L^2}+\Vert g_t \Vert_{L^2} \Vert p_{ttt} \Vert_{L^2}.\\    
\end{aligned}
\end{equation}
Furthermore, making use of Young's inequality together with the embedding $H^1(\Omega)\hookrightarrow L^6(\Omega)$, we obtain 
\begin{equation}
\begin{aligned}
&\frac{b}{2} \ddt \Vert \nabla p_{tt} \Vert^2_{L^2}+\Vert \sqrt{\alpha} p_{ttt} \Vert^2_{L^2} \\
\leq C&(\varepsilon) \Big(\Vert \alpha_t \Vert_{L^3}^2 \Vert \sqrt{b} \nabla p_{tt} \Vert_{L^2}^2+\Vert r_t \Vert_{L^3}^2 \Vert \sqrt{b} \nabla \Delta p \Vert_{L^2}^2+\Vert \sqrt{b} \Delta p_t \Vert_{L^2}^2+\Vert g_t \Vert_{L^2}^2\Big)\\
&+ \varepsilon \Big\Vert \frac{1}{\sqrt{\alpha}} \Big\Vert_{L^\infty}^2 \Vert \sqrt{\alpha} p_{ttt} \Vert_{L^2}^2.  
\end{aligned}
\end{equation}
By selecting  $\varepsilon$ small enough, we find  
\begin{equation} \label{eq19}
\begin{aligned}
&\frac{1}{2} \ddt \Vert \sqrt{b}\nabla p_{tt} \Vert^2_{L^2}
 +\Vert \sqrt{\alpha} p_{ttt} \Vert^2_{L^2}\\
\lesssim&\,  \Vert \alpha_t \Vert_{L^3}^2\Vert \sqrt{b}  \nabla p_{tt} \Vert_{L^2}^2+\Vert r_t \Vert_{L^3}^2\Vert \sqrt{b} \nabla \Delta p \Vert_{L^2}^2+\Vert \sqrt{b} \Delta p_t \Vert_{L^2}^2+\Vert g_t \Vert_{L^2}^2.
\end{aligned}
\end{equation}
Consequently, collecting the estimates \eqref{estimate_nabla_Delta_p} and \eqref{eq19}, we arrive at \eqref{estimate_E_3_p}.
\end{proof}

\begin{proof}[Proof of the Proposition \ref{prop2}] Collecting  the estimates \eqref{estimate_E_1_2_p} and \eqref{estimate_E_3_p} and employing \linebreak
Poincar\'{e}'s inequality to find for all $t \geq 0$
\begin{equation}
\begin{aligned}
&\ddt \mathfrak{E}[p](t)+\mathfrak{D}_0[p](t) \\
\lesssim &\,(1+\Lambda_1(t)+\Lambda_2(t))\mathfrak{E}[p](t)+\Vert \sqrt{b}\Delta p_t(t) \Vert^2_{L^2}+\varepsilon  \Vert \nabla p_{tt}(t) \Vert_{L^2}^2\\
&+ \Vert \nabla g(t) \Vert^2_{L^2}+\Vert g_t(t) \Vert^2_{L^2},
\end{aligned}
\end{equation}
where $\mathfrak{E}[p]$ and $\mathfrak{D}_0[p]$ are  defined in \eqref{Energy_Tot} and \eqref{dissipation_D_0} respectively.  So now the terms depending on $\varepsilon$ can be absorbed into $\mathfrak{D}_0[p]$ for small values of $\varepsilon$. In addition, recalling \eqref{Delta_p_t}
%\begin{equation}
%b \Vert \Delta p_t(t) \Vert^2_{L^2} \lesssim E_2[p](t)+ \Vert g (t) \Vert^2_{L^2}.
%\end{equation}
and using Poincar\'{e}'s inequality, we obtain the following estimate for the total acoustic energy 
\begin{equation} \label{total_energy_p_t}
\begin{aligned}
&\ddt \mathfrak{E}[p](t)+\mathfrak{D}_0[p](t) \lesssim (1+\Lambda(t))\mathfrak{E}[p](t)+ \mathfrak{F}(t),
\end{aligned}
\end{equation}
where the functions $\Lambda$ and $\mathfrak{F}$ are given in \eqref{lambda}  and \eqref{source_term}, respectively. 
%\begin{equation} 
%\begin{aligned}    
%\mathfrak{E}[p](t)+\int_0^t \mathfrak{D}_0[p](s) \ds 
%\lesssim&\, \mathfrak{E}[p](0) \exp \Big(\int_0^t (1+ \Lambda(s)) \ds \Big)\\
%&+\int_0^t \mathfrak{F}(s) \exp \Big( \int_s^t (1+\Lambda (\sigma)) \textup{d} \sigma \Big) \ds\\
%\end{aligned}
%\end{equation}

Next, we multiply the  equation \eqref{linear_West_t} by $-\Delta p_{tt}$ and integrate over $\Om$, then we have
%\begin{equation}
%\intO (\alpha p_{ttt}-r \Delta p_t- b \Delta p_{tt})(-\Delta p_{tt}) \dx=\intO (-\alpha_t p_{tt}+r_t \Delta p +g_t)(-\Delta p_{tt}) \dx
%\end{equation}
\begin{equation}
\begin{aligned}
b\Vert \Delta p_{tt} \Vert^2_{L^2}=& \intO (\alpha p_{ttt} + \alpha_t p_{tt}-  r_t \Delta p- r\Delta p_t- g_t) \Delta p_{tt} \dx.\\
\end{aligned}
\end{equation}
Using H\"{o}lder's inequality, it follows that
\begin{equation}
\begin{aligned}
b\Vert \Delta p_{tt} \Vert^2_{L^2} \leq \Vert \sqrt{\alpha} \Vert_{L^\infty}  & \Vert \sqrt{\alpha} p_{ttt} \Vert_{L^2} \Vert \Delta p_{tt} \Vert_{L^2}+\Vert \alpha_t \Vert_{L^3} \Vert p_{tt} \Vert_{L^6} \Vert \Delta p_{tt} \Vert_{L^2}\\
+\Vert r_t \Vert_{L^3}& \Vert \Delta p \Vert_{L^6} \Vert \Delta p_{tt} \Vert_{L^2}+\Vert r\Vert_{L^\infty} \Vert \Delta p_t \Vert_{L^2}\Vert \Delta p_{tt} \Vert_{L^2}\\
&+\Vert g_t \Vert_{L^2} \Vert \Delta p_{tt} \Vert_{L^2}.\\
\end{aligned}
\end{equation}
Applying Young's inequality, making use of the  continuous embedding $H^1(\Om) \hookrightarrow L^6(\Om)$, we find
\begin{equation}
\begin{aligned}
b\Vert \Delta p_{tt} \Vert^2_{L^2} \lesssim &\,  C(\varepsilon)( \Vert \sqrt{\alpha} \Vert_{L^\infty}^2\Vert \sqrt{\alpha} p_{ttt} \Vert_{L^2}^2 + \Vert \alpha_t \Vert_{L^3}^2 \Vert \sqrt{b} \nabla p_{tt} \Vert_{L^2}^2+\Vert r_t \Vert_{L^3}^2 \Vert \sqrt{b} \nabla \Delta p \Vert_{L^2}^2\\
&+ \Vert r \Vert_{L^\infty}^2\Vert \sqrt{b} \Delta  p_{t} \Vert_{L^2}^2+\Vert g_t \Vert_{L^2}^2) +\varepsilon \Vert \Delta p_{tt} \Vert_{L^2}^2.\\
\end{aligned}
\end{equation}
Fixing  $\varepsilon >0$ small enough and  keeping in mind \eqref{Delta_p_t} and Assumption \ref{Assumption_1}, we get for all $t \geq 0$
\begin{equation}\label{Delta_p_tt}
\begin{aligned}
\frac{b}{2}\Vert \Delta p_{tt}(t) \Vert^2 \lesssim &\, \Vert \sqrt{\alpha} p_{ttt}(t) \Vert_{L^2}^2+ ( \Vert \alpha_t(t) \Vert_{L^3}^2+\Vert r_t(t) \Vert_{L^3}^2) E_3[p](t)\\
&+ E_2[p](t)+\Vert g(t) \Vert_{L^2}^2+\Vert g_t(t) \Vert_{L^2}^2.\\
\end{aligned}
\end{equation}
We sum up $\gamma\times \eqref{Delta_p_tt}$ and the estimate \eqref{total_energy_p_t}, then we take $\gamma>0$ suitably small in order to absorb the first term on the right of \eqref{Delta_p_tt} by the dissipation $\mathfrak{D}_0[p]$ in \eqref{total_energy_p_t}, then, it follows that  
\begin{equation}\label{eq23}
\begin{aligned}   
\ddt \mathfrak{E}[p](t)+\mathfrak{D}_0[p](t)+ b \Vert \Delta p_{tt}(t) \Vert^2 \lesssim (1+ \Lambda(t))\mathfrak{E}[p](t)+\mathfrak{F}(t).
\end{aligned}
\end{equation}
On the other hand, by multiplying the equation \eqref{linear_West} by $\Delta^2 p_t$ and integrating over $\Omega$, we obtain
\begin{equation}
\begin{aligned}
& b  \Vert \nabla \Delta p_t \Vert^2_{L^2}=\intO (\alpha \nabla p_{tt}+p_{tt} \nabla \alpha-\Delta p \nabla r-r\nabla \Delta p-\nabla g)\cdot \nabla \Delta p_t \dx.\\
\end{aligned}
\end{equation}
Applying H\"{o}lder's inequality, we find 
\begin{equation}
\begin{aligned}
b\Vert \nabla \Delta p_t \Vert^2_{L^2} \leq \Vert &  \alpha \Vert_{L^\infty}  \Vert \nabla p_{tt} \Vert_{L^2} \Vert \nabla \Delta p_t \Vert_{L^2}+\Vert \nabla \alpha \Vert_{L^3} \Vert p_{tt} \Vert_{L^6} \Vert \nabla \Delta p_t \Vert_{L^2}\\
&+\Vert r\Vert_{L^\infty}\Vert \nabla \Delta p \Vert_{L^2}\Vert \nabla \Delta p_t \Vert_{L^2}+\Vert \nabla r \Vert_{L^3} \Vert \Delta p \Vert_{L^6} \Vert \nabla \Delta p_t \Vert_{L^2}\\
&+ \Vert \nabla g \Vert_{L^2} \Vert \nabla \Delta p_t \Vert_{L^2}.\\
\end{aligned}
\end{equation}
Furthermore, the embedding $H^1(\Om) \hookrightarrow L^6(\Om)$ together with Young's inequality and the fact that the Dirichlet--Laplacian eigenfunctions satisfy $\Delta p|_{\partial \Om}=0$ yield
\begin{equation}
\begin{aligned}
b \Vert \nabla \Delta p_t \Vert^2_{L^2} \lesssim  &\,  C(\varepsilon)( \Vert \alpha \Vert_{L^\infty}^2\Vert \sqrt{b} \nabla p_{tt} \Vert_{L^2}^2 + \Vert \nabla \alpha \Vert_{L^3}^2 \Vert \sqrt{b} \nabla p_{tt} \Vert_{L^2}^2+ \Vert r\Vert_{L^\infty}^2\Vert \sqrt{b} \nabla \Delta p \Vert^2_{L^2}\\
&+ \Vert \nabla r \Vert_{L^3}^2 \Vert \sqrt{b} \nabla \Delta p \Vert^2_{L^2}+ \Vert \nabla g \Vert^2_{L^2})+ \varepsilon \Vert \nabla \Delta p_t \Vert^2_{L^2}.\\
\end{aligned}
\end{equation}
Hence, selectiong   $\varepsilon$ as small as needed and taking into  account Assumption \ref{Assumption_1}, we infer  that
\begin{equation}\label{eq21}
\begin{aligned}
\frac{b}{2} \Vert \nabla \Delta p_t(t) \Vert^2_{L^2} \lesssim  &\, (1 + \Vert \nabla \alpha \Vert_{L^3}^2 +\Vert \nabla r \Vert_{L^3}^2 )E_3[p](t) +\Vert \nabla g(t) \Vert^2_{L^2}, \quad t \geq 0.\\
\end{aligned}
\end{equation}
Adding up the estimates \eqref{eq23} and  \eqref{eq21} gives
\begin{equation} 
\begin{aligned}    
\ddt \mathfrak{E}[p](t)+\mathfrak{D}[p](t) \lesssim (1+ \Lambda(t))\mathfrak{E}[p](t)+\mathfrak{F}(t).
\end{aligned}
\end{equation}
Consequently, by applying Gronwall's inequality cf. Lemma \ref{gronwall}, we infer that
\begin{equation}\label{eq22} 
\begin{aligned}    
\mathfrak{E}[p](t)+\int_0^t \mathfrak{D}[p](s) \ds 
\lesssim&\, \mathfrak{E}[p](0) \exp \Big(\int_0^t (1+ \Lambda(s)) \ds \Big)\\  
&+\int_0^t \mathfrak{F}(s) \exp \Big( \int_s^t (1+\Lambda (\sigma)) \textup{d} \sigma \Big) \ds.\\
\end{aligned}
\end{equation}
Furthermore, the estimates \eqref{eq22} and \eqref{Delta_p_t} yield  
\begin{equation} \label{estimate_Delta_p_t_inf}
\begin{aligned}
b\Vert \Delta p_t(t) \Vert^2_{L^2}  \lesssim &\, \mathfrak{E}[p](t)+\Vert g(t) \Vert^2_{L^2}\\
\lesssim &\,\mathfrak{E}[p](0) \exp \Big(\int_0^t (1+ \Lambda(s)) \ds \Big)\\
&+\int_0^t \mathfrak{F}(s) \exp \Big( \int_s^t (1+\Lambda (\sigma)) \textup{d} \sigma \Big) \ds.
\end{aligned}
\end{equation}
where we employed the inequality \eqref{1D_embedding} to bound from above the term $\Vert g(t) \Vert^2_{L^2}$. Then, it suffices to sum up \eqref{eq22} and \eqref{estimate_Delta_p_t_inf} in order to reach the estimate \eqref{total_energy_p} for the smooth eigenfunctions of the Dirichlet-Laplacian. Standard compactness method  allows passing to the limit; thus proving the existence of a solution $p \in X_p$  to \eqref{linear_West}. From the weak and weak-$\star$ lower semi-continuity of norms, we get that $p$ satisfies the same energy bound \eqref{total_energy_p}. The uniqueness of the solution $p$ is guaranteed, since the only solution to the homogeneous problem
\begin{equation}
\alpha(x,t) p_{tt}-r(x,t) \Delta p-b \Delta p_t=0, \quad p(x,0)=p_t(x,0)=0, \quad p|_{\partial \Om}=0
\end{equation}
is zero. Indeed, from \eqref{total_energy_p}, we have  $\mathfrak{E}[p](t)=0$ which immediately gives $p=0$. We conclude the proof by noting that $p$ belonging to $X_p$ implies that (see \cite[Chapter  1, Lemma 2.1]{L69} and \cite[Chapter 5, Theorem 2]{evans2010partial}) 
$$p \in C(0, T; H^3(\Om) \cap H^1_0(\Om)), \qquad p_t \in C(0, T; H^2(\Om) \cap H^1_0(\Om)).$$
This concludes the proof of Proposition \ref{prop2}.     
\end{proof}    
%%%%%%%%%%%%%%%%%%%%%%%%%%%%%%%%%%%%%%%%%%%%%%%%%%%%%%%%%%%%%%%%%%%%%%%%%%%%%%%
\section{Uniform local well-posedness: Proof of Theorem \ref{wellposedness_thm}} \label{sect6}
In this section we prove Theorem \ref{wellposedness_thm}. 
The proof  is accomplished using Banach's fixed point theorem. For this purpose, we recall \eqref{Functional_Spaces} and  define the ball $B$ as 
\begin{equation}
\begin{aligned}
B=\Big\{& (p^\ast, \Theta^\ast, q^\ast) \in  \mathcal{X}:  (p^\ast(0), p^\ast_t(0), \Theta^\ast(0), q^\ast(0))=(p_0, p_1, \Theta_0, q_0) \\
& \quad\Vert p^\ast \Vert_{L^\infty L^\infty} \leq \gamma < \frac{1}{2k_1}, \quad \Vert p^\ast \Vert_{X_p} \leq R_1,\quad  \Vert (\Theta^\ast, q^\ast) \Vert_{ X_\Theta \times X_q} \leq R_2\Big\},   
\end{aligned}
\end{equation}
where $\mathcal{X}:=X_p \times X_\Theta \times X_q$. Note that we have two different radii. The reason behind this will be made clear in the proof below, as we will need to impose a smallness condition on $R_1>0$ but not on $R_2>0$.

To check that the ball $B$ is a non-empty subset of $\mathcal{X}$, one can take for instance (see \cite{Nikolic_2022}) $\alpha=r=1$ and $f=g=0$ in \eqref{linear_West} and \eqref{eq_Cattaneo}. Then, based on the estimates \eqref{total_energy_p}, \eqref{Pro_1_Estimate} and \eqref{q_estimate}, the solution $(p,\Theta, q) \in \mathcal{X}$ lies in $B$ if we choose $R_1, R_2$ and $\gamma$ such that 
\begin{equation}
C_T\mathfrak{E}[p](0) \leq R_1^2 \leq \gamma^2 < \frac{1}{(2 k_1)^2}, \quad C_T \big(\Vert q_0 \Vert_{H^1}^2+(1+\bar{\tau}+\bar{\tau}^2)E^{\bar{\tau}}[\Theta, q](0) \big) \leq R_2^2 .
\end{equation} 

The solution spaces $X_p$,  $X_\Theta$ and $X_q$ (defined in \eqref{Functional_Spaces}) are   endowed with the norms
\begin{equation} \label{norms}
\begin{aligned}
\Vert p \Vert_{X_p}:=&\,\Vert p\Vert_{L^\infty H^3}+ \Vert p_t\Vert_{L^\infty H^2}+\Vert \nabla \Delta p_t \Vert_{L^2 L^2}+\Vert\nabla  p_{tt} \Vert_{L^\infty L^2}\\
&+\Vert \Delta p_{tt} \Vert_{L^2 L^2}+\Vert p_{ttt}\Vert_{L^2 L^2},\\
\Vert \Theta \Vert_{X_\Theta}:=&\,\Vert \Theta \Vert_{L^\infty H^2}+\Vert  \Theta_t \Vert_{L^\infty H^1}+\Vert \Theta_{tt}\Vert_{L^\infty L^2},\\
\Vert q \Vert_{X_q}:=&\, \Vert q \Vert_{L^\infty H^1}+\sum_{k=1}^2  \Vert \partial_t^k q \Vert_{L^2L^2}.\\
\end{aligned}
\end{equation}  
Then, the product space $\mathcal{X}$ is equipped with the norm
$$\Vert(p, \Theta, q) \Vert_{\mathcal{X}}^2=\Vert p \Vert_{X_p}^2+\Vert \Theta \Vert_{X_\Theta}^2+\Vert q \Vert_{X_q}^2.$$
These norms are clearly equivalent to the energies $\mathcal{E}[\Theta]$ and $ \mathfrak{E}[p]$. We have for all $t \geq 0$
\begin{equation}    
%\begin{aligned}
\sup_{t \in (0, T)} \mathcal{E} [\Theta](t)+\Vert q \Vert_{L^\infty H^1}^2+ \sum_{k=1}^2 \int_0^T \Vert \partial_t^k q \Vert_{L^2}^2 \lesssim \Vert (\Theta,q) \Vert^2_{X_\Theta \times X_q}
%\end{aligned}
\end{equation}
and 
\begin{equation}
\begin{aligned}
 \Vert ( \Theta,q) \Vert^2_{X_\Theta \times X_q} \lesssim \sup_{t \in (0, T)} \mathcal{E} [\Theta](t)+\Vert q \Vert_{L^\infty H^1}^2+\sum_{k=1}^2 \int_0^T \Vert \partial_t^k q \Vert_{L^2}^2.\\
\end{aligned}
\end{equation}
where $\Vert ( \Theta,q) \Vert^2_{X_\Theta \times X_q}=\Vert \Theta \Vert_{X_\Theta}^2+\Vert q \Vert_{X_q}^2$.
Furthermore, taking into account Assumption \ref{Assumption_1}, especially the boundedness of the functions $\alpha$ and $r$, we obtain
\begin{equation}\label{norm_energy}
\begin{aligned}
\Vert p \Vert_{X_p}^2
 \lesssim &\, \sup_{t \in (0,T)} \mathfrak{E}[p](t)+\sup_{t \in (0,T)} b\Vert \Delta p_t(t) \Vert_{L^2}^2+\int_0^T \Vert \sqrt{\alpha} p_{ttt} \Vert^2_{L^2} \ds\\
 &+\int_0^T b \Vert \nabla \Delta  p_{t} \Vert^2_{L^2} \ds+ \int_0^T b \Vert \Delta p_{tt} \Vert^2_{L^2} \ds\\
 \lesssim &\, \sup_{t \in (0,T)} \mathfrak{E}[p](t)+\sup_{t \in (0,T)} b\Vert \Delta p_t(t) \Vert_{L^2}+ \int_0^T \mathfrak{D}[p](s) \ds. 
\end{aligned}
\end{equation}
The inverse inequality also holds
\begin{equation}
\begin{aligned}
 \sup_{t \in (0,T)} \mathfrak{E}[p](t)+\sup_{t \in (0,T)} b\Vert \Delta p_t(t) \Vert_{L^2}^2+\int_0^T \Vert \sqrt{\alpha} p_{ttt} \Vert^2_{L^2} &\,\ds\\
+\int_0^T b \Vert \nabla \Delta  p_{t} \Vert^2_{L^2} \ds+ \int_0^T b \Vert \Delta p_{tt} \Vert^2_{L^2} \ds &\,\lesssim \Vert p \Vert_{X_p}^2.
\end{aligned} 
\end{equation}
Notice that the norm $\Vert(p, \Theta, q) \Vert_{\mathcal{X}}$ is independent of $\tau$. This plays a crucial role in the study of the limit $\tau\rightarrow 0$.

We consider the operator $\mathcal{T}$ that maps $(p^\ast, \Theta^\ast, q^\ast) \in B \subset \mathcal{X}$ to $(p, \Theta, q) \in \mathcal{X}$ the solution of the coupled problem 
\begin{equation} \label{fixed_point_sys}
\left\{ 
\begin{aligned}
&(1-2k (\Theta^\ast)p^\ast)p_{tt}-h (\Theta^\ast)\Delta p - b \Delta p_t =2 k (\Theta^\ast)(p^\ast_{t})^2,& \quad &\text{in} &\  &\Omega \times (0,T),\\
& m \Theta_t +\nabla\cdot q + \ell \Theta = \mathcal{Q}(p^\ast_t),& \quad &\text{in} &\ &\Omega \times (0,T),
\\
&\tau q_t+q+\kappaa \nabla \Theta=0, &\quad &\text{in} &\ &\Omega \times (0,T).
\end{aligned}
\right.      
\end{equation}
The existence of a unique solution of system \eqref{Main_system} is equivalent to the existence of a unique fixed point in $B$ to the mapping $\mathcal{T}$, which will be  guaranteed by Banach's fixed-point theorem. Therefore, to ensure applicability of the latter, we want to show that for $R_2$ large enough, $R_1$ small enough  and for   $\delta$ small enough, we have:
\begin{description}
\item[(i)] The mapping $\mathcal{T}:B\rightarrow B$ is well defined.

\item[(ii)] $\mathcal{T}$ is a contraction mapping.
\end{description}
This is achieved by following the lines of the analysis presented in \cite[Section 4]{Nikolic_2022}. The results are contained in the next two lemmas.

%\begin{theorem} \label{wellposedness_thm}
%Let $T>0$ and $\tau \in (0, \bar{\tau}]$. Assume that   
%\begin{equation}
%\begin{aligned}
%(p_0, p_1) & \in H^3(\Om)\cap H^1_0(\Om) \times H^2(\Om) \cap H^1_0(\Om),\\  (\Theta_0, q_0) & \in H^2(\Om) \cap H^1_0(\Om) \times (H^1(\Om))^d.
%\end{aligned}
%\end{equation}
%Then, there  exists $\delta=\delta(T)>0$ such that if  $\mathfrak{E}[p](0) \leq \delta$, 
%then  system \eqref{Main_system} has a unique solution $(p, \Theta, q) \in \mathcal{X}$. 
%\end{theorem}

\begin{lemma} \label{lemma3} Given  $ \tau \in (0, \bar{\tau}]$. Then, for small enough $R_1$ and $\delta$, the operator $\mathcal{T}$ is self-mapping; namely $\mathcal{T}(B) \subset B$.
\end{lemma}
\begin{proof} Given $(p^\ast, \Theta^\ast, q^\ast) \in B$. We are looking to show that $(p, \Theta, q)=\mathcal{T}(p^\ast, \Theta^\ast, q^\ast)$ the solution of \eqref{fixed_point_sys} also lies in $B$. To do this, we write  the system \eqref{fixed_point_sys} in the context of the Propositions \ref{total-energy} and  \ref{prop2}. To this end, we set     
\begin{equation}
\begin{aligned}
    &\alpha(x,t)=1-2 k (\Theta^\ast)p^\ast, \qquad r(x,t)=h (\Theta^\ast), \\
    &g(x,t)=2k (\Theta^\ast)(p^\ast_t)^2,\qquad 
    f(x,t)=\mathcal{Q}(p^\ast_t)=\frac{2b}{\rhoa \Ca^4} (p_t^\ast)^2\\
    \end{aligned}
\end{equation} 
and our goal is to prove that these functions satisfy Assumption \ref{Assumption_1}. We begin by checking the nondegeneracy condition. Indeed, we have
\begin{equation}
\Vert 2 k (\Theta^\ast)p^\ast \Vert_{L^\infty L^\infty} \leq 2 k_1 \Vert p^\ast \Vert_{L^\infty L^\infty} \leq 2k_1 \gamma,
\end{equation}
hence
\begin{equation}
0<\alpha_0=1-2k_1\gamma \leq \alpha(x,t) \leq 1+2k_1\gamma=\alpha_1.
\end{equation}
In addition, from \eqref{bound_h}, we have 
\begin{equation}
0<r_0=h_1 \leq h(\Theta^\ast),
\end{equation}
which ensures that the functions $r, \alpha$ do not degenerate.

Next, we focus on the function $\Lambda(t)$ given in \eqref{lambda}. Using H\"{o}lder's inequality, the properties of $k $ (see the assumption \eqref{properties_k}) and the embedding $H^2(\Om) \hookrightarrow L^\infty(\Om)$, we have
\begin{equation}
\begin{aligned}
    \Vert \alpha_t \Vert_{L^2 L^2}&=\Vert \partial_t(2 k (\Theta^\ast)p^\ast) \Vert_{L^2L^2}\\
    &\leq 2 \Vert  k (\Theta^\ast ) p_t^\ast \Vert_{L^2L^2} +2 \Vert k '(\Theta^\ast) \Theta^\ast_t p^\ast \Vert_{L^2L^2}\\
    & \lesssim \Vert  k (\Theta^\ast ) \Vert_{L^\infty L^\infty} \Vert p_t^\ast \Vert_{L^2L^2}+ \Vert  k '(\Theta^\ast )\Vert_{L^\infty L^\infty} \Vert \Theta^\ast_t \Vert_{L^2L^2} \Vert  p^\ast \Vert_{L^\infty L^\infty}\\
    &\lesssim k_1 \Vert p_t^\ast \Vert_{L^2L^2}+ (1+\Vert  \Theta^\ast \Vert_{L^\infty L^\infty}^{\gamma_2+1}) \Vert \Theta^\ast_t \Vert_{L^2L^2} \Vert  p^\ast \Vert_{L^\infty L^\infty}\\
    &\leq C_T (R_1+(1+R_2^{1+\gamma_2})R_1R_2).\\
    \end{aligned}
\end{equation} 
We obtained the last inequality above thanks to the embedding $L^\infty(0, T) \hookrightarrow L^2(0, T)$ and the fact that $(p^\ast, \Theta^\ast, q^\ast)  \in B$.\\   
Let $\beta=\frac{4}{4-d}$. Then, we find
\begin{equation}
\begin{aligned}
    \Vert \alpha_t \Vert_{ L^{\beta}L^2} & \leq 2 \Vert  k (\Theta^\ast ) p_t^\ast \Vert_{L^\beta L^2} +2 \Vert k '(\Theta^\ast) \Theta^\ast_t p^\ast \Vert_{L^\beta L^2}\\
    & \lesssim \Vert  k (\Theta^\ast ) \Vert_{L^\infty L^\infty} \Vert p_t^\ast \Vert_{L^\beta L^2}+ \Vert  k '(\Theta^\ast )\Vert_{L^\infty L^\infty} \Vert \Theta^\ast_t \Vert_{L^\beta L^2} \Vert  p^\ast \Vert_{L^\infty L^\infty}\\   
    &\lesssim k_1 \Vert p_t^\ast \Vert_{L^\beta L^2}+ (1+\Vert  \Theta^\ast \Vert_{L^\infty L^\infty}^{\gamma_2+1}) \Vert \Theta^\ast_t \Vert_{L^\beta L^2} \Vert  p^\ast \Vert_{L^\infty H^3}\\
    &\leq C_T (R_1 +(1+R_2^{1+\gamma_2})R_1R_2),
    \end{aligned}
\end{equation}
 where we used  the embedding $L^\infty(0,T) \hookrightarrow L^\beta(0,T)$.
Further, using the embedding $H^1(\Om) \hookrightarrow L^3(\Om)$ we can estimate $\Vert \alpha_t \Vert_{L^2 L^3}, \Vert \nabla \alpha \Vert_{L^2 L^3}$ as follows
\begin{equation}\label{alpha_t_L^3}
    \begin{aligned}
    \Vert  \alpha_t \Vert_{L^2 L^3} &\leq 2 \Vert  k (\Theta^\ast ) p_t^\ast \Vert_{L^2L^3} +2 \Vert k '(\Theta^\ast) \Theta^\ast_t p^\ast \Vert_{L^2L^3}\\
    & \lesssim \Vert  k (\Theta^\ast ) \Vert_{L^\infty L^\infty} \Vert p_t^\ast \Vert_{L^2L^3}+ \Vert  k '(\Theta^\ast )\Vert_{L^\infty L^\infty} \Vert \Theta^\ast_t \Vert_{L^2L^3} \Vert  p^\ast \Vert_{L^\infty L^\infty}\\
    &\lesssim k_1 \Vert p_t^\ast \Vert_{L^2H^1}+ (1+\Vert  \Theta^\ast \Vert_{L^\infty L^\infty}^{\gamma_2+1}) \Vert \Theta^\ast_t \Vert_{L^2H^1} \Vert  p^\ast \Vert_{L^\infty H^3},\\
    \end{aligned}
\end{equation}
and 
\begin{equation}\label{nabla_alpha_L^3}
\begin{aligned}
 \Vert \nabla \alpha \Vert_{L^2 L^3} &\leq 2 \Vert  k (\Theta^\ast ) \nabla p^\ast \Vert_{L^2L^3} +2 \Vert k '(\Theta^\ast) \nabla \Theta^\ast p^\ast \Vert_{L^2L^3}\\
  & \lesssim \Vert  k (\Theta^\ast ) \Vert_{L^\infty L^\infty} \Vert \nabla p^\ast \Vert_{L^2L^3}+ \Vert  k '(\Theta^\ast )\Vert_{L^\infty L^\infty} \Vert \nabla \Theta^\ast \Vert_{L^2L^3} \Vert  p^\ast \Vert_{L^\infty L^\infty}\\
  &\lesssim k_1 \Vert p^\ast \Vert_{L^2H^2}+ (1+\Vert  \Theta^\ast \Vert_{L^\infty L^\infty}^{\gamma_2+1}) \Vert \Theta^\ast \Vert_{L^2H^2}  \Vert  p^\ast \Vert_{L^\infty H^3}.\\
\end{aligned}
\end{equation}
Hence, it results from \eqref{alpha_t_L^3} and \eqref{nabla_alpha_L^3}  that 
\begin{equation}
\Vert  \alpha_t \Vert_{L^2 L^3}+\Vert \nabla \alpha \Vert_{L^2 L^3}\leq C_T( R_1+(1+R_2^{1+\gamma_2})R_1R_2). 
\end{equation}
Similarly, we can derive estimates for the terms of $\Lambda$ involving the function $r$. On account of the properties of the function $h $ (see \eqref{h'_assump}) and the embedding $L^\infty(0, T) \hookrightarrow L^\beta(0, T)$, we find
%$ L^\infty(0, T) \hookrightarrow L^{\frac{4}{4-d}}(0,T)$.
\begin{equation}
\begin{aligned}
    \Vert r_t\Vert_{L^\beta L^2}&=\Vert h '(\Theta^\ast) \Theta^\ast_t\Vert_{L^\beta L^2}\leq \Vert h '(\Theta^\ast) \Vert_{L^\infty L^\infty} \Vert \Theta^\ast_t\Vert_{L^\beta L^2}\\
    & \leq  C_T (1+\Vert  \Theta^\ast \Vert_{L^\infty L^\infty}^{\gamma_1+1})\Vert \Theta^\ast_t\Vert_{L^\infty L^2} \leq C_T(1+R_2^{\gamma_1+1})R_2,
    \end{aligned}
\end{equation}
and 
\begin{equation} \label{eq25} 
    \begin{aligned}
\Vert \nabla r \Vert_{L^2L^2}&=\Vert h '(\Theta^\ast) \nabla  \Theta^\ast \Vert_{L^2 L^2} \leq \Vert h '(\Theta^\ast) \Vert_{L^\infty L^\infty} \Vert \nabla \Theta^\ast \Vert_{L^2 L^2}\\
& \lesssim (1+\Vert  \Theta^\ast \Vert_{L^\infty L^\infty}^{\gamma_1+1})\Vert \Theta^\ast \Vert_{L^2 H^2} \leq C_T(1+R_2^{\gamma_1+1})R_2.
    \end{aligned}
\end{equation}
Again, using the embedding $H^1(\Om) \hookrightarrow L^3(\Om)$, we have
\begin{equation}
    \begin{aligned}
    \Vert  r_t \Vert_{L^2 L^3}&=\Vert h '(\Theta^\ast) \Theta^\ast_t\Vert_{L^2 L^3} \lesssim (1+\Vert  \Theta^\ast \Vert_{L^\infty L^\infty}^{\gamma_1+1})\Vert \Theta^\ast_t\Vert_{L^2 L^3}\\
    & \lesssim (1+\Vert  \Theta^\ast \Vert_{L^\infty L^\infty}^{\gamma_1+1})\Vert \Theta^\ast_t\Vert_{L^2 H^1}.\\
    \end{aligned}
\end{equation}
Moreover, elliptic regularity allows one to get
\begin{equation}
    \begin{aligned}
        \Vert \nabla r \Vert_{L^2 L^3} &=\Vert h '(\Theta^\ast) \nabla  \Theta^\ast \Vert_{L^2 L^3} \leq \Vert h '(\Theta^\ast) \Vert_{L^\infty L^\infty} \Vert \nabla \Theta^\ast \Vert_{L^2 L^3}\\
& \lesssim (1+\Vert  \Theta^\ast \Vert_{L^\infty L^\infty}^{\gamma_1+1})\Vert  \Theta^\ast \Vert_{L^2 H^2}.\\
    \end{aligned}
\end{equation}
Thus, it follows that
\begin{equation} 
\Vert  r_t \Vert_{L^2 L^3}+\Vert \nabla r \Vert_{L^2 L^3}\leq C_T(1+R_2^{\gamma_1+1})R_2.
\end{equation}
For convenience, we recall the definition of $\Lambda$
\begin{equation} 
\begin{aligned}
\Lambda(t)=\Vert \alpha_t(t) \Vert_{L^2}^2 &+\Vert \alpha_t(t) \Vert_{L^2}^{\frac{4}{4-d}}+\Vert r_t(t) \Vert_{L^2}^{\frac{4}{4-d}}+\Vert \nabla r(t) \Vert_{L^2}^2+\Vert  r_t(t) \Vert_{L^3}^2 \\
&+\Vert  \alpha_t(t) \Vert_{L^3}^2+ \Vert \nabla r(t) \Vert_{L^3}^2+\Vert \nabla \alpha(t) \Vert_{L^3}^2,\\
\end{aligned}
\end{equation}
so altogether the above estimates imply that
\begin{equation} \label{eq31}
\Vert \Lambda \Vert_{L^1(0,t)} \leq C_1(T, R_1, R_2).
\end{equation}
Now, we turn our attention to the source term $g$. Using H\"{o}lder's inequality, the embeddings $H^1(\Om) \hookrightarrow L^4(\Om)$, $H^1(\Om) \hookrightarrow L^6(\Om)$ and the fact that $\Vert (p^\ast_t)^2 \Vert_{ L^3}=\Vert p^\ast_t \Vert_{L^6}^2$, we have 
\begin{equation} 
\begin{aligned}
    \Vert g \Vert_{L^2 H^1} +\Vert g_t \Vert_{L^2 L^2}=&\Vert 4k (\Theta^\ast)\nabla p^\ast_t p^\ast_t +2k '(\Theta^\ast) \nabla \Theta^\ast (p^\ast_t)^2 \Vert_{L^2 L^2}\\
    &+\Vert 4 k (\Theta^\ast) p^\ast_t p^\ast_{tt}+2k '(\Theta^\ast) \Theta^\ast_t (p^\ast_t)^2 \Vert_{L^2 L^2}\\
    \lesssim &\,  \Vert  k (\Theta^\ast ) \Vert_{L^\infty L^\infty} \Vert \nabla p^\ast_t  \Vert_{L^2 L^4} \Vert p^\ast_t \Vert_{L^\infty L^4}\\
     &+\Vert k '(\Theta^\ast ) \Vert_{L^\infty L^\infty}\Vert \nabla \Theta^\ast  \Vert_{L^2 L^6} \Vert (p^\ast_t)^2 \Vert_{L^\infty L^3}\\
    &+\Vert  k (\Theta^\ast ) \Vert_{L^\infty L^\infty}  \Vert p^\ast_t \Vert_{L^\infty L^4} \Vert p^\ast_{tt} \Vert_{L^2 L^4}\\
    &+\Vert  k'(\Theta^\ast ) \Vert_{L^\infty L^\infty}\Vert  \Theta^\ast_t  \Vert_{L^2 L^6} \Vert (p^\ast_t)^2 \Vert_{L^\infty L^3}\\
     \lesssim &\, \Vert p^\ast_t  \Vert_{L^2 H^2} \Vert p^\ast_t \Vert_{L^\infty H^1} +(1+\Vert \Theta^\ast \Vert_{L^\infty L^\infty}^{\gamma_2+1}) \Vert \Theta^\ast  \Vert_{L^2 H^2} \Vert p^\ast_t \Vert_{L^\infty H^1}^2\\
    &+ \Vert p^\ast_t \Vert_{L^\infty H^1} \Vert p^\ast_{tt} \Vert_{L^2 H^1}+(1+\Vert \Theta^\ast \Vert_{L^\infty L^\infty}^{\gamma_2+1}) \Vert  \Theta^\ast_t  \Vert_{L^2 H^1} \Vert p^\ast_t \Vert_{L^\infty H^1}^2, \\
    \end{aligned}
\end{equation}
 which gives 
\begin{equation} \label{eq24}
\Vert g \Vert_{L^2 H^1} +\Vert g_t \Vert_{L^2 L^2} \leq C_TR_1^2(1+R_2+R_2^{2+\gamma_2}).
\end{equation}
Thus, we can bound from above the function $\mathfrak{F}$ defined in  \eqref{source_term} as 
\begin{equation}
    \begin{aligned}
\Vert \mathfrak{F} \Vert_{L^1(0,t)}= \int_0^t ( \Vert \nabla g \Vert_{L^2}^2+ \Vert g_t \Vert_{L^2}^2) \ds  \leq &\, \Vert g \Vert_{L^2 H^1}^2 +\Vert g_t \Vert_{L^2 L^2}^2\\
 \leq &\, C_T R_1^4(1+R_2^2(1+ R_2^{2+2\gamma_2})) . \\
 \end{aligned}
\end{equation}
Hence,  it results that
\begin{equation}
\Vert \mathfrak{F} \Vert_{L^1(0,t)} \leq R_1^4 C_2(T, R_2).
\end{equation}
Consequently, from Proposition \ref{prop2} we have the existence of a unique solution $p \in X_p$ to the first equation in \eqref{fixed_point_sys}. Moreover, since $(p^\ast, \Theta^\ast, q^\ast) \in B$, we have $f= \mathcal{Q}(p^\ast) \in H^2(0, T; L^2(\Om))$. Then, according to Proposition \ref{total-energy}, there exists a unique solution $(\Theta, q) \in X_\Theta \times X_q$ of the second and third equations in \eqref{fixed_point_sys}. That is to say, the mapping $\mathcal{T}$ is well-defined.

On account of \eqref{total_energy_p} and the fact that $\mathfrak{E}[p](0)\leq \delta$,  we obtain
\begin{equation}
\begin{aligned}
\Vert p \Vert_{X_p}^2  \lesssim &\, \sup_{t \in (0,T)} \mathfrak{E}[p](t)+\sup_{t \in (0,T)} b\Vert \Delta p_t(t) \Vert_{L^2}^2+ \int_0^T \mathfrak{D}[p](s) \ds\\
\lesssim &\,  \delta \exp(T (1+C_1(T, R_1, R_2)))+ T R_1^4\exp(T (1+C_1(T, R_1, R_2)))C_2(T, R_2).\\
\end{aligned}
\end{equation}
Thus, by choosing $\delta$ and $R_1$ small enough, we get
\begin{equation}
\Vert p \Vert_{X_p}  \leq R_1.
\end{equation}
Also, observing that  
\begin{equation}
\Vert p \Vert_{L^\infty L^\infty} \lesssim \Vert \Delta p \Vert_{L^\infty L^2} \lesssim \Vert p \Vert_{X_p},
\end{equation}
we obtain  an upper bound $\gamma < \dfrac{1}{2k_1}$ for $\Vert p \Vert_{L^\infty L^\infty}$ by possibly reducing $R_1$.\\
Therefore, it remains to verify that $\Vert (\Theta, q) \Vert_{X_\Theta\times X_q}\leq R_2$. First, we have by employing H\"{o}lder's inequality 
\begin{equation}
\begin{aligned}
 \Vert f \Vert_{H^2L^2}^2
%&\Vert f \Vert_{L^2L^2}^2+\Vert f_t \Vert_{L^2L^2}^2+\Vert f_{tt} \Vert_{L^2L^2}^2 \\
\lesssim &\, \Vert (p_t^\ast)^2 \Vert_{L^2L^2}^2+\Vert  2 p_t^\ast p_{tt}^\ast \Vert_{L^2L^2}^2+\Vert 2 (p_{tt}^\ast)^2+2p_t^\ast p_{ttt}^\ast \Vert_{L^2L^2}^2\\
\lesssim &\, \Vert p_t^\ast \Vert_{L^\infty L^4}^2\Vert p_t^\ast \Vert_{L^2L^4}^2+\Vert   p_t^\ast \Vert_{L^\infty L^4}^2 \Vert p_{tt}^\ast\Vert_{L^2L^4}^2\\
&+\Vert p_{tt}^\ast \Vert_{L^\infty L^4}^2 \Vert p_{tt}^\ast \Vert_{L^2L^4}^2+ \Vert p_t^\ast \Vert_{L^\infty L^\infty}^2 \Vert p_{ttt}^\ast \Vert_{L^2L^2}^2.
\end{aligned}
\end{equation}
Thanks to the embeddings $H^1(\Om) \hookrightarrow L^4(\Om), L^\infty(0, T) \hookrightarrow L^2(0, T)$, we find
\begin{equation}
%\Vert f \Vert_{L^2L^2}^2+\Vert f_t \Vert_{L^2L^2}^2+\Vert f_{tt} %\Vert_{L^2L^2}^2
\Vert f \Vert_{H^2L^2}^2 \leq C_T\Vert p \Vert_{X_p}^4\leq   C_T R_1^4.
\end{equation}   
Then, according to Proposition \ref{total-energy} and particularly the estimates \eqref{Pro_1_Estimate} and \eqref{q_estimate}, we get 
\begin{equation} 
\begin{aligned}
\Vert (\Theta, q) \Vert_{X_\Theta \times X_q}^2 \lesssim C_T \Big(\Vert q_0 \Vert^2_{H^1}+ (1+ \bar{\tau} +\bar{\tau}^2) (E^{\bar{\tau}}[\Theta, q](0)+\Vert f \Vert_{H^2L^2}^2) \Big).\\
\end{aligned}
\end{equation}
We emphasize that the constant $C_T>0$ in this  inequality does not depend on the parameter $\tau$. So it suffices to take $R_2$ large enough such that
\begin{equation}
 C_T\Big(\Vert q_0 \Vert^2_{H^1}+(1+\bar{\tau}+\bar{\tau}^2)( E^{\bar{\tau}}[\Theta, q](0)+ R_1^4)\Big) \leq R_2^2,
\end{equation}
in order to conclude that the solution $(p, \Theta, q )$ of the system \eqref{fixed_point_sys} remains in $B$.    
\end{proof}

\begin{lemma} \label{lemma4}
Let  $ \tau \in (0, \bar{\tau}]$. If $R_1$ and $\delta$ are sufficiently small, then the mapping $\mathcal{T}$ is a contraction on $B$.
\end{lemma}
\begin{proof}
Let $(p^\ast_1, \Theta^\ast_1,q_1^\ast), (p^\ast_2, \Theta^\ast_2, q_2^\ast) \in B$ and let $(p_1, \Theta_1, q_1), (p_2, \Theta_2, q_2)$ stand for their corresponding images by the operator $\mathcal{T}$; that is,  
\begin{equation}   
\mathcal{T}(p^\ast_1, \Theta^\ast_1,q_1^\ast)=(p_1, \Theta_1,q_1)\quad \text{and}\quad  \mathcal{T}(p^\ast_2, \Theta^\ast_2, q_2^\ast)=(p_2, \Theta_2, q_2).
\end{equation}
Since $(p_1, \Theta_1, q_1), (p_2, \Theta_2, q_2)$ are both solutions to system \eqref{fixed_point_sys}, clearly the differences
\begin{equation}
\begin{aligned}
&\hat{p}=p_1-p_2, \quad \hat{\Theta}=\Theta_1-\Theta_2, \quad  \hat{q}=q_1-q_2,\\
&\hat{p}^\ast=p_1^\ast-p_2^\ast, \quad \hat{\Theta}^\ast=\Theta_1^\ast-\Theta_2^\ast, \quad \hat{q}^\ast=q_1^\ast-q_2^\ast\\
\end{aligned}
\end{equation}
solve the following system
\begin{equation} \label{contraction_sys}
\left\{ 
\begin{aligned}
&(1-2k (\Theta^\ast_1)p_1^\ast)\hat{p}_{tt}-h (\Theta_1^\ast)\Delta \hat{p} - b \Delta \hat{p}_t = g_1, & \qquad &\\
& m \hat{\Theta}_t +\nabla\cdot \hat{q} + \ell \hat{\Theta} = f_1,
\\    
&\tau \hat{q}_t+\hat{q}+\kappaa \nabla \hat{\Theta}=0, \\ 
\end{aligned}
\text{in}\quad  \Omega \times (0,T),
\right.
\end{equation}
with the initial and boundary conditions
\begin{equation}
\begin{aligned}
&\hat{p}(x,0)=\hat{p}_t(x,0)=\hat{\Theta}(x,0)=0, \quad &\hat{q}(x,0)=0,  &\quad \text{in} \quad \ \Omega. \\
&\hat{p}=\hat{\Theta}=0, &\qquad &\text{on}\quad   \partial\Omega \times (0,T). 
\end{aligned}
\end{equation}
The forcing terms $f_1$ and $g_1$ are given by
\begin{equation}\label{g_1}
\begin{aligned}
f_1=&\,\mathcal{Q}(p_{1t}^\ast)-\mathcal{Q}(p_{2t}^\ast),\\
g_1=&\,2(k (\Theta_1^\ast)p_1^\ast-k (\Theta_2^\ast)p_2^\ast)p_{2tt}+(h (\Theta_1^\ast)-h (\Theta_2^\ast))\Delta p_2+2k (\Theta_1^\ast)(p^\ast_{1t})^2-2k (\Theta_2^\ast)(p^\ast_{2t})^2\\
=& \,2(k (\Theta_1^\ast)-k (\Theta_2^\ast))(p_2^\ast p_{2tt}+(p^\ast_{2t})^2)+(h (\Theta_1^\ast)-h (\Theta_2^\ast))\Delta p_2\\
&+2k (\Theta_1^\ast)\big((p^\ast_{1t}+p^\ast_{2t})\hat{p}^\ast_t+\hat{p}^\ast p_{2tt}\big)\\
:=&\,g_{11}+g_{12}+g_{13}.
\end{aligned}
\end{equation} 
   
We start by recalling some estimates derived in \cite{Nikolic_2022} on the functions $k , h $ and their derivatives. We have
\begin{equation}
k (\Theta_1^\ast)-k (\Theta_2^\ast)=(\Theta_1^\ast-\Theta_2^\ast) \int_0^1 k '(\Theta_2^\ast+\sigma(\Theta_1^\ast-\Theta_2^\ast) ) \textup{d} \sigma;
\end{equation}
hence using \eqref{properties_k},  together with the Sobolev embedding $H^2\hookrightarrow L^\infty$,  we get
\begin{subequations}   
\begin{equation}\label{k}
\begin{aligned}
\Vert k (\Theta_1^\ast)-k (\Theta_2^\ast) \Vert_{L^\infty L^\infty}&=\Big\Vert (\Theta_1^\ast-\Theta_2^\ast) \int_0^1 k '(\Theta_2^\ast+\sigma(\Theta_1^\ast-\Theta_2^\ast) ) \textup{d} \sigma \Big\Vert_{L^\infty L^\infty}\\
& \lesssim \Vert \Theta_1^\ast-\Theta_2^\ast \Vert_{L^\infty L^\infty} \Big(1+\Vert \Theta_2^\ast+\sigma(\Theta_1^\ast-\Theta_2^\ast) \Vert_{L^\infty L^\infty}^{1+\gamma_2} \Big)\\
& \lesssim \Vert (\hat{p}^\ast, \hat{\Theta}^\ast, \hat{q}^\ast) \Vert_{\mathcal{X}} \Big( 1+ \Vert \Theta_1^\ast \Vert_{L^\infty L^\infty}^{1+\gamma_2}+\Vert \Theta_2^\ast \Vert_{L^\infty L^\infty}^{1+\gamma_2}\Big).\\
\end{aligned}
\end{equation}
%Note that to obtain the last inequality, we have employed the algebraic inequality
%\begin{equation}
%(a+b)^\nu \leq \max\{1, 2^\nu \} (a^\nu+b^\nu),\quad a, b \geq 0, \quad  \nu>0.
%\end{equation}
In a similar fashion, on account of \eqref{h''_assump}, \eqref{h'_assump} and \eqref{properties_k}, we can show that 
\begin{equation}\label{k'}
\Vert k '(\Theta_1^\ast)-k '(\Theta_2^\ast) \Vert_{L^\infty L^\infty}\lesssim \Vert (\hat{p}^\ast, \hat{\Theta}^\ast, \hat{q}^\ast) \Vert_{\mathcal{X}} \Big( 1+ \Vert \Theta_1^\ast \Vert_{L^\infty L^\infty}^{\gamma_2}+\Vert \Theta_2^\ast \Vert_{L^\infty L^\infty}^{\gamma_2}\Big);
\end{equation}
\end{subequations}
\begin{subequations}
\begin{equation}\label{h}
\begin{aligned}
\Vert h (\Theta_1^\ast)-h (\Theta_2^\ast) \Vert_{L^\infty L^\infty} \lesssim  \Vert (\hat{p}^\ast, \hat{\Theta}^\ast, \hat{q}^\ast) \Vert_{\mathcal{X}} \Big( 1+ \Vert \Theta_1^\ast \Vert_{L^\infty L^\infty}^{1+\gamma_1}+\Vert \Theta_2^\ast \Vert_{L^\infty L^\infty}^{1+\gamma_1}\Big);\\
\end{aligned}
\end{equation}
\begin{equation}\label{h'}
\begin{aligned}
\Vert h '(\Theta_1^\ast)-h '(\Theta_2^\ast) \Vert_{L^\infty L^\infty} \lesssim \Vert (\hat{p}^\ast, \hat{\Theta}^\ast, \hat{q}^\ast) \Vert_{\mathcal{X}} \Big( 1+ \Vert \Theta_1^\ast \Vert_{L^\infty L^\infty}^{\gamma_1}+\Vert \Theta_2^\ast \Vert_{L^\infty L^\infty}^{\gamma_1}\Big).\\
\end{aligned}
\end{equation}
\end{subequations}
These estimates will be repeatedly used in the course of this proof.\\
First, we focus on estimating the $L^2$-norm of the gradient of $g_1$. Recall that 
\begin{equation}
g_{11}=2(k (\Theta_1^\ast)-k (\Theta_2^\ast))(p_2^\ast p_{2tt}+(p^\ast_{2t})^2).
\end{equation}
This entails that  
\begin{equation}
\begin{aligned}
  \nabla g_{11} =&\, 2  \nabla (k (\Theta_1^\ast)-k (\Theta_2^\ast))(p_2^\ast p_{2tt}+(p^\ast_{2t})^2)\\
  & +2 (k (\Theta_1^\ast)-k (\Theta_2^\ast))\nabla (p_2^\ast p_{2tt}+(p^\ast_{2t})^2) . 
\end{aligned}
\end{equation}
Since the heat flux $q$ is not present in the expression of the source term $g_1$; we can proceed as in \cite[Section 4]{Nikolic_2022} to estimate $\Vert \nabla g_{11} \Vert_{L^2 L^2}$. The fact that the operator $\mathcal{T}$ is self-mapping along with the estimate 
\begin{equation}
\Vert \Theta_1^\ast-\Theta_2^\ast \Vert_{X_\Theta}^2+\Vert p_1^\ast-p_2^\ast \Vert_{X_p}^2 \leq \Vert (\hat{p}^\ast, \hat{\Theta}^\ast, \hat{q}^\ast) \Vert_{\mathcal{X}}^2,
\end{equation}
allows us to obtain the following bound
\begin{equation}\label{g_11_Est}
\begin{aligned}\Vert \nabla g_{11} \Vert_{L^2 L^2} \leq C_T R_1^2(1+R_2+R_2^{1+\gamma_2}) \Vert (\hat{p}^\ast, \hat{\Theta}^\ast, \hat{q}^\ast) \Vert_{\mathcal{X}}.
\end{aligned}
\end{equation}
Analogously, we can follow the arguments in \cite{Nikolic_2022} in order to deal with the gradient of the second contribution 
\begin{equation}
\nabla g_{12}=\nabla (h (\Theta_1^\ast)-h (\Theta_2^\ast))\Delta p_2+h (\Theta_1^\ast)-h (\Theta_2^\ast))\nabla \Delta p_2.
\end{equation}
In fact, using the assumption \eqref{h'_assump} on the function $h $ together with  \eqref{h} and \eqref{h'}, we find
\begin{equation}
\Vert \nabla g_{12} \Vert_{L^2 L^2} \leq C_T R_1 (1+R_2+R_2^{1+\gamma_1}) \Vert (\hat{p}^\ast, \hat{\Theta}^\ast, \hat{q}^\ast) \Vert_{\mathcal{X}}.
\end{equation}
Thus, it remains to bound the gradient of the last component of $g_1$
\begin{equation}
\nabla g_{13}=2\big( \nabla ( k (\Theta_1^\ast))\big((p^\ast_{1t}+p^\ast_{2t}) \hat{p}^\ast_t+ \hat{p}^\ast p_{2tt}\big)+k (\Theta_1^\ast)\nabla \big((p^\ast_{1t}+p^\ast_{2t}) \hat{p}^\ast_t+ \hat{p}^\ast p_{2tt}\big) \big).
\end{equation}
Again, we can do as in \cite{Nikolic_2022} to get
\begin{equation}
\Vert \nabla g_{13} \Vert_{L^2 L^2} \leq C_T R_1(1+ R_2+R_2^{\gamma_2+2}) \Vert (\hat{p}^\ast, \hat{\Theta}^\ast, \hat{q}^\ast) \Vert_{\mathcal{X}}.
\end{equation}
Collecting the estimates of the components of $g_1$, it follows that
\begin{equation} \label{eq30}
\Vert \nabla g_{1} \Vert_{L^2 L^2} \leq R_1 C(T, R_1, R_2) \Vert (\hat{p}^\ast, \hat{\Theta}^\ast, \hat{q}^\ast) \Vert_{\mathcal{X}}.
\end{equation}

Next, we seek a similar estimate for the derivative in time of $g_1$. Differentiating $g_{11}$ with respect to $t$, we have
\begin{equation}
\begin{aligned}
  \partial_t g_{11} =& 2  \Big(\partial_t (k (\Theta_1^\ast)-k (\Theta_2^\ast))(p_2^\ast p_{2tt}+(p^\ast_{2t})^2) + (k (\Theta_1^\ast)-k (\Theta_2^\ast))\partial_t (p_2^\ast p_{2tt}+(p^\ast_{2t})^2)\Big)\\
  =&  2\Big(\big(k' (\Theta_1^\ast)\hat{\Theta}^\ast_t+(k' (\Theta_1^\ast)-k' (\Theta_2^\ast))\Theta_{2t}^\ast\big)(p_2^\ast p_{2tt}+(p^\ast_{2t})^2)\\
  &+(k (\Theta_1^\ast)-k (\Theta_2^\ast))(p_{2t}^\ast p_{2tt}+p_2^\ast p_{2ttt}+2p^\ast_{2t} p^\ast_{2tt}) \Big). 
\end{aligned}
\end{equation}
Using H\"older's inequality, we obtain
\begin{equation}
\begin{aligned}
\Vert \partial_t g_{11} \Vert_{L^2 L^2} \lesssim &\, \big(\Vert k' (\Theta_1^\ast) \Vert_{L^\infty L^\infty}  \Vert  \hat{\Theta}^\ast_t \Vert_{L^2 L^4}+\Vert k' (\Theta_1^\ast)-k' (\Theta_2^\ast) \Vert_{L^\infty L^\infty} \Vert \Theta_{2t}^\ast \Vert_{L^2 L^4} \big)\\
& \times \big(\Vert p_2^\ast \Vert_{L^\infty L^\infty} \Vert p_{2tt} \Vert_{L^\infty L^4}+\Vert p_{2t}^\ast \Vert_{L^\infty L^\infty} \Vert p_{2t}^\ast \Vert_{L^\infty L^4}\big)\\
&+ \Vert k (\Theta_1^\ast)-k (\Theta_2^\ast)\Vert_{L^\infty L^\infty} \big( \Vert p_{2t}^\ast \Vert_{L^\infty L^4} \Vert p_{2tt}\Vert_{L^2 L^4} +\Vert p_2^\ast\Vert_{L^\infty L^\infty} \Vert p_{2ttt} \Vert_{L^2 L^2}\\
&+ \Vert p^\ast_{2t}\Vert_{L^\infty L^4} \Vert  p^\ast_{2tt} \Vert_{L^2 L^4}\big)
\end{aligned}
\end{equation}
Then taking into account \eqref{properties_k}, \eqref{k}, \eqref{k'} and using the embedding $H^1(\Om) \hookrightarrow L^4(\Om)$, it follows
\begin{equation}
\Vert \partial_t g_{11} \Vert_{L^2 L^2}  \leq C_T R_1^2(1+ R_2+R_2^{1+\gamma_2}) \Vert (\hat{p}^\ast, \hat{\Theta}^\ast, \hat{q}^\ast) \Vert_{\mathcal{X}}.
\end{equation}
The time derivative of $g_{12}$ is given by
\begin{equation}
\begin{aligned}
\partial_t g_{12} =& \partial_t (h (\Theta_1^\ast)-h (\Theta_2^\ast))\Delta p_2+(h (\Theta_1^\ast)-h (\Theta_2^\ast)) \Delta p_{2t},\\
=& \big(h '(\Theta_1^\ast) \hat{\Theta}^\ast_t+(h' (\Theta_1^\ast)-h' (\Theta_2^\ast))\Theta_{2t}^\ast \big)\Delta p_2+(h (\Theta_1^\ast)-h (\Theta_2^\ast)) \Delta p_{2t}.
\end{aligned}
\end{equation}
Hence, it holds that
\begin{equation}
\begin{aligned}
\Vert \partial_t g_{11} \Vert_{L^2 L^2}  \lesssim &\, \big(\Vert h '(\Theta_1^\ast)\Vert_{L^\infty L^\infty}  \Vert  \hat{\Theta}^\ast_t \Vert_{L^2 L^4} +\Vert h' (\Theta_1^\ast)-h' (\Theta_2^\ast)\Vert_{L^\infty L^\infty}  \Vert \Theta_{2t}^\ast \Vert_{L^2 L^4}   \big)\Vert \Delta p_2 \Vert_{L^\infty L^4} \\
&+\Vert h (\Theta_1^\ast)-h (\Theta_2^\ast)\Vert_{L^\infty L^\infty}  \Vert \Delta p_{2t} \Vert_{L^2 L^2}. 
\end{aligned}
\end{equation}
Thanks to \eqref{h'_assump}, \eqref{h} and \eqref{h'}, the above estimate implies that
\begin{equation}
\Vert \partial_t g_{12} \Vert_{L^2 L^2}  \leq C_T R_1(1+ R_2+R_2^{1+\gamma_1}) \Vert (\hat{p}^\ast, \hat{\Theta}^\ast, \hat{q}^\ast) \Vert_{\mathcal{X}}.
\end{equation}
Differentiating in time the last contribution $g_{13}$, we find
\begin{equation}
\begin{aligned}
 \partial_t g_{13}=&2\big(   k' (\Theta_1^\ast)\Theta_{1t}^\ast\big((p^\ast_{1t}+p^\ast_{2t}) \hat{p}^\ast_t+ \hat{p}^\ast p_{2tt}\big)\\
 &+k (\Theta_1^\ast) \big((p^\ast_{1tt}+p^\ast_{2tt}) \hat{p}^\ast_t+(p^\ast_{1t}+p^\ast_{2t}) \hat{p}^\ast_{tt}+ \hat{p}^\ast_t p_{2tt}+\hat{p}^\ast p_{2ttt}\big) \big).\\
\end{aligned}
\end{equation}
Applying H\"older's inequality yields
\begin{equation}
\begin{aligned}     
&\Vert \partial_t g_{13} \Vert_{L^2 L^2}\\
 \lesssim &\,  \Vert  k' (\Theta_1^\ast)\Vert_{L^\infty L^\infty} \Vert \Theta_{1t}^\ast \Vert_{L^2 L^4} \big(\Vert p^\ast_{1t}+p^\ast_{2t}\Vert_{L^\infty L^4} \Vert \hat{p}^\ast_t \Vert_{L^\infty L^\infty}+ \Vert \hat{p}^\ast \Vert_{L^\infty L^\infty} \Vert p_{2tt} \Vert_{L^\infty L^4} \big)\\
&+\Vert k (\Theta_1^\ast)\Vert_{L^\infty L^\infty} \big(\Vert p^\ast_{1tt}+p^\ast_{2tt}\Vert_{L^2 L^4} \Vert \hat{p}^\ast_t \Vert_{L^\infty L^4}+ \Vert p^\ast_{1t}+p^\ast_{2t} \Vert_{L^\infty L^4} \Vert \hat{p}^\ast_{tt} \Vert_{L^2 L^4}\\
&+ \Vert\hat{p}^\ast_t \Vert_{L^\infty L^4} \Vert p_{2tt} \Vert_{L^2 L^4} + \Vert \hat{p}^\ast \Vert_{L^\infty L^\infty} \Vert  p_{2ttt} \Vert_{L^2 L^2} \big).
\end{aligned}
\end{equation}
Then, from the assumption \eqref{properties_k} and the fact that  the mapping $\mathcal{T}$ is self-mapping, we arrive at
\begin{equation}
\Vert \partial_t g_{13} \Vert_{L^2 L^2} \leq C_T R_1( 1+ R_2+R_2^{2+\gamma_2}) \Vert (\hat{p}^\ast, \hat{\Theta}^\ast, \hat{q}^\ast) \Vert_{\mathcal{X}}.
\end{equation}
Putting together the above estimates results in 
\begin{equation}\label{eq29}
\Vert \partial_t g_{1} \Vert_{L^2 L^2} \leq R_1 C(T, R_1, R_2)\Vert (\hat{p}^\ast, \hat{\Theta}^\ast, \hat{q}^\ast) \Vert_{\mathcal{X}}.
\end{equation}
As for the source term $f_1$ in the temperature equation, we can handle it as follows
\begin{equation}
\begin{aligned}
\Vert f_1 \Vert_{H^2 L^2}=&\,\Vert f_1 \Vert_{L^2L^2}+\Vert f_{1t} \Vert_{L^2L^2}+\Vert f_{1tt} \Vert_{L^2L^2}\\
\lesssim &\,\Vert (p_{1t}^\ast)^2-(p_{2t}^\ast)^2 \Vert_{L^2L^2}+ \Vert (p_{1tt}^\ast-p_{2tt}^\ast)(p_{1t}^\ast) \Vert_{L^2L^2} +\Vert(p_{1t}^\ast-p_{2t}^\ast)(p_{2tt}^\ast) \Vert_{L^2L^2}\\    
&+\Vert (p_{1ttt}^\ast-p_{2ttt}^\ast)p_{1t}^\ast  \Vert_{L^2L^2}+\Vert (p_{1tt}^\ast-p_{2tt}^\ast)p_{1tt}^\ast  \Vert_{L^2L^2}\\
&+\Vert (p_{1tt}^\ast-p_{2tt}^\ast)p_{2tt}^\ast \Vert_{L^2L^2}
+\Vert(p_{1t}^\ast-p_{2t}^\ast)p_{2ttt}^\ast \Vert_{L^2L^2}.
\end{aligned}     
\end{equation}
Thus, we get
\begin{equation}
\begin{aligned}
\Vert f_1 \Vert_{H^2 L^2} \lesssim &\, \Vert p_{1t}^\ast-p_{2t}^\ast \Vert_{L^\infty L^\infty }\Vert p_{1t}^\ast+p_{2t}^\ast \Vert_{L^2L^2}+ \Vert p_{1tt}^\ast-p_{2tt}^\ast \Vert_{L^2 L^4} \Vert p_{1t}^\ast \Vert_{L^\infty L^4}\\
&+\Vert p_{1t}^\ast-p_{2t}^\ast \Vert_{L^\infty L^4} \Vert p_{2tt}^\ast\Vert_{L^2 L^4} +\Vert p_{1ttt}^\ast-p_{2ttt}^\ast \Vert_{L^2L^2} \Vert p_{1t}^\ast \Vert_{L^\infty L^\infty}\\
&+\Vert p_{1tt}^\ast-p_{2tt}^\ast \Vert_{L^\infty L^4} \Vert p_{1tt}^\ast \Vert_{L^2L^4}+\Vert p_{1tt}^\ast-p_{2tt}^\ast \Vert_{L^\infty L^4} \Vert p_{2tt}^\ast \Vert_{L^2L^4}\\
&+\Vert p_{1t}^\ast-p_{2t}^\ast \Vert_{L^\infty L^\infty} \Vert p_{2ttt}^\ast \Vert_{L^2L^2}.\\
\end{aligned}
\end{equation}
Using the embeddings $H^1(\Om) \hookrightarrow L^4(\Om)$,  $H^2(\Om) \hookrightarrow L^\infty(\Om)$ and elliptic regularity, we deduce that 
\begin{equation} \label{eq28} 
\Vert f_1 \Vert_{H^2 L^2} \leq R_1 C_T \Vert (\hat{p}^\ast, \hat{\Theta}^\ast, \hat{q}^\ast) \Vert_{\mathcal{X}}.
\end{equation}
From the previous Lemma \ref{lemma3}, we know that the coefficients of the first equation in the system \eqref{contraction_sys} 
$$\alpha_1=1-2 k (\Theta_1^\ast)p_1^\ast, \qquad r_1=h (\Theta_1^\ast)$$
satisfy the requirements of Proposition \ref{prop2}. Further, setting $g=g_1$ and $f=f_1$ and taking into account the above estimates, we deduce that the results of the Propositions \ref{total-energy},  \ref{prop2} hold. In particular, the energy estimates \eqref{Pro_1_Estimate}, \eqref{q_estimate} and \eqref{total_energy_p} hold for the solution $(\hat{p}, \hat{\Theta}, \hat{q})$; that is
\begin{equation}
\begin{aligned}
\Vert (\hat{p}, \hat{\Theta}, \hat{q}) \Vert_{\mathcal{X}}^2=&\Vert \mathcal{T}(p^\ast_1, \Theta^\ast_1, q^\ast_1)-\mathcal{T}(p^\ast_2, \Theta^\ast_2, q^\ast_2) \Vert_{ \mathcal{X}}^2\\
\lesssim &\, \sup_{t \in (0,T)} \mathfrak{E}[\hat{p}](t)+\sup_{t \in (0,T)} b \Vert \Delta \hat{p}_t(t) \Vert_{L^2}^2+\int_0^T \mathfrak{D}[\hat{p}](s) \ds\\
&+ \sup_{t \in (0,T)} \mathcal{E}[\hat{\Theta}](t) + \Vert \hat{q} \Vert_{L^\infty H^1}^2+\sum_{k=1}^2 \int_0^T \Vert \partial_t^k \hat{q}(s) \Vert_{L^2}^2 \ds\\
\lesssim &\, \exp \big(\int_0^T(1+\Lambda(s)) \ds\big) \int_0^T (\Vert \nabla g_1\Vert_{L^2}^2 + \Vert g_{1t} \Vert_{L^2}^2) \ds \\
&+(1+ \bar{\tau} +\bar{\tau}^2)\Vert f_1 \Vert_{H^2L^2}^2.\\
\end{aligned}
\end{equation}
Hence, due to estimates \eqref{eq31}, \eqref{eq30}, \eqref{eq29} and \eqref{eq28}, we find    
\begin{equation}
\begin{aligned}  
\Vert (\hat{p}, \hat{\Theta}, \hat{q}) \Vert_{\mathcal{X}}^2 \lesssim &\, R_1^2 C(T, R_1, R_2)\exp (T +T C_1(T, R_1, R_2))    \big\Vert (\hat{p}^\ast, \hat{\Theta}^\ast, \hat{q}^\ast) \big\Vert_{\mathcal{X}}^2 \\
&+(1+ \bar{\tau} +\bar{\tau}^2)R_1^2 C_T^2\Vert (\hat{p}^\ast, \hat{\Theta}^\ast, \hat{q}^\ast) \Vert_{\mathcal{X}}^2,\\
\end{aligned}   
\end{equation}    
where $C_1(T, R_1, R_2)$ is the constant in \eqref{eq31}. From here, we infer that
\begin{equation}
\Vert \mathcal{T}(p^\ast_1, \Theta^\ast_1, q^\ast_1)-\mathcal{T}(p^\ast_2, \Theta^\ast_2, q^\ast_2) \Vert_{ \mathcal{X}} \lesssim R_1 C(T, R_1, R_2,\bar{\tau}) \Vert (\hat{p}^\ast, \hat{\Theta}^\ast, \hat{q}^\ast) \Vert_{\mathcal{X}}. 
\end{equation}
Consequently, it suffices to select $R_1$ as small as needed to ensure that the mapping $\mathcal{T}$ is a strict contraction. 
\end{proof}

\section{Limiting behaviour  as $\tau \searrow 0$: Proof of Theorem \ref{limit_thm}} \label{sect7}
We aim in this section to prove Theorem \ref{limit_thm}.   More precisely, we prove  that the solution of the Westervelt--Pennes--Cattaneo  system  \eqref{Main_system} converges as $\tau$ goes to zero, to the solution of the limiting system, namely the  Westervelt--Pennes--Fourier system  corresponding to $\tau=0$. For the sake of simplicity and to be able to relate it to the  Westervelt--Pennes--Fourier  problem analyzed in \cite{Nikolic_2022}, we replace the system \eqref{eq_Cattaneo} by the equivalent telegraph equation derived in \eqref{telegraph_eq}. Let $(p^\tau, \Theta^\tau)$ be the solution  of  the following $\tau$-dependent problem
\begin{equation} \label{tau_system}
\left\{
\begin{aligned}
&(1-2k(\Theta^\tau)p^\tau) p_{tt}^\tau-h(\Theta^\tau) \Delta p^\tau -b \Delta p_t^\tau=2k(\Theta^\tau) (p_t^\tau)^2, &\ &\text{in} \ \Om \times (0, T), \\
& \tau m \Theta_{tt}^\tau+(m+\tau \ell)\Theta_t^\tau+\ell \Theta^\tau-\kappaa \Delta \Theta^\tau=\mathcal{Q}(p_t^\tau)+\tau \partial_t\big(\mathcal{Q}(p_t^\tau)\big), &\ &\text{in} \ \Om \times (0, T),
\end{aligned}
\right.
\end{equation} 
subject to homogeneous Dirichlet boundary conditions \eqref{homog_dirichlet} and initial conditions
\begin{equation}
(p^\tau, p_t^\tau)|_{t=0}=(p_0^\tau, p_1^\tau), \quad (\Theta^\tau, \Theta_t^\tau)|_{t=0}=(\Theta_0^\tau, \Theta_1^\tau)
\end{equation}
with $\Theta_1^\tau=\frac{1}{m}\big(-\nabla \cdot q^\tau_0- \ell \Theta_0^\tau+ \mathcal{Q}(p_1^\tau) \big)$.

The proof presented below is based on the method employed in \cite{Kaltenbacher2019TheJE,nikolic2023nonlinear} to examine the weak limit of the Jordan–Moore–Gibson–Thompson-type equations.   

\subsection{Proof of Theorem \ref{limit_thm}}

According to Theorem \ref{wellposedness_thm}, we know that if the initial data $p_0^\tau, p_1^\tau$, $\Theta_0^\tau, \Theta_1^\tau$ satisfies
\begin{equation}
\begin{aligned}
&\mathfrak{E}[p](0) \leq \delta, \\
& \mathcal{E}[\Theta](0) \leq 
C\Big(\Vert q_0 \Vert^2_{H^1}+(1+ \bar{\tau} +\bar{\tau}^2)\big(E^{\bar{\tau}}[\Theta, q](0)+R_1^2 \big) \Big) \leq R_2^2,
\end{aligned} 
\end{equation}
then the system \eqref{tau_system} admits a unique solution verifying
\begin{equation}
\Vert p^\tau \Vert_{X_p} \leq R_1, \quad  \Vert \Theta^\tau \Vert_{X_\Theta} \leq R_2.
\end{equation}
Since the radii $R_1, R_2$ and $\delta$ are independent of $\tau$, we can find subsequences such that
\begin{equation} \label{initial-weak-conv}
\begin{aligned}
 &(p_0^\tau, p_1^\tau) \rightharpoonup (p_0, p_1) & \quad & \text{weakly in} \ & [H^3(\Om)& \cap H^1_0(\Om)] \times [H^3(\Om) \cap H^1_0(\Om)],\\  
 &(\Theta_0^\tau, \Theta_1^\tau) \rightharpoonup (\Theta_0, \Theta_1) & \quad & \text{weakly in} \ & [H^2(\Om)& \cap H^1_0(\Om)] \times H^1_0(\Om).
\end{aligned}
\end{equation}    
as $\tau$ tends to zero. Further, owing to the compactness of the embedding $H^{s+1} \hookrightarrow H^s, s \in  \mathbb{N}$,   it follows that there exists  strongly convergent subsequences still denoted by $(p_0^\tau, p_1^\tau), (\Theta_0^\tau, \Theta_1^\tau)$,  such that 
\begin{equation} \label{initial-str-conv_1}
\begin{aligned}
&(p_0^\tau, p_1^\tau) \rightarrow (p_0, p_1) & \quad & \text{in} \ & [H^2(&\Om) \cap H^1_0(\Om)] \times  [H^2(\Om) \cap H^1_0(\Om)],\\  
 &(\Theta_0^\tau, \Theta_1^\tau) \rightarrow (\Theta_0, \Theta_1) & \quad & \text{in} \ & H^1_0(&\Om) \times L^2(\Om).
\end{aligned}
\end{equation} Similarly, since the solution $(p^\tau, \Theta^\tau)$ is bounded uniformly in   $\tau$, there exists subsequences (which we keep on denoting by the index $\tau$) converging in the weak-$\star$ topology
\begin{equation} \label{weak_star_conv}
\begin{aligned}
&p^\tau \rightharpoonup p &\quad &\text{weakly-$\star$} \ &\text{in} \quad &L^\infty (0, T; H^3(\Om) \cap H^1_0(\Om));\\
&p^\tau_t \rightharpoonup p_t & \quad &\text{weakly-$\star$} \ &\text{in} \quad  &L^\infty  (0, T; H^2(\Om) \cap H^1_0(\Om));\\ 
&p^\tau_{tt} \rightharpoonup p_{tt} & \quad & \text{weakly-$\star$} \ &\text{in} \quad  &L^\infty  (0, T;H^1_0(\Om));\\
&p^\tau_{ttt} \rightharpoonup p_{ttt} & \quad & \text{weakly} \ &\text{in} \quad  &L^2  (0, T;L^2(\Om));\\  
&\Theta^\tau \rightharpoonup \Theta & \quad & \text{weakly-$\star$} \ &\text{in} \quad  &L^\infty  (0, T; H^2(\Om) \cap H^1_0(\Om));\\ 
&\Theta^\tau_t \rightharpoonup \Theta_t & \quad & \text{weakly-$\star$} \ &\text{in} \quad  &L^\infty  (0, T; H^1_0(\Om));\\    
&\Theta^\tau_{tt} \rightharpoonup \Theta_{tt} & \quad & \text{weakly-$\star$} \quad  &\text{in} \quad  &L^\infty  (0, T;  L^2(\Om)). 
\end{aligned}
\end{equation}
Next, we seek to connect the limit of the initial data sequence found in \eqref{initial-str-conv_1} to the initial state of the limit functions in \eqref{weak_star_conv}. For this purpose, we invoke Aubin--Lions lemma \cite[Lemma 1.2]{L69}, which we can apply here thanks to the compactness of the embeddings $H^{s+1} \hookrightarrow H^s, s \in  \mathbb{N}$. Then, we get (up to subsequences)
\begin{equation} \label{cont_strong_conv}
\begin{aligned}
&p^\tau \rightarrow p & \quad &\text{strongly} \ \text{in} &\ &C(0, T; H^2(\Om) \cap H^1_0(\Om));\\
&p^\tau_t \rightarrow p_t &\quad &\text{strongly} \ \text{in} &\ &C(0, T; H^1_0(\Om));\\ 
&p^\tau_{tt} \rightarrow p_{tt} &\quad &\text{strongly} \ \text{in} &\ &L^2 (0, T;L^2(\Om));\\ 
&\Theta^\tau \rightarrow \Theta &\quad &\text{strongly} \ \text{in} &\ &C(0, T; H^1_0(\Om));\\ 
&\Theta^\tau_t \rightarrow \Theta_t &\quad &\text{strongly} \ \text{in} &\ &C(0, T; L^2(\Om)).\\ 
%&\Theta^\tau_{tt} \rightarrow \Theta_{tt} &\quad \text{strongly} \ \text{in} &\ &L^\infty (0, T; H^2(\Om) \cap L^2(\Om)). 
\end{aligned}
\end{equation}
This also leads to strong convergence of the initial data as follows 
\begin{equation} \label{initial-str-conv_2}
\begin{aligned}
&(p^\tau,p_t^\tau)(0) \rightarrow (p,p_t)(0) &\quad &\text{in} &\ &[H^2(\Om) \cap H^1_0(\Om)] \times H^1_0(\Om);\\
&(\Theta^\tau,\Theta_t^\tau)(0) \rightarrow (\Theta,\Theta_t)(0) &\quad &\text{in} &\ &H^1_0(\Om) \times L^2(\Om).
\end{aligned} 
\end{equation}
Hence, the uniqueness of the limit along with \eqref{initial-str-conv_1}, \eqref{initial-str-conv_2} and \eqref{initial-weak-conv} gives
\begin{equation}
\begin{aligned}
& (p,p_t)(0)=(p_0, p_1) \in \, [H^3(\Om) \cap H^1_0(\Om)] \times [H^3(\Om) \cap H^1_0(\Om)],\\
&(\Theta,\Theta_t)(0)=(\Theta_0, \Theta_1) \in \, [H^2(\Om) \cap H^1_0(\Om)] \times H^1_0(\Om).
\end{aligned}
\end{equation} 
This means that the initial states of the limit functions in \eqref{weak_star_conv} are well-defined  and have the necessary smoothness. Therefore, we can now focus on the task of passing to the limit in system \eqref{tau_system} when $\tau$ goes to zero. \\
Let $v \in L^1(0, T; L^2(\Om))$. We denote by $(\hat{p}, \hat{\Theta})$ the differences $\hat{p}=p-p^\tau, \hat{\Theta}=\Theta-\Theta^\tau$. Then, we have
\begin{equation} \label{eq_theta_limit}
\begin{aligned}
&\int^T_0 \intO \big(m \Theta_t -\kappaa \Delta \Theta +\ell \Theta -\mathcal{Q}(p_t)\big)v \dx \dt\\
=&\int^T_0 \intO \big( m \hat{\Theta}_t- \kappaa \Delta \hat{\Theta} + \ell \hat{\Theta} - \tau m \Theta_{tt}^\tau -\tau  \ell \Theta^\tau_t+\mathcal{Q}(p_t^\tau)-\mathcal{Q}(p_t)+\tau \partial_t \mathcal{Q}(p_t^\tau) \big)v \dx \dt.
\end{aligned}
\end{equation}
From the weak-$\star$ convergence \eqref{weak_star_conv}, we get
\begin{equation}
\int^T_0 \intO \big( m \hat{\Theta}_t- \kappaa \Delta \hat{\Theta} + \ell \hat{\Theta}\big)v \dx \dt \longrightarrow 0 \quad \text{as} \ \tau \rightarrow 0.
\end{equation}
Moreover, recalling that $\mathcal{Q}(p_t)=\frac{2b}{\rhoa \Ca^4}(p_t)^2$ and keeping in mind the regularity provided by the fact that $(p^\tau, \Theta^\tau) \in X_p \times X_\Theta$, we obtain 
\begin{equation}
\begin{aligned}
&\int^T_0 \intO \big(- \tau m \Theta_{tt}^\tau -\tau  \ell \Theta^\tau_t+\frac{4b\tau}{\rhoa \Ca^4}p_t^\tau p_{tt}^\tau \big)v \dx \dt \\
\lesssim &\, \tau \big( \Vert \Theta_{tt}^\tau \Vert_{L^\infty L^2}+\Vert \Theta_{t}^\tau \Vert_{L^\infty L^2}+\Vert p_t^\tau \Vert_{L^\infty L^4} \Vert p_{tt}^\tau \Vert_{L^\infty L^4} \big) \Vert v \Vert_{L^1 L^2} \longrightarrow 0
\end{aligned}
\end{equation}
as $\tau$ tends to zero. The two remaining terms on the right-hand side of \eqref{eq_theta_limit} can be treated as follows
\begin{equation}
\begin{aligned}
\int^T_0 \intO \big(\mathcal{Q}(p_t^\tau)-\mathcal{Q}(p_t)\big)v \dx \dt&=-\int^T_0 \intO \frac{2b}{\rhoa \Ca^4} (p_t^\tau+p_t)\hat{p}_t v \dx \dt\\
& \lesssim \Vert (p_t^\tau+p_t) \Vert_{L^\infty L^4} \Vert \hat{p}_t \Vert_{L^\infty L^4} \Vert v \Vert_{L^1 L^2}.
\end{aligned}
\end{equation}
Then, the embedding $H^1(\Om) \hookrightarrow L^4(\Om)$ together with the boundedness of $p^\tau_t, p_t $ and the convergence \eqref{cont_strong_conv} yields
\begin{equation}
\int^T_0 \intO \big(\mathcal{Q}(p_t^\tau)-\mathcal{Q}(p_t)\big)v \dx \dt \longrightarrow 0 \quad \text{as} \ \tau \rightarrow 0.
\end{equation}
Consequently, we infer that $(p,\Theta)$ satisfies in $L^\infty(0, T; L^2(\Om))$ the equation 
\begin{equation}
m \Theta_t -\kappaa \Delta \Theta +\ell \Theta =\mathcal{Q}(p_t).
\end{equation}

Next, we claim that for all $v \in L^1(0, T; L^2(\Om))$, it holds that
%$v \in C^\infty_0(0, T; C^\infty_0(\Om))$
\begin{equation}
\int^T_0 \intO \Big((1-2k(\Theta)p) p_{tt}-h(\Theta) \Delta p -b \Delta p_t-2k(\Theta) (p_t)^2 \Big) v \dx \dt \longrightarrow 0 \quad \text{as} \ \tau \rightarrow 0.
\end{equation}
To see this, we first observe that
\begin{equation} \label{eq_p_limit}
\begin{aligned}
&\int^T_0 \intO \Big((1-2k(\Theta)p) p_{tt}-h(\Theta) \Delta p -b \Delta p_t-2k(\Theta) (p_t)^2\Big) v \dx \dt\\
=& \int^T_0 \intO \Big((1-2k(\Theta)p) \hat{p}_{tt}+(1-2k(\Theta)p) p_{tt}^\tau-h(\Theta) \Delta \hat{p}-h(\Theta) \Delta p^\tau\\
&-b \Delta \hat{p}_t-b \Delta p_t^\tau-2k(\Theta) ((p_t)^2-(p_t^\tau)^2)-2k(\Theta) (p_t^\tau)^2\Big) v \dx \dt\\
=& \int^T_0 \intO \Big((1-2k(\Theta)p) \hat{p}_{tt}-h(\Theta) \Delta \hat{p}-b \Delta \hat{p}_t+(1-2k(\Theta^\tau)p^\tau) p_{tt}^\tau+2\big(k(\Theta^\tau)p^\tau-k(\Theta)p\big)p_{tt}^\tau\\
&-(h(\Theta)-h(\Theta^\tau)) \Delta p^\tau-h(\Theta^\tau) \Delta p^\tau-b \Delta p_t^\tau-2k(\Theta) (p_t+p_t^\tau)\hat{p}_t-2k(\Theta) (p_t^\tau)^2\Big) v \dx \dt\\
=&\int^T_0 \intO  \Big((1-2k(\Theta)p) \hat{p}_{tt}-h(\Theta) \Delta \hat{p}-b \Delta \hat{p}_t-2k(\Theta^\tau)\hat{p}p_{tt}^\tau+2\big(k(\Theta^\tau)-k(\Theta)\big)pp_{tt}^\tau\\
&-(h(\Theta)-h(\Theta^\tau)) \Delta p^\tau-2k(\Theta) (p_t+p_t^\tau)\hat{p}_t+2\big(k(\Theta^\tau)-k(\Theta)\big) (p_t^\tau)^2\Big) v \dx \dt.
\end{aligned}
\end{equation}
We want to prove that the integral on the right-hand side goes to zero as the relaxation time vanishes. Again, the weak-$\star$ convergence of $p^\tau$ to $p$ in \eqref{weak_star_conv} along with the boundedness of $(p,\Theta) \in X_p \times X_\Theta$ implies that
\begin{equation}
\begin{aligned}
&h(\Theta ) \Delta p^\tau \rightharpoonup h(\Theta) \Delta p &\quad &\text{weakly-$\star$} \ &\text{in} \ &L^\infty (0, T; H^1_0(\Om));\\
&((1-2k(\Theta)p) p^\tau_{tt} \rightharpoonup ((1-2k(\Theta)p) p_{tt} & \quad & \text{weakly-$\star$} \ &\text{in} \ &L^\infty  (0, T;H^1_0(\Om)).\\ 
\end{aligned} 
\end{equation}
Then, we can conclude that
\begin{equation}
\int^T_0 \intO  \Big((1-2k(\Theta)p) \hat{p}_{tt}-h(\Theta) \Delta \hat{p}-b \Delta \hat{p}_t\Big) v \dx \dt \longrightarrow 0 \quad \text{as} \ \tau \rightarrow 0.
\end{equation} 
Further, relying on \eqref{cont_strong_conv}, we have 
\begin{equation}
\begin{aligned}
&\int^T_0 \intO  \Big(-2k(\Theta^\tau)\hat{p}p_{tt}^\tau-2k(\Theta) (p_t+p_t^\tau)\hat{p}_t \Big) v \dx\dt\\
\lesssim &\, \Big( \Vert k(\Theta^\tau) \Vert_{L^\infty L^\infty}  \Vert \hat{p} \Vert_{L^\infty L^4}  \Vert p_{tt}^\tau \Vert_{L^\infty L^4}+ \Vert k(\Theta) \Vert_{L^\infty  L^\infty} \Vert p_t+p_t^\tau \Vert_{L^\infty L^4} \Vert \hat{p}_t \Vert_{L^\infty L^4}  \Big) \Vert v  \Vert_{L^1L^2},
\end{aligned}
\end{equation}
which yields that
\begin{equation}
\int^T_0 \intO  \Big(-2k(\Theta^\tau)\hat{p}p_{tt}^\tau-2k(\Theta) (p_t+p_t^\tau)\hat{p}_t \Big) v \dx\dt \longrightarrow 0 \quad \text{as} \ \tau \rightarrow 0.
\end{equation}
In order to handle the terms involving the difference $k(\Theta^\tau)-k(\Theta)$ or $h(\Theta^\tau)-h(\Theta)$, we call upon analogous estimates to \eqref{k}, \eqref{h}. Since we can show as in \eqref{k} that 
\begin{equation}
\Vert k(\Theta^\tau)-k(\Theta) \Vert_{L^\infty L^4} \lesssim \Vert \Theta^\tau-\Theta \Vert_{L^\infty L^4} \Big(1+\Vert \Theta^\tau \Vert^{1+\gamma_2}_{L^\infty L^4}+\Vert \Theta \Vert^{1+\gamma_2}_{L^\infty L^4}\Big),
\end{equation}
the boundedness of $\Theta, \Theta^\tau$ in $X_\Theta$ allows to obtain that
\begin{equation}
\begin{aligned}
&\int^T_0 \intO  2 \Big(\big(k(\Theta^\tau)-k(\Theta)\big)pp_{tt}^\tau+\big(k(\Theta^\tau)-k(\Theta)\big) (p_t^\tau)^2\Big) v \dx \dt\\
\leq &\, 2 \Big(\Vert p \Vert_{L^\infty L^\infty} \Vert p_{tt}^\tau \Vert_{L^\infty L^4} +\Vert p^\tau_t \Vert_{L^\infty L^4}^2 \Big) \Vert k(\Theta^\tau)-k(\Theta) \Vert_{L^\infty L^4} \Vert v \Vert_{L^1 L^2}  \rightarrow 0 \quad \text{as} \ \tau \rightarrow 0.
\end{aligned}
\end{equation}
Likewise, we have
\begin{equation} \label{Ineq}
\int^T_0 \intO   \big(h(\Theta^\tau)-h(\Theta)\big) \Delta p^\tau v \dx \dt \lesssim \Vert h(\Theta^\tau)-h(\Theta) \Vert_{L^\infty L^4} \Vert \Delta p^\tau \Vert_{L^\infty L^4} \Vert v \Vert_{L^1 L^2}. 
\end{equation}
Then the following estimate
\begin{equation}
\Vert h(\Theta^\tau)-h(\Theta) \Vert_{L^\infty L^4} \lesssim \Vert \Theta^\tau-\Theta \Vert_{L^\infty L^4} \Big(1+\Vert \Theta^\tau \Vert^{1+\gamma_1}_{L^\infty L^4}+\Vert \Theta \Vert^{1+\gamma_1}_{L^\infty L^4}\Big),
\end{equation}
leads to the convergence of the integral on the left of \eqref{Ineq} to zero as $\tau$ goes to zero. Thus, all the terms on the right-hand side of \eqref{eq_p_limit} vanish as $\tau$ tends to zero, then we infer that the limit $p$ satisfies the equation
\begin{equation}
(1-2k(\Theta)p) p_{tt}-h(\Theta) \Delta p -b \Delta p_t=2k(\Theta) (p_t)^2,
\end{equation}
which is understood in $L^\infty(0, T; L^2(\Om))$.

We conclude that there exists a subsequence converging in the sense of \eqref{initial-weak-conv}, \eqref{weak_star_conv} whose limit $(p, \Theta)$ is a solution to the limiting system \eqref{limiting_system}. In addition, since we have more regularity than that provided by the well-posedness result in \cite{Nikolic_2022} especially in terms of the derivatives in time $p_t, p_{tt}, \Theta_t, \Theta_{tt}$, the uniqueness of the solution obtained in \cite{Nikolic_2022} extends to the functional setting at hand. Thus, we can confirm that the whole sequence $(p^\tau, \Theta^\tau)_{\tau \in (0,\bar{\tau}]}$ converges to the solution of the parabolic system \eqref{limiting_system} based on \cite[Proposition 10.13]{Zeidler1986}.

 %\bibliography{references}{}
%\bibliographystyle{siam}  

\end{document}